\documentclass[12pt]{amsart}

\usepackage{amsmath,amssymb,amsfonts,amscd}
\usepackage{latexsym}
\usepackage{epic,eepic,epsfig,psfrag}
\usepackage{amscd}
\usepackage{hyperref}
\usepackage{color}
\usepackage{amsbsy}
\usepackage{tensor}
\usepackage{fixmath}
\usepackage{tikz}
\usepackage{float}

\let\oldtocsection=\tocsection
 
\let\oldtocsubsection=\tocsubsection
 
\let\oldtocsubsubsection=\tocsubsubsection
 
\renewcommand{\tocsection}[2]{\hspace{0em}\oldtocsection{#1}{#2}}
\renewcommand{\tocsubsection}[2]{\hspace{2em}\oldtocsubsection{#1}{#2}}
\renewcommand{\tocsubsubsection}[2]{\hspace{4em}\oldtocsubsubsection{#1}{#2}}

\newcommand\ul{\underline}

\def\k{\kappa}

\def\g{\gamma}

\def\s{\sigma}
\def\t{\tau}
\def\i{\iota}

\def\G{\Gamma}

\def\Si{\Sigma}

\def\cB{{\mathcal B}}

\def\cE{{\mathcal E}}

\def\cM{{\mathcal M}}

\def\cO{{\mathcal O}}

\def\cX{{\mathcal X}}

\newtheorem{dfn}{Definition}[section]
\newtheorem{lem}[dfn]{Lemma}
\newtheorem{prp}[dfn]{Proposition}
\newtheorem{thm}[dfn]{Theorem}
\newtheorem{rmk}[dfn]{Remark}
\newtheorem{cor}[dfn]{Corollary}

\newtheorem{fac}[dfn]{Fact}

\newcounter{qcounter}
\newenvironment{enumilistfiber}
   { \begin{list} {\rm (\Roman{qcounter})\;}{\usecounter{qcounter}
     \setlength{\itemsep}{.5ex} \setlength{\leftmargin}{1.2ex} } }
   { \end{list} }

  \newenvironment{enumilistalph}
   { \begin{list} {\rm (\alph{qcounter})\;}{\usecounter{qcounter}
     \setlength{\itemsep}{.5ex} \setlength{\leftmargin}{1.2ex} } }
   { \end{list} }
   
\newenvironment{itemlist}
   { \begin{list} {$\bullet$}
         {  \setlength{\itemsep}{.5ex} \setlength{\leftmargin}{2.5ex} } }
   { \end{list} }

\author{Benjamin Filippenko}

\title{Polyfold regularization of constrained moduli spaces}

\makeatletter
\def\blfootnote{\xdef\@thefnmark{}\@footnotetext}
\makeatother

\begin{document}

\begin{abstract}
We introduce tame sc-Fredholm sections and slices of sc-Fredholm sections. A slice is a notion of subpolyfold that is compatible with the sc-Fredholm section and has finite locally constant codimension. We prove that the subspace of a tame polyfold that satisfies a transverse sc-smooth constraint in a finite dimensional smooth manifold is a slice of any tame sc-Fredholm section compatible with the constraint. Moreover, we prove that a sc-Fredholm section restricted to a slice is a tame sc-Fredholm section with a drop in Fredholm index given by the codimension of the slice. As a corollary, we obtain fiber products of tame sc-Fredholm sections. We describe applications to Gromov-Witten invariants, constructing the Piunikhin-Salamon-Schwarz maps for general closed symplectic manifolds, and avoiding sphere bubbles in moduli spaces of expected dimension $0$ and $1$.
\end{abstract}

\vspace*{-1cm}
\maketitle

\vspace*{-1cm}
\tableofcontents

\blfootnote{This material is based upon work supported by the National Science Foundation under Award No.\ 1708916 and No.\ 1903023.}

\vspace*{-1cm}
\section{Introduction} \label{sec:constrainedmodulispacesintroduction}

Polyfold theory, developed by Hofer-Wysocki-Zehnder \cite{MR2341834} \cite{MR2507223} \cite{MR2515707} \cite{MR2558891} \cite{MR2644764} \cite{MR3683060} \cite{HWZbook}, is an analog of classical nonlinear Fredholm theory designed to realize compact moduli spaces, e.g.\ Gromov-Witten moduli spaces and Floer trajectory spaces, as zero sets of sc-Fredholm sections of polyfold bundles. See the polyfold survey \cite{MR3576532} for an overview and a discussion of applications. The abstract polyfold machinery provides perturbations such that the perturbed sc-Fredholm section is transverse to zero, and hence the perturbed solution space has smooth structure by a polyfold implicit function theorem. Crucially, the perturbations can be chosen so that the perturbed solution space remains compact. This process, beginning with the description of the compact moduli space as the zero set of a sc-Fredholm section and ending with the smooth compact perturbed solution space, is colloquially referred to as ``polyfold regularization'' of a moduli space.

Often in symplectic topology we wish to constrain moduli spaces of pseudoholomorphic curves to consist of those curves satisfying intersection conditions with submanifolds. For example, Gromov-Witten invariants can be defined as counts of curves whose marked points evaluate to submanifolds. Any fiber product of moduli spaces over evaluation maps is another example of such a constraint. The evaluation maps are usually not transverse on the moduli space, however they extend to the ambient polyfold and here they are submersive. Using this transversality, we construct in this paper the constrained polyfold and the constrained sc-Fredholm section. This provides an abstract tool to regularize constrained moduli spaces, whenever the original moduli space is given as the zero set of a sc-Fredholm section.

We state our theorems and outline the structure of the paper in Section~\ref{subsec:results}. Then in Section~\ref{subsec:applications} we explain applications to Gromov-Witten invariants (Section~\ref{subsub:gromov}), constructing the Piunikhin-Salamon-Schwarz maps (Section~\ref{subsub:pss}) to prove the weak Arnold conjecture for general closed symplectic manifolds (see \cite{AFFW} for details), and avoiding sphere bubbles in perturbed moduli spaces of expected dimension $0$ and $1$ (Section~\ref{subsub:avoiding}).

\subsection{Results and outline} \label{subsec:results}

The main goal of this paper is to prove the $M$-polyfold and ep-groupoid (with boundary and corners) versions of the following classical Facts~\ref{fac:restrictbanach}, \ref{fac:globalimplicitbanachfredholm} from non-linear Fredholm theory over Banach manifolds.

\begin{fac} \label{fac:restrictbanach} {\bf (Restrictions of Fredholm sections to sub-Banach manifolds)}
Consider a Banach manifold $B$, a smooth Banach bundle $p : E \rightarrow B$, and a Fredholm section $s : B \rightarrow E$ with Fredholm index $ind_x(s)$ for $x \in B$. If $\tilde{B} \subset B$ is a codimension-$n$ sub-Banach manifold, then the restriction $\tilde{p} : \tilde{E} := p^{-1}(\tilde{B}) \rightarrow \tilde{B}$ is a smooth Banach bundle and the restricted section $\tilde{s} := s|_{\tilde{B}} : \tilde{B} \rightarrow \tilde{E}$ is a Fredholm section with Fredholm index satisfying $ind_x(\tilde{s}) = ind_x(s) - n$ for $x \in \tilde{s}^{-1}(0) \subset \tilde{B}$.
\end{fac}
\begin{proof}
Let $x \in \tilde{B}$ and consider an open neighborhood $U \subset B$ of $x$ and local trivialization $U \times F \rightarrow U$ of $p$. Then $\tilde{U} := U \cap \tilde{B}$ is an open neighborhood of $x$ in $\tilde{B}$ and $\tilde{U} \times F \rightarrow \tilde{U}$ is a local trivialization of $\tilde{p}$. Smooth compatibility of local trivializations of $\tilde{p}$ constructed in this way follows from smooth compatibility of the local trivializations of $p$. Hence $\tilde{p}$ is a smooth Banach bundle. Suppose $s(x) = 0$. The differential of the section $s$ at $x$ projected to the fiber is a bounded linear map $D_x s : T_xU \rightarrow F$ satisfying $\dim \ker(D_x s) - \dim \text{coker}(D_x s) = ind_x(s)$ by definition of Fredholm index. Since $\i : T_x\tilde{U} \rightarrow T_xU$ is a linear codimension-$n$ embedding and $D_x\tilde{s} = D_xs \circ \i : T_x \tilde{U} \rightarrow F$, an exercise in linear algebra shows that $ind_x(\tilde{s}) = ind_x(s) - n.$
\end{proof}

Fact~\ref{fac:globalimplicitbanachfredholm} below follows from Fact~\ref{fac:restrictbanach} together with the codimension-$n$ Banach manifold charts provided by the normal form of a $C^1$ local submersion to $\mathbb{R}^n$, for which we provide a proof (in the context of boundary and corners) in Lemma~\ref{lem:normalformbanachspacecase} for later use.

\begin{fac} \label{fac:globalimplicitbanachfredholm} {\bf (Transverse preimages are sub-Banach manifolds)}
Consider a Banach manifold $B$, a finite dimensional smooth manifold $Y$ together with a codimension-$n$ submanifold $N \subset Y$, and a smooth map $f : B \rightarrow Y$. Assume that $f$ is transverse to $N$.

Then, $\tilde{B} := f^{-1}(N)$ is a codimension-$n$ sub-Banach manifold of $B$. In particular, if $s : B \rightarrow E$ is a Fredholm section of a smooth Banach bundle $p : E \rightarrow B$, then the restriction $\tilde{p} : \tilde{E} := p^{-1}(\tilde{B}) \rightarrow \tilde{B}$ is a smooth Banach bundle and the restricted section $\tilde{s} := s|_{\tilde{B}} : \tilde{B} \rightarrow \tilde{E}$ is a Fredholm section with Fredholm index satisfying $ind_x(\tilde{s}) = ind_x(s) - n$ for $x \in \tilde{s}^{-1}(0) \subset \tilde{B}$.
\hfill$\square$
\end{fac}

The $M$-polyfold versions (with boundary and corners) of the above facts are our main theorems, Theorem~\ref{thm:slicestructure} and Theorem~\ref{thm:globalimplicit}, proved in Section~\ref{sec:globalslicing}. The generalizations of these theorems to the ep-groupoid case, which are required in applications to handle nontrivial isotropy groups, are Corollary~\ref{cor:slicestructureepgroupoid} and Corollary~\ref{cor:globalimplicitepgroupoid}, proved in Section~\ref{sec:ep-groupoids}. In Section~\ref{sec:fiberproducts}, we obtain fiber products of tame sc-Fredholm sections as Corollary~\ref{cor:ep-fiberFredholm}. See the end of this section for a further description of the ep-groupoid situation and how to use it to perform these constructions on polyfolds.

The sections preceding Sections~5-7 are concerned with the local $M$-polyfold constructions required to prove the theorems. These are the technical heart of the paper. In Section~\ref{sec:sccalc}, the main result (Lemma~\ref{lem:sclocalsubmersionnormalform}) is a normal form of a local sc-smooth submersion from a sc-Banach space to $\mathbb{R}^n$, obtained by a change of coordinates in the domain. This can be viewed as an implicit function theorem in sc-calculus in the case of a finite dimensional target. Note that implicit function theorems of this type do not hold in general in sc-calculus; see \cite{FZW}. In Section~\ref{sec:slices}, we introduce sliced sc-retracts, sliced bundle retracts, and sliced sc-Fredholm germs, and establish their properties. Most importantly, these objects induce codimension-$n$ sub-objects that are the local models for the constrained polyfolds and sc-Fredholm sections constructed in this paper; see Lemmas~\ref{lem:inducedtameretraction}, \ref{lem:inducedtamebundleretraction}, \ref{lem:inducedlocalscFredholmgerm}. We also introduce tame sc-Fredholm germs, which are a special class of sc-Fredholm germ in which the change of coordinates to basic germ form is linear. It is necessary for our constructions that all sc-Fredholm sections are locally modeled on tame sc-Fredholm germs, which holds in applications; see Section~\ref{sec:evaluationmaps} for a discussion of examples. In Section~\ref{sec:slicecoordinates}, we prove the technical Lemma~\ref{lem:slicecoordinates}, which states that the sliced objects introduced in Section~\ref{sec:slices} are obtained from tame sc-retracts, tame bundle retracts, and tame sc-Fredholm germs, via the change of coordinates in the domain obtained from a submersion to $\mathbb{R}^n$ described in Section~\ref{sec:sccalc}.

We proceed to describe the global objects used in our constructions, introduce notation, and then state our main theorems. 

Throughout, we denote $M$-polyfolds $\cB$, strong $M$-polyfold bundles $\rho : \cE \rightarrow \cB$, and sc-Fredholm sections $\s : \cB \rightarrow \cE$. The central objects developed in this paper are tame sc-Fredholm sections $\s : \cB \rightarrow \cE$ (Definition~\ref{dfn:tamescfredholmsection}) and slices $\tilde{\cB} \subset \cB$ (Definition~\ref{dfn:globalslice}) of sc-Fredholm sections. They are locally modeled on tame sc-Fredholm germs (Definition~\ref{dfn:tamescFredgerm}) and sliced sc-Fredholm germs (Definition~\ref{dfn:slicescFredholmgerm}), respectively. A slice is a new notion of a finite codimension $M$-polyfold $\tilde{\cB}$ embedded in $\cB$ that is compatible with $\s$. These notions are related by our main Theorems~\ref{thm:slicestructure}, \ref{thm:globalimplicit}. Roughly, the theorems are as follows, with precise statements given below. Given a tame sc-Fredholm section $\s : \cB \rightarrow \cE$ and a sc-smooth map $f : \cB \rightarrow Y$ to a finite dimensional manifold $Y$ that is $\s$-compatibly transverse (Definition~\ref{dfn:transverse}) to a submanifold $N \subset Y$, then $f^{-1}(N)$ is a slice of $\s$. Moreover, given a slice $\tilde{\cB} \subset \cB$ of a sc-Fredholm section $\s : \cB \rightarrow \cE$, the restriction $\s|_{\tilde{\cB}}$ is tame sc-Fredholm. In Section~\ref{sec:evaluationmaps}, we explain why the Cauchy-Riemann section $\overline{\partial}_J : \cB \rightarrow \cE$ is tame sc-Fredholm and why evaluation maps $ev : \cB \rightarrow Y$ at marked points are $\overline{\partial}_J$-compatibly transverse to every submanifold $N \subset Y$.

Before we state the theorems, we briefly recall some more polyfold notation. See Section~\ref{sec:globalslicing} for more detail. Given an $M$-polyfold $\cB$, there is a filtration ${\cB = \cB_0 \supset \cB_1 \supset \cdots}$ induced by the sc-structures in local charts. Each filtration level $\cB_m$ has its own topology which is not the subspace topology, but the inclusions are continuous and dense. The smooth points of $\cB$ are the subset $\cB_{\infty} := \cap_{m \geq 0} \cB_m$, which is also dense in $\cB$. For $x \in \cB$, the degeneracy index $d_{\cB}(x)$ is the number of distinct local boundary faces of $\cB$ intersecting at $x$.

The following result is the $M$-polyfold analog of Fact~\ref{fac:restrictbanach}. See Corollary~\ref{cor:slicestructureepgroupoid} for the generalization to ep-groupoids.

\newtheorem*{thm:slicestructure}{Theorem \ref{thm:slicestructure}}
\begin{thm:slicestructure} {\bf (Restrictions of sc-Fredholm sections to slices)}
\begin{enumilistfiber}
\item Consider a tame $M$-polyfold $\cB$ and a slice $\tilde{\cB} \subset \cB$ (Definition~\ref{dfn:globalslice}). Then, $\tilde{\cB}$ is a tame $M$-polyfold with atlas induced by the sliced charts with respect to $\tilde{\cB} \subset \cB$. For $x \in \tilde{\cB}_1$, the codimension $codim_x(\tilde{\cB} \subset \cB)$ is well-defined and locally constant in $\tilde{\cB}$, i.e.\ it equals $codim_{x'}(\tilde{\cB} \subset \cB)$ for every $x'$ in an open neighborhood of $x$ in $\tilde{\cB}$. For $x \in \tilde{\cB}_{\infty}$, the degeneracy index satisfies $d_{\tilde{\cB}}(x) = d_{\cB}(x)$.\\

\item Consider, in addition, a tame strong bundle $\rho : \cE \rightarrow \cB$. If $\tilde{\cB} \subset \cB$ is a slice of $\rho$, then the restriction $\tilde{\rho} := \rho|_{\tilde{\cE}} : \tilde{\cE} :=  \rho^{-1}(\tilde{\cB}) \rightarrow \tilde{\cB}$ is a tame strong bundle with atlas induced by the sliced bundle charts for $\rho$ with respect to $\tilde{\cB} \subset \cB$.\\

\item Consider, in addition, a sc-Fredholm section $\s : \cB \rightarrow \cE$. If $\tilde{\cB} \subset \cB$ is a slice of $\s$, then the restriction $\tilde{\s} = \s|_{\tilde{\cB}} : \tilde{\cB} \rightarrow \tilde{\cE}$ is a tame sc-Fredholm section (Definition~\ref{dfn:tamescfredholmsection}) of $\tilde{\rho}$ with tame sc-Fredholm charts induced by the sliced sc-Fredholm charts for $\s$ with respect to $\tilde{\cB} \subset \cB$. For $x \in \tilde{\cB}_{\infty}$, the index satisfies $ind_x(\tilde{\s}) = ind_x(\s) - codim_x(\tilde{\cB} \subset \cB)$. If $\s^{-1}(0)$ is compact and $\tilde{\cB}_{\infty} \subset \cB_{\infty}$ is closed, then $\tilde{\s}^{-1}(0)$ is compact.
\end{enumilistfiber}
\end{thm:slicestructure}

\begin{rmk}
There are three notions of a slice $\tilde{\cB} \subset \cB$ (Definition~\ref{dfn:globalslice}) appearing in Theorem~\ref{thm:slicestructure}: (I) a slice $\tilde{\cB} \subset \cB$ of a tame $M$-polyfold, (II) a slice $\tilde{\cB} \subset \cB$ of a tame strong bundle $\rho : \cE \rightarrow \cB$, and (III) a slice $\tilde{\cB} \subset \cB$ of a sc-Fredholm section $\s : \cB \rightarrow \cE$. Each successive notion requires further compatibility of the subset $\tilde{\cB} \subset \cB$ with the additional structure. This is in contrast to the Banach manifold situation in Fact~\ref{fac:restrictbanach} where bundles and Fredholm sections automatically restrict to any finite codimension sub-Banach manifold $\tilde{B} \subset B$.

The required compatibilities are roughly as follows. (I) There are charts on $\cB$ to $\mathbb{R}^n$-sliced sc-retracts $\cO$ (Definition~\ref{dfn:sliceretraction}) that locally identify $\tilde{\cB} \subset \cB$ with the induced tame sc-retract $\tilde{\cO} \subset \cO$ from Lemma~\ref{lem:inducedtameretraction}. (II) There are bundle charts on $\rho$ to $\mathbb{R}^n$-sliced bundle retracts $K$ (Definition~\ref{dfn:slicebundleretraction}) covering $\mathbb{R}^n$-sliced sc-retracts $\cO$. In this case, $\rho^{-1}(\tilde{\cB})$ is locally identified with the induced tame bundle retract $\tilde{K} \subset K$ from Lemma~\ref{lem:inducedtamebundleretraction}, and the restriction $\rho^{-1}(\tilde{\cB}) \rightarrow \tilde{\cB}$ is locally identified with the induced tame local bundle model $\tilde{K} \rightarrow \tilde{\cO}$. (III) There are sc-Fredholm charts for $\s$ at every $x \in \tilde{\cB}_{\infty}$ to $\mathbb{R}^n$-sliced sc-Fredholm germs $\cO \rightarrow K$ (Definition~\ref{dfn:slicescFredholmgerm}). In this case, the restriction $\s|_{\tilde{\cB}} : \tilde{\cB} \rightarrow \tilde{\cE}$ is locally identified with the induced tame sc-Fredholm germ $\tilde{\cO} \rightarrow \tilde{K}$ from Lemma~\ref{lem:inducedlocalscFredholmgerm}.

The reason for the further requirements in the $M$-polyfold setting is the non-trivial sc-retractions and sc-Fredholm fillings: compatibility of $\tilde{\cB}$ with the sc-retractions on $\cB$ does not imply compatibility of $\tilde{\cB}$ with the bundle retractions on $\cE$ or with the local fillings of $\s$.
\end{rmk}

We now state our main theorem, which is the $M$-polyfold analog of the classical Fact~\ref{fac:globalimplicitbanachfredholm}. See Corollary~\ref{cor:globalimplicitepgroupoid} for the generalization to ep-groupoids. See the following Remark~\ref{rmk:technicalitiesintro} for a discussion of the technicalities in the statement. Given an $M$-polyfold $\cB$, there is an $m$-shifted $M$-polyfold $\cB^m$ for each $m \geq 1$ which is obtained by forgetting about the filtration levels $\cB_0,\ldots,\cB_{m-1}$ of $\cB$ discussed above.

\newtheorem*{thm:globalimplicit}{Theorem \ref{thm:globalimplicit}}
\begin{thm:globalimplicit} {\bf (Transverse preimages are slices of sc-Fredholm sections)}
\begin{enumilistfiber}
\item Consider a tame $M$-polyfold $\cB$, a smooth manifold $Y$ together with a codimension-$n$ submanifold $N \subset Y$, and a sc-smooth map $$f : \cB \rightarrow Y.$$ Assume that $f$ is transverse to $N$ (Definition~\ref{dfn:transverse}).

Then, there exists an open neighborhood
$$\tilde{\cB} \subset f^{-1}(N) \cap \cB_1$$
of ${f^{-1}(N) \cap \cB_{\infty}}$ such that $\tilde{\cB}$ is a slice of $\cB^1$ with $codim_x(\tilde{\cB} \subset \cB^1) = n$ for every $x \in \tilde{\cB}_1 = \tilde{\cB} \cap \cB_2$. In particular, $\tilde{\cB}$ is a tame $M$-polyfold with degeneracy index satisfying $d_{\tilde{\cB}}(x) = d_{\cB}(x)$ for all $x \in \tilde{\cB}_{\infty}$.\\

\item Consider, in addition, a tame strong bundle $\rho : \cE \rightarrow \cB$. Then, there exists a possibly smaller neighborhood $\tilde{\cB}$ in (I) that is a slice of the bundle $\rho|_{\cE^1} : \cE^1 \rightarrow \cB^1$. In particular, the restriction
$$\tilde{\rho} := \rho|_{\tilde{\cE}} : \tilde{\cE} := (\rho|_{\cE^1})^{-1}(\tilde{\cB}) \rightarrow \tilde{\cB}$$
is a tame strong bundle.\\

\item Consider, in addition, a tame sc-Fredholm section $\s : \cB \rightarrow \cE$ (Definition~\ref{dfn:tamescfredholmsection}) of $\rho$. Assume that $f$ is $\s$-compatibly transverse to $N$ (Definition~\ref{dfn:transverse}). Then, there exists a possibly smaller neighborhood $\tilde{\cB}$ in (II) that is a slice of the tame sc-Fredholm section $\s|_{\cB^1} : \cB^1 \rightarrow \cE^1$. In particular, the restriction
$$\tilde{\s} := \s|_{\tilde{\cB}} : \tilde{\cB} \rightarrow \tilde{\cE}$$
is a tame sc-Fredholm section of $\tilde{\rho}$ with index satisfying $$ind_x(\tilde{\s}) = ind_x(\s) - n$$ for all $x \in \tilde{\cB}_{\infty}$. If $N$ is closed as a subset of $Y$ and $\s^{-1}(0)$ is compact, then $\tilde{\s}^{-1}(0)$ is compact.
\end{enumilistfiber}
\end{thm:globalimplicit}

\begin{rmk} \label{rmk:technicalitiesintro} \text{}
\begin{enumerate}
\item The notion of $\s$-compatibly transverse (Definition~\ref{dfn:transverse}) requires compatibility between the tangent map $D_xf$ at $x \in f^{-1}(N) \cap \cB_{\infty}$ with the change of coordinates on the base of the local sc-Fredholm filling of $\s$ at $x$ that brings the filling into basic germ form. See Section~\ref{sec:evaluationmaps} for an explanation why evaluation maps $f = ev$ at marked points are compatible with the Cauchy-Riemann section $\s = \overline{\partial}_J$ in applications.

\item The reason Theorem~\ref{thm:globalimplicit} holds only in some neighborhood $\tilde{\cB}$ of the smooth points of the preimage $f^{-1}(N) \cap \cB_{\infty}$ is as follows. The tame sc-retracts modeling $\tilde{\cB}$ are built from the subspace of the tangent space $T_x\cB$ at $x \in f^{-1}(N)$ that is mapped by the tangent map $D_xf : T_x\cB \rightarrow T_{f(x)}Y$ onto $T_{f(x)}N$, and this tangent space $T_x\cB$ has the structure of a sc-Banach space only at smooth points $x \in \cB_{\infty}$. So we can only hope to construct a sc-retract modeling a neighborhood of $x$ in $f^{-1}(N)$ around smooth points $x$.

\item The neighborhood $\tilde{\cB}$ is open only in the $1$-level of the preimage $f^{-1}(N) \cap \cB_1$ because, in the proof of the local submersion normal form in sc-calculus (Lemma~\ref{lem:sclocalsubmersionnormalform}), we must $1$-shift the sc-Banach space to obtain $C^1$ regularity in order to use the classical $C^1$ local submersion normal form (Lemma~\ref{lem:normalformbanachspacecase}).
\end{enumerate}
\end{rmk}

\begin{rmk} \label{rmk:conventionsformanifolds}
All manifolds $Y$ and submanifolds $N \subset Y$ in this paper are smooth, finite dimensional, and without boundary:
$$\partial N = \emptyset \,\, \text{ and } \,\, \partial Y = \emptyset.$$
This suffices for our initial intended applications to evaluation maps with target a closed symplectic manifold $Y$.

It is possible to generalize our theorems to the case where both $N$ and $Y$ are smooth finite dimensional orbifolds with boundary and corners. This generalization will be useful in applications. For example, polyfolds $\cB$ constructed for regularization of moduli spaces of pseudoholomorphic curves in symplectic topology come with an everywhere submersive sc-smooth forgetful map $\cB \rightarrow Y$ to the Deligne-Mumford space $Y$ consisting of all domains of curves in $\cB$. The Deligne-Mumford space $Y$ usually has an orbifold structure with non-trivial isotropy, and when the domains have boundary, $Y$ has boundary and corners.
\end{rmk}

In Section~\ref{sec:ep-groupoids}, we address the finite isotropy groups arising in applications to pseudoholomorphic curves by extending our theorems from $M$-polyfolds to ep-groupoids. An ep-groupoid is the ``orbifold version'' of an $M$-polyfold. The theory of sc-Fredholm sections over ep-groupoids is the generalization of sc-Fredholm sections over $M$-polyfolds that incorporates finite isotropy; see \cite{HWZbook} for a detailed treatment. We review the necessary facts in Section~\ref{sec:ep-groupoids}.

All of the ep-groupoid results in Section~\ref{sec:ep-groupoids} follow from our $M$-polyfold results plus additional bookkeeping. In fact, this is a general polyfold philosophy: the $M$-polyfold situation is where all of the analytic data (e.g.\ the sc-structures on base and bundle and the sc-Fredholm properties) is stored. So, if an application of polyfold theory works assuming that everything in sight is an $M$-polyfold, we expect that there is a suitable upgrade to ep-groupoids obtained by keeping track of the isotropy using the ep-groupoid machinery in \cite{HWZbook}. This is the case in this paper.

This polyfold philosophy goes further. When performing any polyfold construction, one thinks through the details first assuming that everything is a finite dimensional smooth manifold with boundary and corners. In particular, sc-Fredholm sections in the finite dimensional setting are the same as ordinary smooth sections. We expect that any construction \emph{that is motivated by regularization of some moduli space of pseudoholomorphic curves}, and that works assuming everything is finite dimensional, will go through for $M$-polyfolds and ep-groupoids.

To perform these constructions on polyfolds, it suffices to take representative ep-groupoids and then perform the constructions to obtain a representative ep-groupoid of a new polyfold. A polyfold \cite[Def.~16.3]{HWZbook} is a topological space $Z$ together with an equivalence class of polyfold structures \cite[Def.~16.1]{HWZbook}. A polyfold structure for $Z$ is an ep-groupoid together with a homeomorphism of its orbit space with $Z$. Polyfold structures are equivalent \cite[Def.~16.2]{HWZbook} if they are related by a generalized isomorphism \cite[Def.~10.8]{HWZbook} compatible with the homeomorphisms with $Z$. In particular, every ep-groupoid defines a canonical polyfold structure on its orbit space.

\subsection{Applications} \label{subsec:applications}

We discuss applications of our polyfold results to pseudoholomorphic curves in symplectic manifolds.

First, we describe an expected alternative interpretation of Gromov-Witten invariants, defined using polyfold theory in \cite{MR3683060} as integrals of differential forms over a perturbed moduli space, as counts of points in a $0$-dimensional constrained moduli space, where the constraints are evaluation maps at marked points that are required to evaluate to submanifolds. Then we describe the construction of the Hamiltonian Piunikhin-Salamon-Schwarz maps for general closed symplectic manifolds, which is carried out in detail in \cite{AFFW} using our theorems to construct the fiber product of Morse moduli spaces and Symplectic Field Theory polyfolds \cite{HWZsymplecticfieldtheory}\cite{PolyfoldConstructions}\cite{FHLecturesOnSFT}\cite{SFTPrimerI}\cite{SFTPrimerII}, providing a proof of the weak Arnold conjecture. Last, we describe a method for perturbing expected dimension $0$ and $1$ moduli spaces so that the perturbed moduli space does not contain any curves with a sphere bubble.

The applicable polyfold results in this paper are the ep-groupoid generalizations of the theorems presented above, because in applications there will be nontrivial isotropy groups. The ep-groupoid results are in Sections~\ref{sec:ep-groupoids},\ref{sec:fiberproducts}.

Throughout, let $(Y,\omega)$ be a closed symplectic manifold of dimension $2n$.

\subsubsection{Gromov-Witten invariants} \label{subsub:gromov} The Gromov-Witten invariants are defined in \cite{MR3683060} for general closed symplectic manifolds as integrals of differential forms over the solution set of a perturbed sc-Fredholm section. The results in this paper provide an expected alternative interpretation as counts of $0$-dimensional perturbed moduli spaces of curves satisfying intersection conditions with submanifolds. These definitions can be shown to agree by results in \cite{WolfPullbackPerturbation}.

For a given homology class $A \in H_2(Y)$ and integers $g,m \geq 0$ satisfying $2g + m \geq 3$, the Gromov-Witten invariant with respect to the fundamental class $[\overline{\cM}_{g,m}]$ of the Deligne-Mumford space $\overline{\cM}_{g,m}$ of closed genus $g$ curves with $m$ marked points is a multilinear map
$$\Psi_{A,g,m} : H^*(Y;\mathbb{R})^{\otimes m} \rightarrow \mathbb{R}.$$
This map is defined for $\alpha_1,\ldots,\alpha_m \in H^*(Y;\mathbb{R})$ with degrees satisfying $|\alpha_1| + \cdots + |\alpha_m| = 2c_1(A) + (2n-6)(1-g) + 2m$ as follows, and for other choices of $\alpha_i$ it is defined to be $0$. For an $\omega$-compatible almost complex structure $J$ on $Y$, a sc-Fredholm section
$$\overline{\partial}_J : X_{A,g,m}  \rightarrow E$$
of a polyfold bundle $E \rightarrow X_{A,g,m}$ with Fredholm index $$ind(\overline{\partial}_J) = 2c_1(A) + (2n-6)(1-g) + 2m$$
is constructed in \cite{MR3683060} such that the solution set
$$\overline{\partial}_J^{-1}(0) = \overline{\cM}_{g,m}(Y,A,J)$$
is the Gromov compactified moduli space of $J$-holomorphic curves in $Y$ of genus $g$ with $m$ marked points that represent the class $A$. Then the abstract perturbation theory in \cite{HWZbook} provides a sc$^+$-multisection $\Lambda$ of the bundle so that the perturbed solution space $S(\overline{\partial}_J,\Lambda)$ is a smooth compact oriented weighted branched orbifold of dimension $$\dim S(\overline{\partial}_J,\Lambda) = ind(\overline{\partial}_J)$$ over which we can integrate differential forms using the integration theory from \cite{MR2558891}. For $|\alpha_1| + \cdots + |\alpha_m| = ind(\overline{\partial}_J)$, the Gromov-Witten invariant is defined by
$$\Psi_{A,g,m}(\alpha_1 \otimes \cdots \otimes \alpha_m) := \int_{S(\overline{\partial}_J,\Lambda)} ev_1^*(\alpha_1) \wedge \cdots \wedge ev_m^*(\alpha_m),$$
where
$$ev_k : X_{A,g,m} \rightarrow Y$$
is evaluation at the $k$-th marked point.

We now explain how to use the results in this paper to construct the Gromov-Witten invariant as a count of points in a $0$-dimensional moduli space of curves evaluating to submanifolds of $Y$. For $k = 1,\ldots,m$ let $L_k \subset Y$ be an oriented submanifold such that $$[L_k] = PD(\alpha_k) \in H_*(Y;\mathbb{R}),$$ where $PD$ denotes the Poincar\'e dual. Then the codimension of $L_k$ in $Y$ is equal to the degree $|\alpha_k|$, and so $L_1 \times \cdots \times L_m$ is a codimension-$ind(\overline{\partial}_J)$ submanifold of $Y^m$. The total evaluation map
\begin{align*}
ev_1 \times \cdots \times ev_m : X_{A,g,m} &\rightarrow Y^m\\
u &\mapsto (ev_1(u),\ldots,ev_m(u))
\end{align*}
records the positions of all the marked points. Consider the subspace
$$X_{A,g,m}^{L_1,\ldots,L_m} := (ev_1 \times \cdots \times ev_m)^{-1}(L_1 \times \cdots \times L_m)$$
of $X_{A,g,m}$, which consists of those curves whose $k$-th marked point evaluates to $L_k$ for every $k = 1,\ldots,m$. Then
$$\overline{\cM}^{L_1,\ldots,L_m}_{g,m}(Y,A,J) = \overline{\partial}_J^{-1}(0) \cap X_{A,g,m}^{L_1,\ldots,L_m}$$
is the compactified moduli space of $J$-holomorphic curves $u$ in the moduli space $\overline{\cM}_{g,m}(Y,A,J)$ satisfying the point constraints $ev_k(u) \in L_k$ for all $k = 1,\ldots,m$.

To perturb this constrained moduli space $\overline{\cM}^{L_1,\ldots,L_m}_{g,m}(Y,A,J)$ so that it is cut out transversely, we first apply Corollary~\ref{cor:globalimplicitepgroupoid} to obtain a description of $\overline{\cM}^{L_1,\ldots,L_m}_{g,m}(Y,A,J)$ as the zero set of a sc-Fredholm section of a polyfold bundle, as follows. The hypotheses of Corollary~\ref{cor:globalimplicitepgroupoid} are satisfied because the total evaluation map $$ev_1 \times \cdots \times ev_m : X_{A,g,m} \rightarrow Y^m$$ is submersive, and moreover it is $\overline{\partial}_J$-compatibly transverse to the submanifold $L_1 \times \cdots \times L_m \subset Y^m$ as explained in Section~\ref{sec:evaluationmaps}. Hence, applying the corollary, there exists an open neighborhood $\tilde{X} \subset X_{A,g,m}^{L_1,\ldots,L_m}$ of the constrained moduli space $\overline{\cM}^{L_1,\ldots,L_m}_{g,m}(Y,A,J)$ such that the restricted section $\overline{\partial}_J|_{\tilde{X}} : \tilde{X} \rightarrow E|_{\tilde{X}}$ is sc-Fredholm and has Fredholm index $0$. Moreover, the solution space
$$\overline{\partial}_J|_{\tilde{X}}^{-1}(0) = \overline{\cM}^{L_1,\ldots,L_m}_{g,m}(Y,A,J)$$
is the constrained moduli space. The perturbation theory in \cite{HWZbook} then provides a sc$^+$-multisection $\tilde{\Lambda}$ of the restricted bundle $E|_{\tilde{X}} \rightarrow \tilde{X}$ so that the perturbed solution space $S(\overline{\partial}_J|_{\tilde{X}}^{-1}(0), \tilde{\Lambda})$ is a smooth compact oriented weighted branched orbifold of dimension $0$.

Notice that, even after perturbation by $\tilde{\Lambda}$, all curves $$u \in S(\overline{\partial}_J|_{\tilde{X}}^{-1}(0), \tilde{\Lambda})$$ are guaranteed to satisfy the constraints $$ev_k(u) \in L_k$$ for $k = 1,\ldots,m$ since $S(\overline{\partial}_J|_{\tilde{X}}^{-1}(0), \tilde{\Lambda})$ is contained in $X_{A,g,m}^{L_1,\ldots,L_m}$.

Morally, the weighted count $\# S(\overline{\partial}_J|_{\tilde{X}}^{-1}(0), \tilde{\Lambda})$ is the Gromov-Witten invariant. To prove the equality
$$\# S(\overline{\partial}_J|_{\tilde{X}}^{-1}(0), \tilde{\Lambda}) = \Psi_{A,g,m}(\alpha_1 \otimes \cdots \otimes \alpha_m),$$
one needs perturbations satisfying $\tilde{\Lambda} = \Lambda|_{\tilde{X}}$. It is possible to construct such perturbations by the result \cite[Thm.~1.7,~4.5]{WolfPullbackPerturbation}, which provides a more general construction of transverse perturbations that are compatible with respect to pullbacks satisfying suitable properties.

\subsubsection{The Piunikhin-Salamon-Schwarz morphism} \label{subsub:pss} Let $$H : S^1 \times Y \rightarrow \mathbb{R}$$ be a nondegenerate Hamiltonian and $(f,g)$ a Morse-Smale pair with Morse function $$f : Y \rightarrow \mathbb{R}$$ and Riemannian metric $g$. The original application that motivates this project is the construction of the Piunikhin-Salamon-Schwarz (PSS) morphism
$$PSS : H_*^{Morse}(Y;f,g) \rightarrow H_*^{Floer}(Y;H),$$
where $H_*^{Morse}(Y;f,g)$ is the Morse homology of $(f,g)$ and $H_*^{Floer}(Y;H)$ is the Floer homology of $H$. This was originally done under the assumption that $(Y,\omega)$ is semi-positive in \cite{MR1432464}, where it is proved that this map is an isomorphism, proving the weak Arnold conjecture; see also \cite{MR2954391} for a proof. In \cite{AFFW}, we carry out a version of this construction, joint with Katrin Wehrheim, for general closed symplectic manifolds, proving the Arnold conjecture in full generality. In this section, we describe the general method of \cite{AFFW} and how it uses the results of this paper.

The moduli spaces from which the PSS morphism is constructed are as follows. Consider a critical point $p$ of $f$, a contractible $1$-periodic orbit $\g$ of the Hamiltonian vector field associated to $H$, and a singular homology class $A \in H_2(Y)$. Let $\overline{\cM}(p,Y)$ denote the compactified moduli space of half-infinite gradient flow lines $\t : (-\infty,0] \rightarrow Y$ that limit to $p$ on their infinite end and evaluate to $ev(\t) := \t(0)$ in the unstable manifold of $p$. The compactification includes broken flow lines that start at $p$, break at finitely many other critical points, and end in a half-infinite flow line originating from the final critical point at which breaking occurs and evaluating to its unstable manifold. Then $\overline{\cM}(p,Y)$ can be given the structure of a smooth compact manifold with boundary and corners (see for example \cite{MR3084244}) equipped with a smooth evaluation map
$$ev_p : \overline{\cM}(p,Y) \rightarrow Y.$$

Fix smooth capping discs on each periodic orbit of $H$ and an $\omega$-compatible almost complex structure $J$ on $Y$. Then let $\cM(\g,A)$ denote the moduli space appearing in Symplectic Field Theory \cite{MR1826267} consisting of smooth maps $\mathbb{C} \rightarrow X$ satisfying the Cauchy-Riemann equation $\overline{\partial}_J$ near $0$ and the Floer equation near $\infty$ (with a fixed interpolation in between given by a cutoff function that turns off the Hamiltonian term in Floer's equation near $0$), and such that the map glued to the capping disc on $\g$ represents the homology class $A$. The compactified moduli space $\overline{\cM}(\g,A)$ also includes configurations with broken Floer trajectories and sphere bubble trees. There is an evaluation map
$$ev_{\g} : \overline{\cM}(\g,A) \rightarrow Y$$
given by evaluating at $0 \in \mathbb{C}$. The PSS moduli spaces are then the fiber products
\begin{equation*} \label{eq:pssfiberproductmoduli}
\overline{\cM}(p,\g,A) := \overline{\cM}(p,Y) \tensor[_{ev_p}]{\times}{_{ev_\g}} \overline{\cM}(\g,A).
\end{equation*}
See Figure~\ref{fig:pss} for a diagram of an element of $\overline{\cM}(p,\g,A)$.

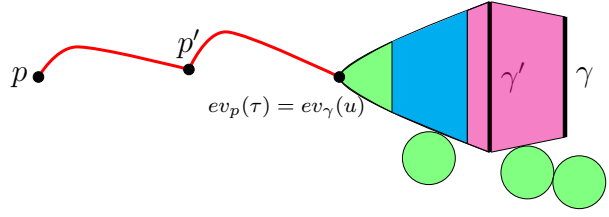
\begin{figure}
\centering
\begin{tikzpicture}
\draw [line width = .4mm, red] plot [smooth] coordinates {(-4,0) (-3.5,.4) (-2,.1)};
\filldraw [line width = .1mm, fill opacity = 1, fill = black] (-4,0) circle (2 pt);
\draw (-4,0) node[left] {$p$};
\draw [line width = .4mm, red] plot [smooth] coordinates {(-2,.1) (-1.5,.6) (0,0)};
\filldraw [line width = .1mm, fill opacity = 1, fill = black] (-2,.1) circle (2 pt);
\draw (-2,.1) node[above] {$p'$};

\draw [black] plot [smooth] coordinates {(2,1) (0,0) (2,-1)};

\filldraw [fill opacity=0.5, fill=green] plot [smooth] coordinates {(.7,.47) (0,0) (.7,-.47)};
\filldraw [line width = 0mm, fill opacity=0.5, fill=magenta] plot coordinates {(1.99,1) (1.7,.885) (1.7,-.885) (1.99,-1)};
\draw (2,0) node[right] {$\g'$};

\filldraw [line width = 0mm, fill opacity=1, fill=cyan] plot coordinates {(.7,.47) (.7,-.47) (1.7,-.885) (1.7,.885)};

\draw [line width=.6mm, black] plot coordinates {(2,1) (2,-1)};
\filldraw [line width = .1mm, fill opacity=0.5, fill=magenta] plot coordinates {(2.02,1) (3,.8) (3,-.8) (2.02,-1)};
\draw (3,0) node[right] {$\g$};

\draw [line width=.6mm, black] plot coordinates {(3,.8) (3,-.8)};

\filldraw [line width = .1mm, fill opacity = 0.5, fill = green] (2.5,-1.27) circle (10 pt);
\filldraw [line width = .1mm, fill opacity = 0.5, fill = green] (3.19,-1.4) circle (10 pt);

\filldraw [line width = .1mm, fill opacity = 0.5, fill = green] (1.19,-1.08) circle (10 pt);

\filldraw [line width = .1mm, fill opacity = 1, fill = black] (0,0) circle (2 pt);
\draw (-.7,-.4) node {\tiny $ev_p(\t) = ev_{\g}(u)$};
\end{tikzpicture}
\caption{An element $(\t,u)$ in the moduli space $\overline{\cM}(p,\g,A)$. The red lines represent $\t \in \overline{\cM}(p,Y)$ consisting of a Morse trajectory from $p$ to $p'$ and a half-infinite Morse trajectory starting at $p'$ and evaluating to $ev_p(\t)$ in the unstable manifold of $p'$. The green near $ev_{\g}(u)$ represents the neighborhood of $0 \in \mathbb{C}$ on which the map $\mathbb{C} \rightarrow Y$ satisfies the $J$-holomorphic curve equation. As the map limits to the Hamiltonian orbit $\g'$, the $J$-curve equation interpolates in the blue region to Floer's equation represented in magenta. A Floer trajectory from $\g'$ to $\g$ has broken off. The green circles represent bubbled off $J$-holomorphic spheres. The evaluations $ev_p(\t) = ev_{\g}(u)$ agree since $\overline{\cM}(p,\g,A)$ is a fiber product.}
\label{fig:pss}
\end{figure}

If all choices can be made so that $\overline{\cM}(p,\g,A)$ is smooth (and of the expected dimension), then the coefficient of $PSS(p)$ on the generator $(\g,A)$ is defined to be the count $\# \overline{\cM}(p,\g,A)$ if the expected dimension is $0$, and the coefficient is $0$ otherwise. Now, for general $Y$, the compact moduli space $\overline{\cM}(\g,A)$ will not be cut out transversely for any $J$, and hence has no reason to be smooth. Moreover, even if $\overline{\cM}(\g,A)$ is cut out transversely, there is no reason to expect that the fiber product with $\overline{\cM}(p,Y)$ is transverse.

In \cite{AFFW}, these transversality issues are overcome using the results in this paper as follows. The Symplectic Field Theory polyfolds in \cite{HWZsymplecticfieldtheory}\cite{PolyfoldConstructions}\cite{FHLecturesOnSFT}\cite{SFTPrimerI}\cite{SFTPrimerII} include polyfold bundles
$$E(\g,A) \rightarrow X(\g,A)$$
and sc-Fredholm sections
$$\s(\g,A) : X(\g,A) \rightarrow E(\g,A)$$
with solution set the SFT moduli space
$$\s(\g,A)^{-1}(0) = \overline{\cM}(\g,A).$$
Here $X(\g,A)$ is the polyfold of broken and nodal maps of the same form as those in $\overline{\cM}(\g,A)$ but not necessarily satisfying any equation, and the section $\s(\g,A)$ is the equation that maps in $\overline{\cM}(\g,A)$ are required to satisfy. Moreover there is a sc-smooth evaluation map
$$ev_{\g} : X(\g,A) \rightarrow Y$$
which evaluates at $0 \in \mathbb{C}$. This evaluation map is a submersion on the ambient space $X(\g,A)$, and it restricts to the evaluation map on the moduli space $ev_{\g} : \overline{\cM}(\g,A) \rightarrow Y$.

Applying the fiber product result Corollary~\ref{cor:ep-fiberFredholm} to the zero section of the rank-$0$ bundle over the Morse moduli space $\overline{\cM}(p,Y)$ and the sc-Fredholm section $\s(\g,A)$, we obtain an open neighborhood $X(p,\g,A)$ of the zero set of the fiber product section
\begin{align*}
\overline{\cM}(p,Y) \tensor[_{ev_p}]{\times}{_{ev_{\g}}} X(\g,A) &\rightarrow E(\g,A)\\
(\t,u) &\mapsto \s(\g,A)(u)
\end{align*}
such that the restricted section
$$\s(p,\g,A) : X(p,\g,A) \rightarrow E(p,\g,A) := E(\g,A)|_{X(p,\g,A)}$$
is sc-Fredholm with index $$ind(\s(p,\g,A)) = \dim \overline{\cM}(p,Y) + ind(\s(\g,A)) - 2n.$$ Its zero set is compact and equal to the PSS moduli space
$$\s(p,\g,A)^{-1}(0) = \overline{\cM}(p,\g,A).$$
The perturbation theory in \cite{HWZbook} then provides a sc$^+$-multisection $\Lambda$ of the polyfold bundle $E(p,\g,A) \rightarrow X(p,\g,A)$ so that the perturbed solution space $S(\s(p,\g,A),\Lambda)$ is a smooth compact weighted branched orbifold; see Figure~\ref{fig:pssperturbed} for a diagram of an element in $S(\s(p,\g,A),\Lambda)$. The weighted count of points in this perturbed moduli space provides the definition of the PSS map. That is, the coefficient of $PSS(p)$ on $(\g,A)$ is given by $\langle PSS(p), (\g,A) \rangle := \# S(\s(p,\g,A),\Lambda).$

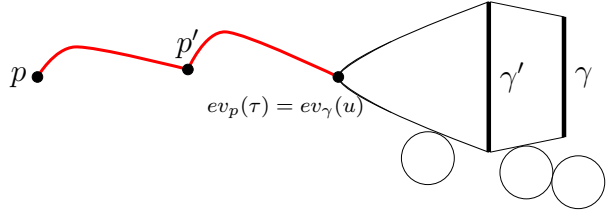
\begin{figure}[H]
\centering
\begin{tikzpicture}
\draw [line width = .4mm, red] plot [smooth] coordinates {(-4,0) (-3.5,.4) (-2,.1)};
\filldraw [line width = .1mm, fill opacity = 1, fill = black] (-4,0) circle (2 pt);
\draw (-4,0) node[left] {$p$};
\draw [line width = .4mm, red] plot [smooth] coordinates {(-2,.1) (-1.5,.6) (0,0)};
\filldraw [line width = .1mm, fill opacity = 1, fill = black] (-2,.1) circle (2 pt);
\draw (-2,.1) node[above] {$p'$};

\draw [black] plot [smooth] coordinates {(2,1) (0,0) (2,-1)};

\filldraw [fill opacity=0.5, fill=white] plot [smooth] coordinates {(.7,.47) (0,0) (.7,-.47)};
\draw (2,0) node[right] {$\g'$};

\draw [line width=.6mm, black] plot coordinates {(2,1) (2,-1)};
\filldraw [line width = .1mm, fill opacity=0.0, fill=white] plot coordinates {(2.02,1) (3,.8) (3,-.8) (2.02,-1)};
\draw (3,0) node[right] {$\g$};

\draw [line width=.6mm, black] plot coordinates {(3,.8) (3,-.8)};

\filldraw [line width = .1mm, fill opacity = 0.5, fill = white] (2.5,-1.27) circle (10 pt);
\filldraw [line width = .1mm, fill opacity = 0.5, fill = white] (3.19,-1.4) circle (10 pt);

\filldraw [line width = .1mm, fill opacity = 0.5, fill = white] (1.19,-1.08) circle (10 pt);

\filldraw [line width = .1mm, fill opacity = 1, fill = black] (0,0) circle (2 pt);
\draw (-.7,-.4) node {\tiny $ev_p(\t) = ev_{\g}(u)$};
\end{tikzpicture}
\caption{An element $(\t,u)$ in the perturbed moduli space $S(\s(p,\g,A),\Lambda) \subset \overline{\cM}(p,Y) \tensor[_{ev_p}]{\times}{_{ev_{\g}}} X(\g,A)$. The red lines represent the broken Morse trajectory $\t \in \overline{\cM}(p,Y)$ with finite end evaluating to $ev_p(\t)$. The region from $ev_{\g}(u)$ to $\g'$ is a map $\mathbb{C} \rightarrow Y$ limiting to the Hamiltonian orbit $\g'$ at $\infty$. The region in between $\g'$ and $\g$ represents a cylinder limiting to these orbits on its two ends, and the circles represent attached sphere bubbles. Together the white regions represent the element $u \in X(\g,A)$. They are colored white, in contrast to Figure~\ref{fig:pss}, to indicate that they do not necessarily satisfy any equation due to the perturbation $\Lambda$. The evaluations $ev_p(\t) = ev_{\g}(u)$ still agree since $S(\s(p,\g,A),\Lambda)$ is contained in the fiber product.}
\label{fig:pssperturbed}
\end{figure}

\subsubsection{\texorpdfstring{Avoiding sphere bubbles in expected dimension $0$ and $1$}{Avoiding sphere bubbles in expected dimension 0 and 1}} \label{subsub:avoiding} \text{}

A common mantra in symplectic topology is that ``sphere bubbling is a codimension-$2$ phenomenon'' and hence sphere bubbles do not appear in regularized moduli spaces of dimension $0$ and $1$. The notions of a sliced sc-retract (Definition~\ref{dfn:sliceretraction}) and a sliced sc-Fredholm germ (Definition~\ref{dfn:slicescFredholmgerm}) introduced in this paper provide a method for making this precise in the context of polyfold theory.

For simplicity, we consider the case of a curve with $1$ interior node, for example a curve with a single sphere bubble. This curve naturally sits inside a $\mathbb{R}^{2}$-sliced sc-retract $(\cO, \mathbb{R}^{2} \times C, \mathbb{R}^{2} \times \mathbb{E})$ which locally models a neighborhood of the curve in an ambient polyfold. The $\mathbb{R}^{2}$-sliced sc-retraction $r : U \rightarrow U$ with image the sc-retract $r(U) = \cO$ is (roughly) the splicing \cite[Def.~2.18]{HWZbook} obtained from pregluing at the node. In particular, $\cO$ is homeomorphic to the image of the pregluing map, so conceptually we identify them. Since pregluing with gluing parameter $0 \in \mathbb{R}^{2}$ preserves the node, the induced tame sc-retract defined in Lemma~\ref{lem:inducedtameretraction} by $\tilde{\cO} = \cO \cap (\{0\} \times C)$ consists of the curves in $\cO$ that have $1$ interior node. This formalizes the notion that a curve with $1$ interior node sits inside a codimension-$2$ stratum consisting of nearby nodal curves.

Moreover, the Cauchy-Riemann section $\overline{\partial}_J : \cO \rightarrow K$ of the local bundle model $K \rightarrow \cO$ is a $\mathbb{R}^{2}$-sliced sc-Fredholm germ, and hence by Lemma~\ref{lem:inducedlocalscFredholmgerm} its restriction to $\tilde{\cO}$ is sc-Fredholm with index satisfying $ind(\overline{\partial}_J|_{\tilde{\cO}}) = ind(\overline{\partial}_J) - 2$. If the original section has index satisfying $ind(\overline{\partial}_J) \leq 1$, i.e.\ the expected dimension of the moduli space is $\leq 1$, then its restriction to the nodal curves in $\tilde{\cO}$ satisfies $ind(\overline{\partial}_J|_{\tilde{\cO}}) < 0$. So, after perturbing the restricted section, the transversely cut out zero set must be empty as a smooth object with negative dimension.  Extending this perturbation over all of $\cO$, the perturbed zero set will not intersect $\tilde{\cO}$, meaning that the perturbed zero set will not contain any nodal curves.

This perturbation extension result will be part of a future work, which to be applied must include an inductive procedure that perturbs and extends starting with the highest codimension strata corresponding to solutions with the most nodes. Indeed, a curve with $k \geq 1$ nodes sits inside the intersection of $k$ distinct codimension-$2$ strata, and this intersection is a codimension-$2k$ stratum. One must perform the local perturbations and extensions coherently with respect to the intersections of these nodal strata.\\

{\bf Acknowledgments:}
I am deeply grateful to my PhD advisor, Katrin Wehrheim, for warmly inviting me into this area of research, leading me to fruitful areas to explore, keeping my vision clear as I navigate, and teaching me how to write. I'd like to thank Helmut Hofer and Zhengyi Zhou for useful comments, conversations, and feedback that resulted in improvements to this paper. Sophia Skowronski's partnership was integral to my ability to perform this research, and she will always be in my heart.

\section{\texorpdfstring{Sc-calculus: the normal form of a local sc-smooth submersion to $\mathbb{R}^n$}{Sc-calculus: the normal form of a local sc-smooth submersion to R\textasciicircum n}} \label{sec:sccalc}

The purpose of this section is to establish the normal form of a sc-smooth local submersion $f : [0,\infty)^s \times \mathbb{E} \rightarrow \mathbb{R}^n$, where $\mathbb{E}$ is a sc-Banach space; see Lemma~\ref{lem:sclocalsubmersionnormalform}. This is a sc-calculus analog of the classical local submersion normal form (Lemma~\ref{lem:normalformbanachspacecase}) in the case where $\mathbb{E}$ is an ordinary Banach space $E$. In the Banach case, the normal form follows from the inverse function theorem for $C^1$ maps between open subsets of quadrants $[0,\infty)^s \times E$. The inverse function theorem does not hold in sc-calculus \cite{FZW}, however, using relationships between classical differentiability and sc-differentiability, we leverage the classical normal form to prove the normal form in sc-calculus. The key ingredient is that the target of these submersions is finite dimensional $\mathbb{R}^n$, on which all sc-structures are trivial (i.e.\ all levels are isomorphic to the infinity level).

For completeness, we now provide a proof of the normal form in the classical Banach case. We view a quadrant $[0,\infty)^s \times E$ as a Banach space with boundary and corners. It suffices to consider neighborhoods in $[0,\infty)^s \times E$ of a point $x$ that sits in the maximally degenerate corner $x \in \{0\} \times E$.

\begin{lem} {\bf (Normal form of a $C^1$ local submersion to $\mathbb{R}^n$)} \label{lem:normalformbanachspacecase} Consider a Banach space $E$, an open subset $U \subset [0,\infty)^s \times E$ for some $s \in \mathbb{Z}_{\geq 0}$, and a $C^1$ map $f : U \rightarrow \mathbb{R}^n$. Suppose that, for some point $x \in U \cap (\{0\}^s \times E)$, the tangent map $d_x(f|_{\{0\}^s \times E}) : E \rightarrow \mathbb{R}^n$ of the restriction to the corner $f|_{U \cap (\{0\}^s \times E)}$ is surjective.

Then, for any complement $L$ of $K := \ker (d_xf)|_{\{0\}^s \times E}$ in $E$, there exist open neighborhoods $x \in \hat{U} \subset U$ and $U' \subset \mathbb{R}^n \times [0,\infty)^s \times K$ such that, writing $v \in [0,\infty)^s$ and $e \in E$, the map
\begin{align*}
g : \hat{U} &\rightarrow U'\\
(v, e) &\mapsto (f(v,e), v, pr(e))
\end{align*}
is a $C^1$-diffeomorphism, where $pr : E = K \oplus L \rightarrow K$ is the projection along $L$.
\end{lem}
\begin{proof}
Denote the restriction of $f$ to the corner by
$$\ul{f} := f|_{U \cap (\{0\}^s \times E)} : U \cap (\{0\}^s \times E) \rightarrow \mathbb{R}^n.$$
Then $d_x\ul{f} : E \rightarrow \mathbb{R}^n$ is surjective by hypothesis. Let $L$ be any complement of $K = \ker d_x\ul{f}$ in $E$. In particular, note that this means the restriction
\begin{equation} \label{eq:isooncomplement}
d_x\ul{f}|_L : L \rightarrow \mathbb{R}^n
\end{equation}
is an isomorphism. Writing $v \in [0,\infty)^s$ and $e \in E$, define the map
\begin{align*}
g : U &\rightarrow \mathbb{R}^n \times [0,\infty)^s \times K\\
(v,e) &\mapsto (f(v,e),v,pr(e)),
\end{align*}
where $pr : E = K \oplus L \rightarrow K$ is the projection along $L$. Note that $g$ is $C^1$ since $f$ is $C^1$ and $pr$ is $C^{\infty}$.

We claim that the tangent map
\begin{align*}
d_xg : \mathbb{R}^s \times E &\rightarrow \mathbb{R}^n \times \mathbb{R}^s \times K\\
(v,e) &\mapsto (d_xf(v,e),v,pr(e))
\end{align*}
is an isomorphism. To verify injectivity, suppose
$$(0,0,0) = (d_xf(v,e),v,pr(e)).$$
Then $e \in \ker(pr) = L$, and moreover $0 = d_xf(0,e) = d_x\ul{f}(e)$ means ${e \in K}$, hence $e = 0$. To verify surjectivity, let $(p,v,k) \in \mathbb{R}^n \times \mathbb{R}^s \times K$. Since the map \eqref{eq:isooncomplement} is an isomorphism, there exists $l \in L$ such that $d_x\ul{f}(l) = p - d_xf(v,k)$. It follows that $$d_xf(v,k+l) = d_xf(v,k) + d_x\ul{f}(l) = p$$ and hence $d_xg(v,k+l) = (d_xf(v,k+l),v,pr(k+l)) = (p,v,k)$. So $d_xg$ is an isomorphism, as claimed.

Since $g$ is a $C^1$ map whose tangent map $d_xg$ at $x$ is an isomorphism, the inverse function theorem\footnote{See, for example, \cite[Thm.~2.2.4]{MR1173211}.} for $C^1$ maps between quadrants of Banach spaces applies: There exists an open neighborhood $\hat{U} \subset U$ of $x$ and an open set $U' \subset \mathbb{R}^n \times [0,\infty)^s \times K$ such that the restriction $g|_{\hat{U}} : \hat{U} \rightarrow U'$ is a $C^1$-diffeomorphism, as claimed.
\end{proof}

We briefly review basics about sc-calculus on sc-Banach spaces from \cite[Sec.~1.1]{HWZbook} to prepare for Lemma~\ref{lem:sclocalsubmersionnormalform}. A \emph{sc-Banach space} \cite[Def.~1.1]{HWZbook} is a sequence of Banach spaces and continuous linear injections
$$\mathbb{E} := \big ( E_0 \hookleftarrow E_1 \hookleftarrow \cdots \big )$$
such that the map $E_{m+1} \hookrightarrow E_m$ is a compact operator for every $m \geq 0$ and the intersection 
$$E_{\infty} := \cap_{m \geq 0} E_m$$
is dense in every $E_m$. We call $E_{\infty}$ the \emph{smooth points} of $\mathbb{E}$. By a \emph{subset $S$ of a sc-Banach space $\mathbb{E}$} we mean a subset $S_0$ of $E_0$, which then induces subsets
\begin{align*}
S_m &:= S \cap E_m \subset E_m \text{ for all } m \geq 0,\\
S_{\infty} &:= S \cap E_{\infty} \subset E_{\infty}.
\end{align*}
In particular, if $U$ is open in $E_0$ then $U_m \subset E_m$ is open for all $m \geq 0$, so  we say that $U$ is open in the sc-Banach space $\mathbb{E}$.

A \emph{sc-subspace} $\mathbb{F} \subset \mathbb{E}$ \cite[Def.~1.4]{HWZbook} is a closed subspace $F \subset E_0$ such that the induced subsets $$F_m := F \cap E_m$$ define a sc-Banach space $$\mathbb{F} = (F_0 \hookleftarrow F_1 \hookleftarrow \cdots ).$$ Given sc-subspaces $\mathbb{F}, \mathbb{F}' \subset \mathbb{E}$ such that, for every $m \geq 0$, the Banach space $E_m$ splits as a direct sum $E_m = F_m \oplus F'_m$, we say that there is a \emph{sc-splitting} $\mathbb{E} = \mathbb{F} \oplus \mathbb{F}'$ and that $\mathbb{F}, \mathbb{F}'$ are \emph{sc-complements} in $\mathbb{E}$. Given a sc-subspace $\mathbb{F} \subset \mathbb{E}$, the quotient space $\mathbb{E}/\mathbb{F}$ has the structure of a sc-Banach space with $m$-level $E_m/F_m$; see \cite[Prop.~1.2]{HWZbook}. The following fact is established in the proof of \cite[Prop.~1.4]{HWZbook}. We reproduce the proof here for completeness.

\begin{lem} \label{lem:complementfinitecodimension}
Consider a sc-Banach space $\mathbb{E}$ and a sc-subspace $\mathbb{F} \subset \mathbb{E}$ such that the quotient $\mathbb{E}/\mathbb{F}$ is finite dimensional. Then, there exists a sc-complement $\mathbb{F} \oplus L = \mathbb{E}$ and moreover $L \subset E_{\infty}$.
\end{lem}
\begin{proof}
Consider the sc-continuous quotient map $p : \mathbb{E} \rightarrow \mathbb{E}/\mathbb{F}$. Since $E_{\infty} \subset E_0$ is a dense linear subspace and $p : E_0 \rightarrow E_0/F_0$ is surjective, it follows that $p(E_{\infty})$ is a dense linear subspace of the finite dimensional space $E_0/F_0$, and so we have $p(E_{\infty}) = E_0/F_0$. Hence, choosing any basis of $E_0/F_0$, there are preimages of the basis elements in $E_{\infty}$, and these preimages span a subspace $L \subset E_{\infty}$. We claim that $L$ is sc-complementary to $\mathbb{F}$ in $\mathbb{E}$. Indeed, for every $m \geq 0$ the subspace $E_m \subset E_0$ is dense and so $p_m : E_m \rightarrow E_0/F_0$ is surjective. Moreover, we have $F_m = \ker(p_m)$ and the restriction $p_m : L \rightarrow E_0/F_0$ is an isomorphism, so $F_m \oplus L = E_m$ holds.
\end{proof}

For every $l \geq 0$ there is a \emph{$l$-shifted} sc-Banach space defined by
$$\mathbb{E}^l := (E_l \hookleftarrow E_{l+1} \hookleftarrow \cdots),$$
that is,
$$(\mathbb{E}^l)_m = E_{m+l} \text{ for all } m \geq 0.$$
Conceptually, we are forgetting about finitely many levels. Note that $l$-shifting does not change the $\infty$-level. Analogously, for any subset ${S \subset \mathbb{E}}$, we define $S^l \subset \mathbb{E}^l$ by $(S^l)_m := S_{m+l}$. In particular, if $U \subset \mathbb{E}$ is open, then $U^l \subset \mathbb{E}^l$ is open for all $l \geq 0$.

A Cartesian product $\mathbb{E} \times \mathbb{F}$ of sc-Banach spaces has a natural sc-structure with $m$-level given by $(\mathbb{E} \times \mathbb{F})_m := E_m \times F_m$ equipped with any standard Banach norm on a Cartesian product. In this paper, we use the convention that all norms on Cartesian products are the sum norm $||(\cdot,\cdot)||_{E_m \times F_m} = ||\cdot||_{E_m} + ||\cdot||_{F_m}$, which is equivalent to all standard choices.

The finite dimensional space $\mathbb{E} = \mathbb{R}^n$ has a canonical sc-structure given by $E_m = \mathbb{R}^n$ equipped with the standard Euclidean norm for every $m \geq 0$, and where every inclusion $E_{m+1} \rightarrow E_m$ is the identity map.

The \emph{tangent space} \cite[Def.~1.8]{HWZbook} of a sc-Banach space $\mathbb{E}$ is the sc-Banach space
$$T\mathbb{E} := \mathbb{E}^1 \times \mathbb{E},$$
with sc-structure given by $(T\mathbb{E})_m = E_{m+1} \times E_m$ for $m \geq 0$. Given an open subset $U \subset [0,\infty)^s \times \mathbb{E}$ for some $s \geq 0$, its tangent space is
$$TU := U^1 \times (\mathbb{R}^s \times \mathbb{E}).$$

Consider sc-Banach spaces $\mathbb{E}, \mathbb{F}$, and open subsets $U \subset [0,\infty)^s \times \mathbb{E}$ and $V \subset [0,\infty)^{s'} \times \mathbb{F}$. Then a map
$$f : U \rightarrow V$$
is called \emph{sc$^0$} or \emph{sc-continuous} \cite[Def.~1.7]{HWZbook} if, for all $m \geq 0$, we have $f(U_m) \subset V_m$ and the map $f : U_m \rightarrow V_m$ is continuous. A $sc^0$ map $f : U \rightarrow V$ is called \emph{sc$^1$ with tangent map} \cite[Def.~ 1.9]{HWZbook}
\begin{equation} \label{eq:tangentmapdef}
Tf : TU \rightarrow TV
\end{equation}
defined by
\begin{align*}
Tf : U^1 \times (\mathbb{R}^s \times \mathbb{E}) &\rightarrow V^1 \times (\mathbb{R}^{s'} \times \mathbb{F})\\
(x,\xi) &\mapsto (f(x),D_xf(\xi))
\end{align*}
if, for every $x \in U^1$, there exists a bounded linear operator 
$$D_xf : \mathbb{R}^s \times E_0 \rightarrow \mathbb{R}^{s'} \times F_0$$
such that, for $\xi \in E_1$ satisfying $x+\xi \in U_1$, 
$$\lim_{|\xi|_1 \rightarrow 0} \frac{|f(x+\xi)-f(x)-D_xf(\xi)|_0}{|\xi|_1} = 0,$$
holds, and moreover such that $Tf$ is $sc^0$. Iterating the definition of sc$^1$ yields the notions of sc$^k$ for $k \geq 0$ and sc-smooth (denoted sc$^{\infty}$); see the discussion after \cite[Def.~1.9]{HWZbook}.

An important note is that, for a sc$^1$ map $f : U \rightarrow V$ and $x \in U^1$, the bounded linear operator $D_xf : \mathbb{R}^s \times E_0 \rightarrow \mathbb{R}^{s'} \times F_0$ is not necessarily sc-continuous when considered as a map between sc-Banach spaces $D_xf : \mathbb{R}^s \times \mathbb{E} \rightarrow \mathbb{R}^{s'} \times \mathbb{F}$; that is, continuity on levels higher than $0$ can fail. However, if $x \in U_{\infty}$ is a smooth point, then by \cite[Prop.~1.5]{HWZbook} the map $D_xf$ is indeed a \emph{sc-operator} \cite[Def.~1.2]{HWZbook}, i.e.\ a sc-continuous linear map. For this reason, we consider only smooth points $x$ in the following lemma so that the kernel of $D_xf$ is a sc-Banach space as the kernel of a sc-operator.

\begin{lem} {\bf (Normal form of a sc-smooth local submersion to $\mathbb{R}^n$)} \label{lem:sclocalsubmersionnormalform} Consider a sc-Banach space $\mathbb{E}$, an open subset $U \subset [0,\infty)^s \times \mathbb{E}$ for some $s \geq 0$, and a sc-smooth map $f : U \rightarrow \mathbb{R}^n$. Suppose that, for some smooth point $x \in U_{\infty} \cap (\{0\} \times \mathbb{E})$, the tangent map $(D_xf)|_{\{0\} \times \mathbb{E}} : \mathbb{E} \rightarrow \mathbb{R}^n$ of the restriction to the corner $f|_{U \cap (\{0\} \times \mathbb{E})}$ is surjective.

Then, for any sc-complement\footnote{A sc-complement $L$ of $\mathbb{K}$ exists by Lemma~\ref{lem:complementfinitecodimension}, since the surjection $(D_xf)|_{\{0\} \times \mathbb{E}} : \mathbb{E} \rightarrow \mathbb{R}^n$ induces an isomorphism $\mathbb{E}/\mathbb{K} \cong \mathbb{R}^n$.} $L$ of $\mathbb{K} := \ker (D_xf)|_{\{0\} \times \mathbb{E}}$ in $\mathbb{E}$, there exist open neighborhoods $x \in \hat{U} \subset U^1$ and $U' \subset \mathbb{R}^n \times [0,\infty)^s \times \mathbb{K}^1$ such that, writing $v \in [0,\infty)^s$ and $e \in \mathbb{E}^1$, the map
\begin{align*}
g : \hat{U} &\rightarrow U'\\
(v,e) &\mapsto (f(v,e),v,pr(e))
\end{align*}
is a sc-diffeomorphism, where $pr : \mathbb{E} = \mathbb{K} \oplus L \rightarrow \mathbb{K}$ is the projection along $L$. Moreover, for all $m \geq 0$, the map $g|_{\hat{U}_m} : \hat{U}_m \rightarrow U'_m$ is a $C^{m+1}$-diffeomorphism.

In particular, the following statements hold:
\begin{itemlist}
\item The composition $f \circ g^{-1} : U' \rightarrow \mathbb{R}^n$ is projection onto the $\mathbb{R}^n$-coordinate.

\item $g$ preserves the $[0,\infty)^s$-coordinate.
\end{itemlist}
\end{lem}
\begin{proof}
We claim that the Banach space $E_1$ (the $1$-level of the sc-Banach space $\mathbb{E}$), the open subset $U_1 \subset [0,\infty)^s \times E_1$, and the $C^1$-map
$$f_1 := f|_{U_1} : U_1 \rightarrow \mathbb{R}^n$$
satisfy the hypotheses of Lemma~\ref{lem:normalformbanachspacecase} (the local submersion normal form in the classical Banach case) at the given point $x$. First of all, the map $f_1$ is indeed $C^1$ by \cite[Prop.~1.7]{HWZbook}. By \cite[Prop.~1.5]{HWZbook}, we have $d_xf_1 = (D_xf)|_{\mathbb{R}^s \times E_1}.$

We now deduce the surjectivity of the tangent map
\begin{equation} \label{eq:theclassicaltangent}
(d_xf_1)|_{\{0\} \times E_1} = (D_xf)|_{\{0\} \times E_1} : E_1 \rightarrow \mathbb{R}^n.
\end{equation}
By hypothesis, the map $(D_xf)|_{\{0\} \times E_0}$ is surjective. Then since $E_1 \subset E_0$ is dense, it follows that $(D_xf)|_{\{0\} \times E_0}(E_1)$ is a dense linear subspace of $\mathbb{R}^n$ and hence is equal to $\mathbb{R}^n$, proving the claimed surjectivity of the map \eqref{eq:theclassicaltangent}.

Let $L \subset E_{\infty}$ be any sc-complement of $\mathbb{K} = \ker (D_xf)|_{\{0\} \times \mathbb{E}}$ in $\mathbb{E}$, which exists by Lemma~\ref{lem:complementfinitecodimension} since the surjection $(D_xf)|_{\{0\} \times \mathbb{E}} : \mathbb{E} \rightarrow \mathbb{R}^n$ induces an isomorphism $\mathbb{E}/\mathbb{K} \cong \mathbb{R}^n$. In particular, on the $1$-level, we have $K_1 \oplus L = E_1$. Notice that $K_1$ is the kernel of the map \eqref{eq:theclassicaltangent}.

We have shown that the map $f_1$ satisfies the hypotheses of the classical local submersion normal form (Lemma~\ref{lem:normalformbanachspacecase}), yielding an open neighbhorhood $\hat{U} \subset U_1$ of $x$ and an open subset $U' \subset \mathbb{R}^n \times [0,\infty)^s \times K_1$ such that, writing $v \in [0,\infty)^s$ and $e \in E_1$, the map
\begin{align*}
g : \hat{U} &\rightarrow U'\\
(v,e) &\mapsto (f_1(v,e),v,pr(e))
\end{align*}
is a $C^1$-diffeomorphism, where $pr : E_1 = K_1 \oplus L \rightarrow K_1$ is the projection along $L$.

We may view $\hat{U}$ and $U'$ as open neighborhoods in the sc-calculus sense, i.e.
$$\hat{U} \subset [0,\infty)^s \times \mathbb{E}^1 \, \text{ and } \, U' \subset \mathbb{R}^n \times [0,\infty)^s \times \mathbb{K}^1.$$
We claim that $g$ is a sc-diffeomorphism. First of all, it is sc-smooth since $f_1$ is sc-smooth by hypothesis and since the projection $$pr : \mathbb{E}^1 = \mathbb{K}^1 \oplus L \rightarrow \mathbb{K}^1$$ is sc-smooth.

To show that $g^{-1}$ is sc-smooth, we show that it satisfies the conditions of \cite[Prop.~1.8]{HWZbook}. Let $m,l \geq 0$. We must show that $g^{-1}$ induces a map $g^{-1}|_{U'_{m+l}} : U'_{m+l} \rightarrow \hat{U}_{m}$ that is $C^{l+1}$.

It suffices to show that, for all $m \geq 0$, $g^{-1}$ restricts to a $C^{m+1}$-map $g^{-1}|_{U'_m} : U'_m \rightarrow \hat{U}_m$, because then given $m,l \geq 0$ the composition $U'_{m+l} \xrightarrow{g^{-1}} \hat{U}_{m+l} \hookrightarrow \hat{U}_{m}$ is $C^{m+l+1}$ since the inclusion $\hat{U}_{m+l} \hookrightarrow \hat{U}_{m}$ is continuous and linear hence $C^{\infty}$. So, to complete the proof of the lemma, it suffices to show that
$$g : \hat{U}_m \rightarrow U'_m$$
is a $C^{m+1}$-diffeomorphism for all $m \geq 0$.

By \cite[Prop.~ 1.7]{HWZbook}, the restriction
$$f|_{\hat{U}_{m}} : \hat{U}_{m} = \hat{U} \cap (\mathbb{R}^s \times E_{m+1}) \rightarrow \mathbb{R}^n$$
is $C^{m+1}$. It follows that the restriction
$$g|_{\hat{U}_{m}} : \hat{U}_{m} \rightarrow U'_{m}$$
is $C^{m+1}$.

To see that $g|_{\hat{U}_{m}} : \hat{U}_m \rightarrow U'_m$ is a bijection, note first that injectivity holds since it is a restriction of the bijection $g$. To see surjectivity, note first that, since $g$ is surjective onto all of $U'$, it suffices to show that $g(v,e) \in U'_{m} \implies (v,e) \in \hat{U}_{m}.$ Note that $U'_{m} \subset \mathbb{R}^n \times [0,\infty)^s \times K_{m+1}$. So, from the definition of $g$, we have the implication $$g(v,e) \in U'_{m} \implies pr(e) \in K_{m+1} \subset E_{m+1}.$$ Since $e - pr(e) \in L \subset E_{\infty},$ we conclude that $e \in E_{m+1}.$ Hence indeed $(v,e) \in \hat{U}_{m} = \hat{U} \cap (\mathbb{R}^s \times E_{m+1})$ holds, proving surjectivity of $g|_{\hat{U}_{m}}$ onto $U'_{m}$. The same reasoning shows that the classical tangent map
$$d(g|_{\hat{U}_{m}}) : \hat{U}_{m} \times (\mathbb{R}^s \times E_{m+1}) \rightarrow U'_{m} \times (\mathbb{R}^n \times \mathbb{R}^s \times K_{m+1})$$
is bijective.

The inverse
$$(g^{-1})|_{U'_{m}} = (g|_{\hat{U}_{m}})^{-1} : U'_{m} \rightarrow \hat{U}_{m}$$
is $C^{m+1}$ because it is the inverse of a $C^{m+1}$ map with invertible derivative. We have shown that $g|_{\hat{U}_{m}} : \hat{U}_m \rightarrow U'_m$ is a $C^{m+1}$-diffeomorphism, completing the proof of the lemma.
\end{proof}

We briefly review general partial quadrants, which up to a linear change of coordinates are the same as the standard quadrants \eqref{eq:partialquadrantstandardform}, i.e.\ of the form $[0,\infty)^s \times \mathbb{E}$ for some sc-Banach space $\mathbb{E}$. This level of generality makes constructions more convenient and is equivalent to working with standard partial quadrants only.

A \emph{partial quadrant} \cite[Def.~1.6]{HWZbook} in a sc-Banach space $\mathbb{E}$ is a closed convex subset $C \subset \mathbb{E}$ such that there exists another sc-Banach space $\mathbb{E}'$ and a linear sc-isomorphism
\begin{equation} \label{eq:linearisotostandardquadrantdef}
\Psi : \mathbb{E} \rightarrow \mathbb{R}^s \times \mathbb{E}' \text{ satisfying }
\Psi(C) = [0,\infty)^s \times \mathbb{E}' \text{ for some } s \geq 0.
\end{equation}
That is, all partial quadrants come from applying a linear change of coordinates to a partial quadrant in the \emph{standard form}:
\begin{equation}\label{eq:partialquadrantstandardform}
C = [0,\infty)^s \times \mathbb{E}' \subset \mathbb{R}^s \times \mathbb{E}' = \mathbb{E}.
\end{equation}
The \emph{degeneracy index} \cite[Def.~1.10]{HWZbook}  $d_{C} : C \rightarrow \mathbb{N}_0$ is defined for $x \in C$ by
$$d_{C}(x) := \# \{ i \in \{1,\ldots,s\} \, | \, \text{the } i\text{-th coordinate of } \Psi(x) \text{ is } 0 \},$$
which is independent of the choice of $\Psi$ by \cite[Lem.~ 1.1]{HWZbook}. Conceptually, the degeneracy index of a point in a partial quadrant is the local notion of boundary and corner index in an $M$-polyfold.

Later, we need the following properties of the degeneracy index. Let $C$ be a partial quadrant of a sc-Banach space $\mathbb{E}$ and $n \geq 0$. Then $\mathbb{R}^n \times C$ is a partial quadrant of $\mathbb{R}^n \times \mathbb{E}$ and
\begin{equation} \label{eq:degslice}
d_{\mathbb{R}^n \times C}(p,x) = d_C(x) \text{ for all } (p,x) \in \mathbb{R}^n \times C.
\end{equation}
Let $C_i$ be a partial quadrant of $\mathbb{E}_i$ for $i = 1,2.$ Then $C_1 \times C_2$ is a partial quadrant of $\mathbb{E}_1 \times \mathbb{E}_2$ and
\begin{equation} \label{eq:degcartesian}
d_{C_1 \times C_2}(x_1,x_2) = d_{C_1}(x_1) + d_{C_2}(x_2) \text{ for all } (x_1,x_2) \in C_1 \times C_2.
\end{equation}

We recall from \cite[Def.~2.16]{HWZbook} the following linear sc-subspace $\mathbb{E}_x \subset \mathbb{E}$ associated to a point $x$ in a partial quadrant $C \subset \mathbb{E}$. First assume that $C$ is in the standard form \eqref{eq:partialquadrantstandardform} and write $x = (x_1,\ldots,x_s,e_x) \in C.$ Then, the sc-subspace
\begin{equation} \label{eq:Exstandard}
\mathbb{E}_x := \{ (v_1,\ldots,v_s,e) \in \mathbb{R}^s \times \mathbb{E}' \,\, | \,\, v_i = 0 \text{ if } x_i = 0\} \subset \mathbb{E}
\end{equation}
conceptually is the tangent space of the intersection of all of the faces of $C$ that contain $x$. For a general partial quadrant $C \subset \mathbb{E}$ and $x \in C$, the subspace $\mathbb{E}_x \subset \mathbb{E}$ is given by
\begin{equation} \label{eq:Ex}
\mathbb{E}_x := \Psi^{-1}((\mathbb{R}^s \times \mathbb{E}')_{\Psi(x)}),
\end{equation}
where $\Psi$ is any linear isomorphism of the form \eqref{eq:linearisotostandardquadrantdef}.

\section{Slices: the local picture} \label{sec:slices}

\subsection{Sliced sc-retracts} \label{subsec:sliceretracts}
In this section, we introduce the new notion of $\mathbb{R}^n$-sliced sc-retracts (Definition~\ref{dfn:sliceretraction}), which we use later as the local models in our definition of a slice $\tilde{\cB} \subset \cB$ (Definition~\ref{dfn:globalslice}) of a tame $M$-polyfold $\cB$ (Definition~\ref{dfn:tameMpolyfold}). We prove in Lemma~\ref{lem:inducedtameretraction} that a $\mathbb{R}^n$-sliced sc-retract $\cO$ induces a tame sc-retract $\tilde{\cO} \subset \cO$ which has codimension-$n$ tangent spaces $T_x\tilde{\cO} \subset T_x \cO$ at every $x \in \tilde{\cO}_1$. The global definition of a slice $\tilde{\cB} \subset \cB$ is then a subspace such that around every point $x \in \tilde{\cB}$ there is an $M$-polyfold chart to a $\mathbb{R}^{n_x}$-sliced sc-retract $\cO$ that locally identifies $\tilde{\cB}$ with the induced tame sc-retract $\tilde{\cO}$.

We first recall the local structure of tame $M$-polyfolds: tame sc-retracts (Definition~\ref{dfn:tameretraction}). Consider a relatively open subset $U$ of a partial quadrant $C$ in a sc-Banach space $\mathbb{E}$. A sc-smooth map $$r : U \rightarrow U$$ satisfying $$r \circ r = r$$ is called a \emph{sc-smooth retraction (or sc-retraction)} \cite[Def.~2.1]{HWZbook} on $U$, and the image $\cO := r(U)$ of such a map is called a \emph{sc-smooth retract (or sc-retract)}. The triple $(\cO, C, \mathbb{E})$ is also called a \emph{sc-retract} \cite[Def.~2.2]{HWZbook}.

We note that the notion of a smooth retract makes sense in the classical Banach space setting, i.e.\ given an ordinary Banach space $E$, we can define a smooth retract $\mathfrak{O}$ to be any image $\mathfrak{O} = r(U)$ of a smooth map $r : U \rightarrow U$ that satisfies $r \circ r = r$, where $U \subset [0,\infty)^s \times E$ is open. However, modeling spaces on these smooth retracts reproduces the definition of a Banach manifold because, by \cite[Prop.~2.1]{HWZbook}, a smooth retract $\mathfrak{O}$ is a $C^{\infty}$-sub-Banach manifold of $E$. The sc-retracts can have much more complicated structure, including locally varying dimension. This is a key difference between classical differentiability and sc-differentiability which allows $M$-polyfolds to have local dimension jumps and other non-manifold-like structure. Polyfolds arising in applications have these local dimension jumps near broken and nodal curves.

A map $\varphi : \cO \rightarrow \cO'$ between sc-retracts $(\cO,C,\mathbb{E})$ and $(\cO',C',\mathbb{E}')$ is called \emph{sc-smooth} \cite[Def.~2.4]{HWZbook} if the composition $\varphi \circ r : U \rightarrow \cO' \subset \mathbb{E}'$ is sc-smooth as a map $U \rightarrow \mathbb{E}'$, where $U \subset C$ is open and $r : U \rightarrow U$ is any sc-retraction onto $r(U) = \cO$. This definition is independent of the choice of open set $U$ and sc-retraction $r$ by \cite[Prop.~2.3]{HWZbook}. The chain rule holds for sc-smooth maps between sc-retracts; see \cite[Thm.~2.1]{HWZbook}.

The \emph{tangent space \cite[Def.~2.3]{HWZbook} of a sc-retract} $(\cO,C,\mathbb{E})$ is the image
\begin{equation} \label{eq:tangentspacedef}
T\cO := Tr(TU),
\end{equation}
where $r : U \rightarrow U$ is any sc-retraction on some open subset $U \subset C$ with image $r(U) = \cO$ and $Tr : TU \rightarrow TU$ is the tangent map (see \eqref{eq:tangentmapdef}) of $r$. The tangent space $T\cO$ is well-defined, i.e.\ independent of $U$ and $r$, by \cite[Prop.~2.2]{HWZbook}. The tangent space at $x \in \cO_1$ is defined to be
\begin{equation} \label{eq:tangentspaceatpoint}
T_x\cO := D_xr(T_xU).
\end{equation}
For a smooth point $x \in \cO_{\infty}$, the tangent space $T_x\cO$ is a sc-Banach space; see \cite[Def.~2.10]{HWZbook}. The \emph{reduced tangent space} \cite[Def.~2.15]{HWZbook} is the subspace of $T_x\cO$ defined by
\begin{equation} \label{eq:reducedtangentspacedef}
T_x^R\cO := T_x\cO \cap \mathbb{E}_x,
\end{equation}
where $\mathbb{E}_x \subset \mathbb{E}$ is the subspace from \eqref{eq:Ex}. Conceptually, $T_x^R\cO$ consists of those tangent vectors that point in directions that preserve the degeneracy index, i.e.\ they lie along the intersection of all of the local faces that contain $x$. Note that in \cite[Def.~2.15]{HWZbook} the reduced tangent space is only defined at smooth points $x \in \cO_{\infty}$. This is because $T_x^R\cO$ can be proven to be invariant under sc-diffeomorphisms $\varphi : \cO \rightarrow \cO'$, i.e.\ $D_x\varphi(T_x^R\cO) = T_{\varphi(x)}^R\cO'$ holds, only for smooth points $x$; see \cite[Prop.~2.8]{HWZbook}. This invariance proves that the reduced tangent space at a smooth point in an $M$-polyfold is well-defined, i.e.\ independent of the chart. The invariance is proven using the characterization \cite[Lem.~2.4]{HWZbook} of the reduced tangent space $T_x^R\cO$ at smooth points $x \in \cO_{\infty}$ as the closure of the space of derivatives of sc-smooth paths $\g : (-\epsilon,\epsilon) \rightarrow \cO$ satisfying $\g(0) = x$. This only works for smooth points $x \in \cO_{\infty}$ since the image of any sc-continuous map $(-\epsilon, \epsilon) \rightarrow \cO$ is contained in $\cO_{\infty}$ because $(-\epsilon,\epsilon) \subset \mathbb{R}$ has the trivial sc-structure where all levels are equal.

As in \cite{HWZbook}, we must require sc-retractions to be well-behaved with respect to the boundary faces of the partial quadrant $C$ in the following way in order for the full polyfold machinery to work as required in applications.

\begin{dfn} \label{dfn:tameretraction} \cite[Def.~2.17]{HWZbook}
Consider an open subset $U$ of a partial quadrant $C$ of a sc-Banach space $\mathbb{E}$. A sc-retraction $r : U \rightarrow U$ is called a {\bf tame sc-retraction} if it satisfies the following conditions:

(1) $d_C(r(x)) = d_C(x)$ for all $x \in U$.

(2) At every smooth point $x \in \cO_{\infty} = \cO \cap E_{\infty}$, there exists a sc-subspace $A \subset \mathbb{E}$ such that $\mathbb{E} = T_x\cO \oplus A$ and $A \subset \mathbb{E}_x$ (see \eqref{eq:Ex} for $\mathbb{E}_x$).

If so, then the sc-retract $\cO = r(U)$ is called a {\bf tame sc-retract} (and so is the triple $(\cO, C, \mathbb{E})$).
\end{dfn}

We introduce the following new notions of $\mathbb{R}^n$-sliced sc-retractions and $\mathbb{R}^n$-sliced sc-retracts.

\begin{dfn} \label{dfn:sliceretraction}
Consider a partial quadrant $C$ of a sc-Banach space $\mathbb{E}$ and an open subset $U \subset \mathbb{R}^n \times C$ for some $n \geq 0$. A tame sc-retraction $r : U \rightarrow U$ is called a {\bf $\mathbold{\mathbb{R}^n}$-sliced sc-retraction} if it satisfies
\begin{equation} \label{eq:sliceretraction}
\pi_{\mathbb{R}^n}  \circ r = \pi_{\mathbb{R}^n} \text{ on } U,
\end{equation}
i.e.\ $r$ preserves the $\mathbb{R}^n$-coordinate.

If so, then the tame sc-retract $\cO = r(U)$ and the triple $$(\cO,\mathbb{R}^n \times C, \mathbb{R}^n \times \mathbb{E})$$ is called a {\bf $\mathbold{\mathbb{R}^n}$-sliced sc-retract}.
\end{dfn}

In the following lemma, we show that for any $\mathbb{R}^n$-sliced sc-retract $\cO$ in $\mathbb{R}^n \times C$, the set $\tilde{\cO} := \cO \cap (\{0\} \times C)$ is a tame sc-retract. Later, we use the inclusion $\tilde{\cO} \subset \cO$ to define the local models for a slice $\tilde{\cB} \subset \cB$ (Definition~\ref{dfn:globalslice}), which is our new notion of an $M$-polyfold $\tilde{\cB}$ embedded with finite codimension in an ambient $M$-polyfold $\cB$.

\begin{lem} \label{lem:inducedtameretraction}
Consider a partial quadrant $C$ of a sc-Banach space $\mathbb{E}$ and a $\mathbb{R}^n$-sliced sc-retract $(\cO, \mathbb{R}^n \times C, \mathbb{R}^n \times \mathbb{E})$.

Then, for any open subset $U \subset \mathbb{R}^n \times C$ and $\mathbb{R}^n$-sliced sc-retraction $r : U \rightarrow U$ such that $r(U) = \cO$, the set $\tilde{U} := U \cap (\{0\} \times C)$ is open in $C$ and the restriction
$$\tilde{r} := r|_{\tilde{U}} : \tilde{U} \rightarrow \tilde{U}$$
is a tame sc-retraction onto $\tilde{\cO} := \tilde{r}(\tilde{U}).$  We call $\tilde{r}$ the {\bf tame sc-retraction induced by $r$}.

Moreover,
\begin{equation} \label{eq:inducedscretract}
\tilde{\cO} = \cO \cap (\{0\} \times C)
\end{equation}
holds, so in particular $\tilde{\cO}$ does not depend\footnote{The independence of $\tilde{\cO}$ on the choice of $U$ and $r$ frees us from keeping track of these choices when performing future constructions. That is, later, we are free to assume that $\tilde{\cO}$ is induced by any such $U$ and $r$.} on the choices of $U$ and $r$. We may view $\tilde{\cO}$ as a subset of $C$, and we call $(\tilde{\cO},C,\mathbb{E})$ the {\bf tame sc-retract induced by the $\mathbb{R}^n$-sliced sc-retract} $(\cO, \mathbb{R}^n \times C, \mathbb{R}^n \times \mathbb{E})$.

At every $x \in \tilde{\cO}_1$, the inclusion $\tilde{\cO} \subset \cO$ induces an inclusion of tangent spaces $T_x \tilde{\cO} \subset T_x \cO$ satisfying
\begin{align}
T_x \tilde{\cO} &= T_x \cO \cap (\{0\} \times \mathbb{E}) \label{eq:inducedtangentspacerelationship},\\
T_x^R\tilde{\cO} &= T_x^R \cO \cap (\{0\} \times \mathbb{E}) \label{eq:inducedreducedtangentspacerelationship},
\end{align}
and
\begin{equation} \label{eq:codimentangentspacesslicedretract}
T_x \cO / T_x\tilde{\cO} \cong \mathbb{R}^n.
\end{equation}
We say that $\tilde{\cO}$ is {\bf codimension-$\mathbold{n}$} in $\cO$.

If $x \in \tilde{\cO}_{\infty}$ is a smooth point, then the inclusion $T_x\tilde{\cO} \hookrightarrow T_x\cO$ induces a linear isomorphism
\begin{equation} \label{eq:inducedreducatedtangentspacequotientrelationship}
T_x\tilde{\cO}/T^R_x\tilde{\cO} \cong T_x\cO/T_x^R\cO.
\end{equation}
\end{lem}
\begin{proof}
The defining property \eqref{eq:sliceretraction} of the $\mathbb{R}^n$-sliced sc-retraction $r$ implies $r(\tilde{U}) \subset \tilde{U}$, so indeed the map $\tilde{r} := r|_{\tilde{U}} : \tilde{U} \rightarrow \tilde{U}$ takes values in $\tilde{U}$. Moreover, $\tilde{r}$ inherits sc-smoothness and the retraction property $\tilde{r} \circ \tilde{r} = \tilde{r}$ from the corresponding properties of $r$. So $\tilde{r}$ is a sc-retraction onto the sc-retract $\tilde{\cO}$. We prove the other statements in the lemma before showing that $\tilde{r}$ is tame.

We now verify that \eqref{eq:inducedscretract} holds. The forwards inclusion is immediate from the definitions of the sets involved. To prove the reverse inclusion, let $x \in \cO \cap (\{0\} \times C).$ Then since $\cO \subset U$ we have $x \in \tilde{U}$ and so $\tilde{r}(x) \in \tilde{\cO}$. We claim that $x = \tilde{r}(x)$, proving \eqref{eq:inducedscretract}. Indeed, since $x \in \cO$ and $r$ is a retraction with image $\cO$, it follows that $x = r(x) = \tilde{r}(x).$

We now verify \eqref{eq:inducedtangentspacerelationship} and \eqref{eq:codimentangentspacesslicedretract}. Let $x \in \tilde{\cO}_1$. By definition of tangent space \eqref{eq:tangentspaceatpoint} of a sc-retract, we have
$$T_x\cO = D_xr(\mathbb{R}^n \times \mathbb{E}) \subset \mathbb{R}^n \times \mathbb{E}.$$
Since $r$ preserves the $\mathbb{R}^n$-coordinate by the sliced retraction property \eqref{eq:sliceretraction}, the same is true for $D_xr,$ from which it follows that $$D_xr(\mathbb{R}^n \times \mathbb{E}) \cap (\{0\} \times \mathbb{E}) = D_xr(\{0\} \times \mathbb{E}).$$ Hence we have
\begin{equation} \label{eq:tangentspaceintersectinduced}
T_x\cO \cap (\{0\} \times \mathbb{E}) = D_x\tilde{r}(\{0\} \times \mathbb{E}) = T_x {\tilde{\cO}},
\end{equation}
proving \eqref{eq:inducedtangentspacerelationship}. Moreover, the projection $\pi : T_x\cO \rightarrow \mathbb{R}^n$ to the first factor of $\mathbb{R}^n \times \mathbb{E}$ is a surjection. Since $\ker(\pi) = T_x\cO \cap (\{0\} \times \mathbb{E})$, we conclude that $\pi$ induces an isomorphism $T_x\cO / T_x\tilde{\cO} \cong \mathbb{R}^n,$ proving \eqref{eq:codimentangentspacesslicedretract}.

To verify \eqref{eq:inducedreducedtangentspacerelationship}, first note that by \eqref{eq:Exstandard},\eqref{eq:Ex}, we have
\begin{equation} \label{eq:inducedreducedtangentproof}
(\mathbb{R}^n \times \mathbb{E})_x \cap (\{0\} \times \mathbb{E}) = (\{0\} \times \mathbb{E})_x,
\end{equation}
and hence we have
\begin{align*}
T_x^R\tilde{\cO} &= T_x\tilde{\cO} \cap (\{0\} \times \mathbb{E})_x\\
&\stackrel{\eqref{eq:inducedreducedtangentproof}}{=} T_x\tilde{\cO} \cap (\mathbb{R}^n \times \mathbb{E})_x \cap (\{0\} \times \mathbb{E})\\
&\stackrel{\eqref{eq:tangentspaceintersectinduced}}{=} T_x\cO \cap (\mathbb{R}^n \times \mathbb{E})_x \cap (\{0\} \times \mathbb{E})\\
&= T_x^R\cO \cap (\{0\} \times \mathbb{E}),
\end{align*}
as required.

To verify \eqref{eq:inducedreducatedtangentspacequotientrelationship}, it suffices to consider the case $$x = 0$$ and $$C = [0,\infty)^s \times \mathbb{E}' \subset \mathbb{R}^s \times \mathbb{E}' = \mathbb{E}$$ is in standard form \eqref{eq:partialquadrantstandardform}, since $x$ is assumed to be a smooth point. By \eqref{eq:Exstandard}, we have $(\mathbb{R}^n \times \mathbb{R}^s \times \mathbb{E}')_x = \mathbb{R}^n \times \{0\}^s \times \mathbb{E}'$. Since $r$ preserves the degeneracy index by Definition~\ref{dfn:tameretraction}, we claim it follows that
\begin{equation}\label{eq:tangentmapreservescorner}
D_xr(\mathbb{R}^n \times \{0\}^s \times \mathbb{E}') \subset \mathbb{R}^n \times \{0\}^s \times \mathbb{E}'.
\end{equation}
Indeed, to prove \eqref{eq:tangentmapreservescorner}, consider first a smooth point $\xi \in \mathbb{R}^n \times \{0\}^s \times E'_{\infty}$.  There exists a sc-smooth path $\alpha : (-\epsilon,\epsilon) \rightarrow U \cap (\mathbb{R}^n \times \{0\}^s \times \mathbb{E}')$ satisfying $\alpha(0) = x$ and $\alpha'(0) = \xi$, for example take $\alpha(t) = x + t\xi$. Since $r$ preserves degeneracy index we have $r \circ \alpha(t) \in \mathbb{R}^n \times \{0\}^s \times \mathbb{E}'$ for all $t$, and hence we have $D_xr(\xi) = D_x(r \circ \alpha)(0) \in \mathbb{R}^n \times \{0\}^s \times \mathbb{E}'$. This proves \eqref{eq:tangentmapreservescorner} for smooth points $\xi$. For an arbitrary point $\xi \in \mathbb{R}^n \times \{0\}^s \times \mathbb{E}'$, the result follows by considering a sequence $\{\xi_k\}_{k \geq 0}$ of smooth points that converges to $\xi$, which exists by density of the inclusion $E'_{\infty} \subset E'_0$, and applying continuity of $D_xr$.

We now finish verifying \eqref{eq:inducedreducatedtangentspacequotientrelationship}. Since $D_xr$ preserves the $\mathbb{R}^n$-coordinate, the projection $\pi : T_x\cO \rightarrow \mathbb{R}^n$ is surjective, and its kernel is $T_x\tilde{\cO}$ by \eqref{eq:inducedtangentspacerelationship}. By definition \eqref{eq:reducedtangentspacedef} we have $T_x^R\cO = D_xr(\mathbb{R}^n \times \mathbb{R}^s \times \mathbb{E'}) \cap (\mathbb{R}^n \times \{0\}^s \times \mathbb{E}')$, and so it follows from \eqref{eq:tangentmapreservescorner} that $\pi$ restricts to a surjection $T_x^R\cO \rightarrow \mathbb{R}^n$. The kernel of this surjection is $T_x^R\tilde{\cO}$ by \eqref{eq:inducedreducedtangentspacerelationship}. So we have a short exact sequence of sc-Banach spaces $0 \rightarrow T_x^R\tilde{\cO} \rightarrow T_x^R\cO \xrightarrow{\pi} \mathbb{R}^n \rightarrow 0$ that includes into the short exact sequence $0 \rightarrow T_x\tilde{\cO} \rightarrow T_x\cO \xrightarrow{\pi} \mathbb{R}^n.$ This implies \eqref{eq:inducedreducatedtangentspacequotientrelationship}.

To prove the lemma, it remains to show that $\tilde{r}$ is tame. The $\mathbb{R}^n$-sliced sc-retraction $r$ is tame by definition. Hence, for all $x \in \tilde{U}$, we compute, using \eqref{eq:degslice} and Definition~\ref{dfn:tameretraction}(1),
$$d_{\{0\} \times C}(\tilde{r}(x)) = d_{\mathbb{R}^n \times C}(\tilde{r}(x)) = d_{\mathbb{R}^n \times C}(r(x)) = d_{\mathbb{R}^n \times C}(x) = d_{\{0\} \times C}(x),$$
verifying that Definition~\ref{dfn:tameretraction}(1) holds for $\tilde{r}$.

To verify that $\tilde{r}$ satisfies Definition~\ref{dfn:tameretraction}(2), let $x \in \tilde{\cO}_{\infty}.$ Then $x \in \cO_{\infty}$, and so by the corresponding property of $r$ and by \cite[Prop.~2.9]{HWZbook}, the sc-subspace $A := (id_{\mathbb{R}^n \times \mathbb{E}} - D_xr)(\mathbb{R}^n \times \mathbb{E})$ of $\mathbb{R}^n \times \mathbb{E}$ satisfies
\begin{equation} \label{eq:tameproofbaseEx}
\mathbb{R}^n \times \mathbb{E} = T_x\cO \oplus A
\end{equation}
and $A \subset (\mathbb{R}^n \times \mathbb{E})_x.$ By the sliced retraction property \eqref{eq:sliceretraction} of $r$ and the definition of $A$, we conclude that $A \subset \{0\} \times \mathbb{E}$ holds. Then we have $A \subset (\mathbb{R}^n \times \mathbb{E})_x \cap (\{0\} \times \mathbb{E}) = (\{0\} \times \mathbb{E})_x$ by \eqref{eq:inducedreducedtangentproof}. We claim that $\{0\} \times \mathbb{E} = T_x\tilde{\cO} \oplus A$ holds, completing the proof that $\tilde{r}$ is tame. Indeed, it follows from \eqref{eq:inducedtangentspacerelationship}, \eqref{eq:tameproofbaseEx}, and $A \subset \{0\} \times \mathbb{E}$ that we have
$$\{0\} \times \mathbb{E} = (T_x \cO \cap (\{0\} \times \mathbb{E})) \oplus A = T_x\tilde{\cO} \oplus A.$$
\end{proof}

\subsection{Sliced bundle retracts} \label{subsec:slicebundleretracts}

In this section, we introduce the new notion of $\mathbb{R}^n$-sliced bundle retracts (Definition~\ref{dfn:slicebundleretraction}). In Lemma~\ref{lem:inducedtamebundleretraction}, we prove that a $\mathbb{R}^n$-sliced bundle retract $K$ covering a $\mathbb{R}^n$-sliced sc-retract $\cO$ induces a tame bundle retract $\tilde{K} \subset K$ covering the induced tame sc-retract $\tilde{\cO} \subset \cO$ from Lemma~\ref{lem:inducedtameretraction}. The global definition of a slice $\tilde{\cB} \subset \cB$ of a bundle $\rho : \cE \rightarrow \cB$ (Definition~\ref{dfn:globalslice}) is then a subspace such that around every point $x \in \tilde{\cB}$ there is a bundle chart for $\rho$ to a $\mathbb{R}^{n_x}$-sliced bundle retract $K$ that locally identifies $\rho^{-1}(\tilde{\cB})$ with the induced tame bundle retract $\tilde{K}$.

We first recall the local structure of tame strong bundles: tame bundle retracts. Consider a relatively open subset $U$ of a partial quadrant $C$ of a sc-Banach space $\mathbb{E}$, and another sc-Banach space $\mathbb{F}$. Then the trivial bundle
\begin{equation} \label{eq:trivialbundleoverquadrant}
U \lhd \mathbb{F} \rightarrow U
\end{equation}
has total space $U \lhd \mathbb{F} = U \times \mathbb{F}$ as a set, and the map is projection onto $U$. The triangle $\lhd$ signifies the extra structure of a double filtration on the set $U \times \mathbb{F}$. That is, for $0 \leq k \leq m+1$, we have
$$(U \lhd \mathbb{F})_{m,k} := U_m \oplus F_k.$$
Then, for $i = 0,1$, we define the sc-structure $(U \lhd \mathbb{F})[i]$ by
\begin{equation} \label{eq:doublefiltrationlocal}
((U \lhd \mathbb{F})[i])_m := U_m \oplus F_{m+i}, \,\, m \geq 0.
\end{equation}
The purpose of defining these two filtrations is that they correspond to the two notions of smoothness of a section of a bundle that are important for polyfold theory. Precisely, a section $s : U \rightarrow U \lhd \mathbb{F}$ is called \emph{sc-smooth} if it is sc-smooth as a map to $(U \lhd \mathbb{F})[0]$. If, moreover, we have $s(U) \subset (U\lhd\mathbb{F})[1]$ and the map $s : U \rightarrow (U\lhd\mathbb{F})[1]$ is sc-smooth, then $s$ is called a \emph{sc$^+$-section}. See \cite[Def.~2.24]{HWZbook} for a detailed discussion. We remark that sc$^+$-sections are the sc-analogue of compact perturbations from classical Fredholm theory; in particular, the sc-Fredholm property is stable under sc$^+$-perturbation \cite[Thm.~3.2]{HWZbook}.

A \emph{strong bundle map} \cite[Def.~2.22]{HWZbook} $$\Phi : U \lhd \mathbb{F} \rightarrow U' \lhd \mathbb{F}'$$ is a map which preserves the double filtration and is of the form $$\Phi(x,\xi) = (\varphi(x),\Gamma(x,\xi)),$$ where the map $\Gamma : U \lhd \mathbb{F} \rightarrow \mathbb{F}'$ is linear in $\xi$. In addition, for $i = 0,1,$ we require that the maps
$\Phi : (U \lhd \mathbb{F})[i] \rightarrow (U' \lhd \mathbb{F}')[i]$
are sc-smooth. A \emph{strong bundle isomorphism} is an invertible strong bundle map whose inverse is also a strong bundle map.

To extend \eqref{eq:trivialbundleoverquadrant} to a notion of a local bundle model over a sc-retract, we employ the following notion of a retraction in the fibers. A \emph{strong bundle retraction} is a strong bundle map $R : U \lhd \mathbb{F} \rightarrow U \lhd \mathbb{F}$ satisfying $R \circ R = R$. As a consequence, the map $R$ has the form
\begin{equation} \label{eq:bundleretractionwrittentout}
R(x,\xi) = (r(x),\Gamma(x,\xi)),
\end{equation}
where $r : U \rightarrow U$ is a sc-smooth retraction and $\G(x,\cdot) : \mathbb{F} \rightarrow \mathbb{F}$ is a linear projection for every $x \in U$. If $r$ is tame, then $R$ is called a \emph{tame strong bundle retraction}. The image $K := R(U \lhd \mathbb{F})$ of $R$ is called a \emph{strong bundle retract} \cite[Def.~2.23]{HWZbook}, as is the triple $(K,C \lhd \mathbb{F}, \mathbb{E} \lhd \mathbb{F}).$ We say that $K$ \emph{covers the sc-retract} $\cO = r(U)$. If $R$ is tame, then $K$ is called a \emph{tame strong bundle retract}. The projection $U \lhd \mathbb{F} \rightarrow U$ induces a mapping $K \rightarrow \cO$, which we call a \emph{strong local bundle model}.

We now introduce the new notion of a $\mathbb{R}^n$-sliced bundle retract.

\begin{dfn} \label{dfn:slicebundleretraction}
Consider a partial quadrant $C$ of a sc-Banach space $\mathbb{E}$, an open subset $U \subset \mathbb{R}^n \times C$ for some $n \geq 0$, and another sc-Banach space $\mathbb{F}$.

A tame bundle retraction $R : U \lhd \mathbb{F} \rightarrow U \lhd \mathbb{F}$ is called a {\bf $\mathbold{\mathbb{R}^n}$-sliced bundle retraction} if the tame sc-retraction $r : U \rightarrow U$ covered by $R$ (see \eqref{eq:bundleretractionwrittentout}) is a $\mathbb{R}^n$-sliced sc-retraction (Definition~\ref{dfn:sliceretraction}).

If so, then the tame bundle retract $K = R(U \lhd \mathbb{F})$ (and the triple $(K,\mathbb{R}^n \times C \lhd \mathbb{F}, \mathbb{R}^n \times \mathbb{E} \lhd \mathbb{F})$) is called a {\bf $\mathbold{\mathbb{R}^n}$-sliced bundle retract} and the tame local bundle model $K \rightarrow \cO := r(U)$ is called a {\bf $\mathbold{\mathbb{R}^n}$-sliced local bundle model}.
\end{dfn}

In the following lemma, we show that for any $\mathbb{R}^n$-sliced bundle retract $K$ in $\mathbb{R}^n \times C \lhd \mathbb{F}$, the set $\tilde{K} := K \cap (\{0\} \times C \lhd \mathbb{F})$ is a tame bundle retract. Later, we use the inclusion $\tilde{K} \subset K$ to define the local models for the restriction of a bundle to a slice (Definition~\ref{dfn:globalslice}).

\begin{lem} \label{lem:inducedtamebundleretraction}
Consider a partial quadrant $C$ of a sc-Banach space $\mathbb{E}$, another sc-Banach space $\mathbb{F}$, and a $\mathbb{R}^n$-sliced bundle retract $$(K,\mathbb{R}^n \times C \lhd \mathbb{F}, \mathbb{R}^n \times \mathbb{E} \lhd \mathbb{F})$$ covering a $\mathbb{R}^n$-sliced sc-retract $(\cO, \mathbb{R}^n \times C, \mathbb{R}^n \times \mathbb{E})$. Let $\pi : K \rightarrow \cO$ denote the local bundle model given by restriction of the projection along the fiber $\mathbb{R}^n \times C \lhd \mathbb{F} \rightarrow \mathbb{R}^n \times C$.

Then, for any open subset $U \subset \mathbb{R}^n \times C$ and $\mathbb{R}^n$-sliced bundle retraction $R : U \lhd \mathbb{F} \rightarrow U \lhd \mathbb{F}$ covering a $\mathbb{R}^n$-sliced sc-retraction $r : U \rightarrow U$ such that $r(U) = \cO$ and $R(U \lhd \mathbb{F}) = K$, the set $\tilde{U} := U \cap (\{0\} \times C)$ is open in $C$ and the restriction
$$\tilde{R} := R|_{\tilde{U}  \lhd \mathbb{F}} : \tilde{U} \lhd \mathbb{F} \rightarrow \tilde{U} \lhd \mathbb{F}$$
is a tame bundle retraction onto $\tilde{K} := \tilde{R}(\tilde{U} \lhd \mathbb{F})$ covering the induced tame sc-retraction $\tilde{r} : \tilde{U} \rightarrow \tilde{U}$ onto the induced tame sc-retract $\tilde{\cO}$ defined in Lemma~\ref{lem:inducedtameretraction}. We call $\tilde{R}$ the {\bf tame bundle retraction induced by $R$}.

Moreover,
\begin{equation} \label{eq:inducedscretractbundle}
\tilde{K} = K \cap (\{0\} \times C \lhd \mathbb{F}) = \pi^{-1}(\tilde{\cO})
\end{equation}
holds, so in particular $\tilde{K}$ does not depend on the choices of $U$ and $R$. We may view $\tilde{K}$ as a subset of $C \lhd \mathbb{F}$, and we call $$(\tilde{K},C \lhd \mathbb{F},\mathbb{E} \lhd \mathbb{F})$$ the {\bf tame bundle retract induced by the $\mathbb{R}^n$-sliced bundle retract} $(K, \mathbb{R}^n \times C \lhd \mathbb{F}, \mathbb{R}^n \times \mathbb{E} \lhd \mathbb{F})$.

In particular, the $\mathbb{R}^n$-sliced local bundle model $\pi : K \rightarrow \cO$ restricted to $\tilde{K}$ is a tame local bundle model
\begin{equation} \label{eq:inducedbundlemodel}
\tilde{\pi} := \pi|_{\tilde{K}} : \tilde{K} \rightarrow \tilde{\cO},
\end{equation}
which we call the {\bf induced tame local bundle model}.
\end{lem}
\begin{proof} The map $\tilde{R}$ is a strong bundle map and satisfies $\tilde{R} \circ \tilde{R} = \tilde{R}$, by the corresponding properties of $R$, so $\tilde{R}$ is a strong bundle retraction. Moreover, $\tilde{R}$ is tame because it covers the tame sc-retraction $\tilde{r}$.

We now verify that the first equality in \eqref{eq:inducedscretractbundle} holds. The forwards inclusion is immediate from the definitions. To see the reverse inclusion, let $(x \lhd \xi) \in K \cap (\{0\} \times C \lhd \mathbb{F}).$ Then $x = \pi(x \lhd \xi) \in \cO \cap (\{0\} \times C) = \tilde{\cO}$, where $\pi : K \rightarrow \cO$ is the local bundle model given by restriction of the projection along the fiber $\pi : \mathbb{R}^n \times C \lhd \mathbb{F} \rightarrow \mathbb{R}^n \times C$. In particular, $x \lhd \xi$ is in the domain of $\tilde{R} = R|_{\tilde{U} \lhd \mathbb{F}}$. Since $R$ is a retraction onto $K$, it follows that $(x \lhd \xi) = R(x \lhd \xi) = \tilde{R}(x \lhd \xi) \in im(\tilde{R}) = \tilde{K}$, as required. The second equality in \eqref{eq:inducedscretractbundle} holds from the definitions and $\tilde{\cO} = \cO \cap (\{0\} \times C)$.
\end{proof}

\subsection{Sliced sc-Fredholm germs} \label{subsec:slicescfredholmgerms}

In this section, we review the standard notion of a local sc-Fredholm germ (Definition~\ref{dfn:localscFredgerm}) and we introduce the new notions of a tame sc-Fredholm germ (Definition~\ref{dfn:tamescFredgerm}) and a $\mathbb{R}^n$-sliced sc-Fredholm germ (Definition~\ref{dfn:slicescFredholmgerm}). We later use tame sc-Fredholm germs as the local models for our new notion of a tame sc-Fredholm section $\s : \cB \rightarrow \cE$ (Definition~\ref{dfn:tamescfredholmsection}) of a tame strong bundle $\rho : \cE \rightarrow \cB$, and we use $\mathbb{R}^n$-sliced sc-Fredholm germs as the local models in our definition of a slice $\tilde{\cB} \subset \cB$ (Definition~\ref{dfn:globalslice}) of a sc-Fredholm section $\s$.

We prove in Lemma~\ref{lem:inducedlocalscFredholmgerm} that a $\mathbb{R}^n$-sliced sc-Fredholm germ $$\s : \cO \rightarrow K$$ of a $\mathbb{R}^n$-sliced local bundle model $K \rightarrow \cO$ induces a tame sc-Fredholm germ $$\tilde{\s} := \s|_{\tilde\cO}: \tilde{\cO} \rightarrow \tilde{K}$$ of the induced tame local bundle model $\tilde{K} \rightarrow \tilde{\cO}$ from \eqref{eq:inducedbundlemodel}. The sc-Fredholm index \eqref{eq:fredholmindexdef} satisfies $$ind(\tilde{\s}) = ind(\s) - n.$$

We now briefly review sc-germ language. Let $C$ be a partial quadrant of a sc-Banach space $\mathbb{E}$. Then a \emph{sc-germ of neighborhoods around $0 \in C$} \cite[Def.~3.1]{HWZbook}, denoted
$$U(C,0),$$
is a sequence
$$U = U(0) \supset U(1) \supset U(2) \supset \cdots$$
where $U(m)$ is a relatively open neighborhood of $0$ in $C \cap E_m$. We often write $U = U(C,0)$ for brevity. A \emph{sc$^0$-germ} $h : U(C,0) \rightarrow \mathbb{F}$ \cite[Def.~3.2]{HWZbook} into the sc-Banach space $\mathbb{F}$ is a continuous map $h : U(0) \rightarrow \mathbb{F}$ such that $h(U(m)) \subset F_m$ and $h : U(m) \rightarrow F_m$ is continuous for all $m \geq 0$. A \emph{sc$^1$-germ} $h : U(C,0) \rightarrow \mathbb{F}$ is a sc$^0$-germ which of is class sc$^1$ in the same sense as for a usual map except that the sc-differential $D_xh$ is required to exist only for $x \in U(1)$ (where $U(1)$ can be smaller than $U(0) \cap E_1$ in the germ case); see \cite[Def.~3.2]{HWZbook} for a precise definition. We can in turn define a tangent map on the tangent of a germ and then iterate the notion of sc$^1$ to define a \emph{sc-smooth germ}.

It is convenient to denote any section of a trivial bundle by $$h(x) = (x,{\bf h}(x)) : U \rightarrow U \lhd \mathbb{F}.$$ We refer to ${\bf h} : U \rightarrow \mathbb{F}$ as the \emph{principal part} of $h$.

Recall the standard notion of a local sc-Fredholm germ.

\begin{dfn} \label{dfn:localscFredgerm} \cite[Defs.~3.5, 3.6, 3.7]{HWZbook}
Consider a tame strong bundle retract $(K,C \lhd \mathbb{F},\mathbb{E} \lhd \mathbb{F})$ covering the tame sc-retract $(\cO,C,\mathbb{E})$ and a sc-smooth section $\s : \cO \rightarrow K$ of the local bundle model $K \rightarrow \cO$. Assume $0 \in \cO$.

Then $\s$ is called a {\bf local sc-Fredholm germ} if the following conditions hold:
\begin{enumilistalph}
\item There exists a sc-germ of neighborhoods $U(C,0)$ around $0 \in C$ and a tame sc-retraction $r : U \rightarrow U$ onto $r(U) = \cO$ covered by a tame bundle retraction $R : U \lhd \mathbb{F} \rightarrow U \lhd \mathbb{F}$ onto $R(U \lhd \mathbb{F}) = K$.

\item The principal part $\mathbold{\s} : \cO \rightarrow \mathbb{F}$ of $\s$ has the property that the composition
$$\mathbold{\s} \circ r : U(C,0) \rightarrow \mathbb{F}$$
possesses a {\bf filling}
$$h : U(C,0) \rightarrow U(C,0) \lhd \mathbb{F},$$
which is a section of the trivial bundle $U(C,0) \lhd \mathbb{F}  \rightarrow U(C,0)$ whose principal part $\mathbold{h} : U(C,0) \rightarrow \mathbb{F}$ is a sc-smooth germ and such that the following conditions $(i)-(iii)$ hold. Recall that $R$ is of the form $R(x,\xi) = (r(x),\G(x,\xi))$ where $\G(x,-) : \mathbb{F} \rightarrow \mathbb{F}$ is a linear projection.\\

\begin{enumerate}
\item $\s(x) = h(x)$ for $x \in \cO$.

\item If $x \in U$ and ${\bf h}(x) = \Gamma(r(x),{\bf h}(x))$, then $x \in \cO$.

\item The linearization of the map $x \mapsto (id_{\mathbb{F}} - \Gamma(r(x),\cdot)) \circ {\bf h}(x)$ at the point $0$, restricted to $\ker D_0r$, defines a linear sc-isomorphism $\ker D_0 r \rightarrow \ker \Gamma(0,\cdot)$.
\end{enumerate}

\item There exists a sc$^+$-section $\mathfrak{s} : U \rightarrow U \lhd \mathbb{F}$ satisfying $\mathfrak{s}(0) = h(0)$, a sc-Banach space $\mathbb{W}$, a sc-germ of neighborhoods
$$U' = U'(C',0) \subset C' = [0,\infty)^s \times \mathbb{R}^{k-s} \times \mathbb{W}$$
centered around $0 \in C'$ for some $k \geq s \geq 0$, and a strong bundle isomorphism
$$\Psi : U \lhd \mathbb{F} \rightarrow U' \lhd \mathbb{R}^{k'} \times \mathbb{W}$$
(for some $k' \geq 0$) covering a sc-diffeomorphism
$$\psi : U \rightarrow U'$$
satisfying $\psi(0) = 0$ and such that the principal part of the section
$$b := \Psi \circ (h - \mathfrak{s}) \circ \psi^{-1} : U' \rightarrow U' \lhd \mathbb{R}^{k'} \times \mathbb{W}$$
is a {\bf basic germ}. This means that the principal part
$${\bf b} :  U' \rightarrow \mathbb{R}^{k'} \times \mathbb{W}$$
of $b$ is a sc-smooth germ satisfying ${\bf b}(0) = 0$ and having the following property: Letting $P : \mathbb{R}^{k'} \times \mathbb{W} \rightarrow \mathbb{W}$ denote projection onto $\mathbb{W}$, the germ $P \circ {\bf b} : U' \rightarrow \mathbb{W}$ is of the form
$$P \circ {\bf b}(a,w) = w - B(a,w),$$
for $(a,w) \in ([0,\infty)^s \times \mathbb{R}^{k-s}) \times \mathbb{W}$, where $B$ is a sc-smooth germ and $B(0) = 0$. Also, $B$ is required to satisfy a contraction property: For every $\epsilon > 0$ and integer $m \geq 0$ there exists $\delta > 0$ such that
\begin{equation} \label{eq:basicgerm}
||(a,w)||_m, \,\, ||(a,w')||_m < \delta \implies ||B(a,w) - B(a,w')||_m \leq \epsilon ||w - w'||_m.
\end{equation}
The notation $|| \cdot ||_m$ means the Banach norm in the $m$-level of the sc-structure. We use the convention\footnote{This convention is equivalent to using any norm on $\mathbb{R}^k$ and any standard norm on a Cartesian product that is equivalent to the sum norm.} that the $m$-norm $||\cdot||_m$ on ${(\mathbb{R}^k \times \mathbb{W})_m = \mathbb{R}^k \times W_m}$ is the sum of the standard Euclidean norm on $\mathbb{R}^k$ plus the $W_m$-norm.
\end{enumilistalph}

The {\bf index of the local sc-Fredholm germ} $\s$ is the integer
\begin{equation} \label{eq:fredholmindexdef}
ind(\s) := k - k',
\end{equation}
where $k$ and $k'$ are the dimensions of the finite dimensional spaces split off in the domain $U' \subset ([0,\infty)^s \times \mathbb{R}^{k-s}) \times \mathbb{W}$ and codomain $\mathbb{R}^{k'} \times \mathbb{W}$, respectively, of the basic germ ${\bf b}$.
\end{dfn}

We now introduce a new class of sc-Fredholm germs, called tame sc-Fredholm. All currently known sc-Fredholm germs arising in applications to symplectic topology are tame sc-Fredholm: The required linearity of $\psi$ holds in applications because $\psi$ is essentially given by splitting off the kernel of a linear map (see Section~\ref{sec:evaluationmaps}). This linearity is required in the construction of slice coordinates in Section~\ref{sec:slicecoordinates}.

\begin{dfn} \label{dfn:tamescFredgerm}
A local sc-Fredholm germ $\s : \cO \rightarrow K$ is called a {\bf tame sc-Fredholm germ} if the structures that exist by Definition~\ref{dfn:localscFredgerm} of local sc-Fredholm germ can be chosen such that the partial quadrant $C$ is in the standard form \eqref{eq:partialquadrantstandardform}, i.e.\ $C = [0,\infty)^s \times \mathbb{E}' \subset  \mathbb{R}^s \times \mathbb{E}' = \mathbb{E}$ for some sc-Banach space $\mathbb{E}'$ and integer $s \geq 0$, and such that the required sc-diffeomorphism $\psi : U \rightarrow U'$ in Definition~\ref{dfn:localscFredgerm}(c) is the restriction of a linear sc-isomorphism of the form
\begin{align*}
\psi = id_{[0,\infty)^s} \times \overline{\psi} : [0,\infty)^s \times \mathbb{E}' &\rightarrow [0,\infty)^s \times \mathbb{R}^{k-s} \times \mathbb{W},\\
(v,e) &\mapsto (v, \overline{\psi}(e))
\end{align*}
for some linear sc-isomorphism 
$$\overline{\psi} : \mathbb{E}' \rightarrow \mathbb{R}^{k-s} \times \mathbb{W}.$$
\end{dfn}

The following new class of sc-Fredholm germs, called $\mathbb{R}^n$-sliced sc-Fredholm, is defined for sections of sliced local bundle models $K \rightarrow \cO$ (Definition~\ref{dfn:slicebundleretraction}). We prove in Lemma~\ref{lem:inducedlocalscFredholmgerm} that a $\mathbb{R}^n$-sliced sc-Fredholm germ $\s : \cO \rightarrow K$ induces a tame sc-Fredholm germ $\tilde{\s} := \s|_{\tilde\cO}: \tilde{\cO} \rightarrow \tilde{K}$ of the induced tame local bundle model $\tilde{K} \rightarrow \tilde{\cO}$ defined in \eqref{eq:inducedbundlemodel}.

\begin{dfn} \label{dfn:slicescFredholmgerm}
Consider a $\mathbb{R}^n$-sliced bundle retract (Definition~\ref{dfn:slicebundleretraction})
$$(K, \mathbb{R}^n \times C \lhd \mathbb{F},\mathbb{R}^n \times \mathbb{E} \lhd \mathbb{F})$$ 
covering the $\mathbb{R}^n$-sliced sc-retract
$$(\cO,\mathbb{R}^n \times C,\mathbb{R}^n \times \mathbb{E}).$$ Assume $0 \in \cO$ and that the partial quadrant is in the standard form \eqref{eq:partialquadrantstandardform}, i.e.\ $C = [0,\infty)^s \times \mathbb{E}' \subset \mathbb{R}^s \times \mathbb{E}' = \mathbb{E}$ for some $s \geq 0$ and sc-Banach space $\mathbb{E}'$.

Then a local sc-Fredholm germ
$$\s : \cO \rightarrow K$$
of the $\mathbb{R}^n$-sliced local bundle model $K \rightarrow \cO$ is called a {\bf $\mathbold{\mathbb{R}^n}$-sliced sc-Fredholm germ} if the structures that exist by Definition~\ref{dfn:localscFredgerm} of local sc-Fredholm germ can be chosen such that the sc-Banach space $\mathbb{W}$ and sc-diffeomorphism $\psi : U \rightarrow U'$ from Definition~\ref{dfn:localscFredgerm}(c) have the following form: First, we have
$$\mathbb{W} = \tilde{\mathbb{W}} \times \mathbb{R}^n$$
for some other sc-Banach space $\tilde{\mathbb{W}}$. Moreover, the sc-diffeomorphism $\psi : U \rightarrow U'$ is of the form
\begin{align}
\psi : \mathbb{R}^n \times [0,\infty)^s \times \mathbb{E}' \supset U &\rightarrow U' \subset [0,\infty)^s \times \mathbb{R}^{k-s} \times \tilde{\mathbb{W}} \times \mathbb{R}^n \label{eq:slicedgermform}\\
(p,v,e) &\mapsto (v, \overline{\psi}(e), \lambda(p,v,e)) \nonumber
\end{align}
for some linear sc-isomorphism
$$\overline{\psi} : \mathbb{E}' \rightarrow \mathbb{R}^{k-s} \times \tilde{\mathbb{W}}$$
and such that the map
$$(p,v,e) \mapsto \lambda(p,v,e) \in \mathbb{R}^n$$
is $C^1$ on all levels $\mathbb{R}^n \times [0,\infty)^s \times E'_m$ for $m \geq 0$.
\end{dfn}

\begin{rmk} \label{rmk:levelshift}
Given a local sc-Fredholm germ $\s : \cO \rightarrow K$ of a local bundle model $K \rightarrow \cO$, all essential properties of the setup are preserved after restricting to the $m$-level of the sc-structure for any $m \geq 0$. Precisely, the $m$-shifted map $K^m \rightarrow \cO^m$ is a local bundle model and $\s|_{\cO^m} : \cO^m \rightarrow K^m$ is a local sc-Fredholm germ \cite[Cor.~5.1]{HWZbook} with the same index $ind(\s|_{\cO^m}) = ind(\s)$.

In particular, if a local sc-Fredholm germ $\s : \cO \rightarrow K$ satisfies the properties required of a $\mathbb{R}^n$-sliced sc-Fredholm germ except for the $C^1$ regularity of the map $(p,v,e) \mapsto \lambda(p,v,e)$, then the $1$-shifted section $\s|_{\cO^1} : \cO^1 \rightarrow K^1$ is $\mathbb{R}^n$-sliced sc-Fredholm since the map $\lambda$ has the required $C^1$ regularity on all levels $m \geq 1$ by \cite[Prop.~1.7]{HWZbook}.
\end{rmk}

In the following lemma, we prove that a $\mathbb{R}^n$-sliced sc-Fredholm germ of a $\mathbb{R}^n$-sliced local bundle model $K \rightarrow \cO$ (Definition~\ref{dfn:slicebundleretraction}) restricts to a tame sc-Fredholm germ of the induced tame local bundle model $\tilde{K} \rightarrow \tilde{\cO}$ from \eqref{eq:inducedbundlemodel}.

\begin{lem} \label{lem:inducedlocalscFredholmgerm}
Consider a $\mathbb{R}^n$-sliced bundle retract $$(K, \mathbb{R}^n \times C \lhd \mathbb{F},\mathbb{R}^n \times \mathbb{E} \lhd \mathbb{F})$$ covering a $\mathbb{R}^n$-sliced sc-retract $(\cO,\mathbb{R}^n \times C,\mathbb{R}^n \times \mathbb{E})$. Recall the induced tame local bundle model $\tilde{K} \rightarrow \tilde{\cO}$ from \eqref{eq:inducedbundlemodel}.

Let $\s : \cO \rightarrow K$ be a $\mathbb{R}^n$-sliced sc-Fredholm germ. Then, the restriction $\tilde{\s} := \s|_{\tilde{\cO}} : \tilde{\cO} \rightarrow \tilde{K}$ is a tame sc-Fredholm germ with index satisfying $ind(\tilde{\s}) = ind(\s) - n$.

We call $\tilde{\s}$ the {\bf tame sc-Fredholm germ induced by the $\mathbb{R}^n$-sliced sc-Fredholm germ} $\s$.
\end{lem}
\begin{proof}
Since $\s$ is a $\mathbb{R}^n$-sliced sc-Fredholm germ, the partial quadrant is in the standard form ${C = [0,\infty)^s \times \mathbb{E}' \subset \mathbb{R}^s \times \mathbb{E}' = \mathbb{E}}$ and there exists an open $U \subset \mathbb{R}^n \times C$ and a $\mathbb{R}^n$-sliced sc-retraction $r : U \rightarrow U$ onto $r(U) = \cO$ that satisfies the conditions of Definition~\ref{dfn:sliceretraction} and Definition~\ref{dfn:slicescFredholmgerm}, which we recall as we use them.

Since $r$ satisfies the properties in Definition~\ref{dfn:localscFredgerm} of local sc-Fredholm germ, we can assume that there exists a sc-germ of neighborhoods $U(\mathbb{R}^n \times C, 0)$ around $0$ satisfying $U(0) = U$ (i.e.\ the $0$-level open set in the germ is the open set $U$ from above) such that $r$ is covered by a tame bundle retraction $R : U \lhd \mathbb{F} \rightarrow U \lhd \mathbb{F}$ such that $\mathbold{\s} \circ r : U(\mathbb{R}^n \times C,0) \rightarrow \mathbb{F}$ posesses a filling
$$h : U(\mathbb{R}^n \times C,0) \rightarrow U(\mathbb{R}^n \times C,0) \lhd \mathbb{F}.$$

Since $r$ is a $\mathbb{R}^n$-sliced sc-retraction covered by $R$, it follows by Definition~\ref{dfn:slicebundleretraction} that $R$ is a $\mathbb{R}^n$-sliced bundle retraction. As in Lemma~\ref{lem:inducedtamebundleretraction}, set
$$\tilde{U} := U \cap (\{0\}^n \times C)$$
and denote the induced tame sc-retraction and tame bundle retraction by
\begin{align*}
\tilde{r} &: \tilde{U} \rightarrow \tilde{U},\\
\tilde{R} &: \tilde{U} \lhd \mathbb{F} \rightarrow \tilde{U} \lhd \mathbb{F},
\end{align*}
and the induced tame sc-retract and tame bundle retract by
\begin{align*}
\tilde{\cO} &= \tilde{r}(\tilde{U}) = \cO \cap (\{0\}^n \times C),\\
\tilde{K} &= \tilde{R}(\tilde{U} \lhd \mathbb{F}) = K \cap (\{0\}^n \times C \lhd \mathbb{F}).
\end{align*}
Also, let $\tilde{U}(C,0)$ be the sc-germ of neighborhoods given on level-$m$ by $\tilde{U}(m) := U(m) \cap (\{0\}^n \times C),$ or more concisely
$$\tilde{U}(C,0) := U(\mathbb{R}^n \times C,0) \cap (\{0\}^n \times C).$$

We claim that the restriction
$$\tilde{h} := h|_{\tilde{U}(C,0)} : \tilde{U}(C,0) \rightarrow \tilde{U}(C,0) \lhd \mathbb{F}$$
is a filling for the composition $\mathbold{\tilde{\s}} \circ \tilde{r} : \tilde{U}(C,0) \rightarrow \mathbb{F}$. It is a section of the trivial bundle $\tilde{U}(C,0) \lhd \mathbb{F} \rightarrow \tilde{U}(C,0)$ and its principal part ${\bf \tilde{h}}$ is a sc-smooth germ, because we have
$${\bf \tilde{h}} = {\bf h}|_{\tilde{U}(C,0)}$$
and the corresponding properties hold for ${\bf h}$. It remains to verify that the filler properties in Definition~\ref{dfn:localscFredgerm}(b).(i)-(iii) hold for $\tilde{\s}$ and $\tilde{h}$. These follow from the corresponding properties of $\s$ and $h$, as we now describe. Property (i) is immediate since $\tilde{\s}$ and $\tilde{h}$ are restrictions of $\s$ and $h$, respectively. Write $R(x,\xi) = (r(x), \G(x,\xi))$, as in \eqref{eq:bundleretractionwrittentout}. Consider the restriction $\tilde{\G} := \G|_{\tilde{U} \lhd \mathbb{F}}.$ Then we have $\tilde{R}(x,\xi) = (\tilde{r}(x),\tilde{\G}(x,\xi)).$ To verify $(ii)$, let $x \in \tilde{U}$ and assume ${\bf \tilde{h}}(x) = \tilde{\G}(\tilde{r}(x), {\bf \tilde{h}}(x)).$ It follows that ${\bf h}(x) = \Gamma(r(x),{\bf h}(x))$, which implies $x \in \cO$ by property (ii) for ${\bf h}$. Hence $x \in \tilde{U} \cap \cO = \tilde{\cO}$, as required. It remains to verify (iii) for ${\bf \tilde{h}}$. We claim that
\begin{equation} \label{eq:kernelrelationrestrictionproof}
\ker D_0{\tilde{r}} = \ker D_0r.
\end{equation}
The forwards inclusion follows from $\tilde{r}$ being the restriction of $r$. To see the reverse inclusion, let $\xi \in \ker D_0r.$ By the defining property \eqref{eq:sliceretraction} of a $\mathbb{R}^n$-sliced sc-retraction, we see that $D_0r$ preserves the $\mathbb{R}^n$-coordinate of $\xi$ and hence $\xi \in \{0\}^n \times \mathbb{R}^s \times \mathbb{E}'.$ In particular, $\xi$ is in the domain of $D_0\tilde{r}$, and moreover $D_0\tilde{r}(\xi) = D_0r(\xi) = 0$. Hence \eqref{eq:kernelrelationrestrictionproof} holds. We now verify that Definition~\ref{dfn:localscFredgerm}(iii) holds for ${\bf \tilde{h}}$. From the corresponding property of ${\bf h}$, the linearization $D_0L$ at $0$ of the map
\begin{align*}
L : U &\rightarrow \mathbb{F}\\
x &\mapsto (id_{\mathbb{F}} - \Gamma(r(x),\cdot)) \circ {\bf h}(x)
\end{align*}
restricts to a linear sc-isomorphism $\ker D_0r \rightarrow \ker \Gamma(0,\cdot)$. We must show that the linearization at $0$ of the map
\begin{align*}
\tilde{L} : \tilde{U} &\rightarrow \mathbb{F}\\
x &\mapsto (id_{\mathbb{F}} - \tilde{\G}(\tilde{r}(x),\cdot)) \circ \tilde{{\bf h}}(x)
\end{align*}
restricts to a linear sc-isomorphism $\ker D_0\tilde{r} \rightarrow \ker \tilde{\G}(0,\cdot) = \ker \G(0,\cdot)$. This follows from \eqref{eq:kernelrelationrestrictionproof} and since $\tilde{L}$ and $\tilde{r}$ are the restrictions of $L$ and $r$, respectively, to $\tilde{U}$. This completes the proof that $\tilde{h}$ is a filling for $\mathbold{\tilde{\s}} \circ \tilde{r},$ as claimed.

To show that $\tilde{\s}$ is a local sc-Fredholm germ, it remains to verify the properties in Definition~\ref{dfn:localscFredgerm}(c). Since $\s$ is a $\mathbb{R}^n$-sliced sc-Fredholm germ, the corresponding properties in Definition~\ref{dfn:localscFredgerm}(c) hold for $\s$, and in addition the stronger properties in the definition of $\mathbb{R}^n$-sliced sc-Fredholm germ (Definition~\ref{dfn:slicescFredholmgerm}) hold: There exists a sc$^+$-section $$\mathfrak{s} : U \rightarrow U \lhd \mathbb{F}$$ satisfying $\mathfrak{s}(0) = h(0)$, a sc-Banach space of the form $$\mathbb{W} = \tilde{\mathbb{W}} \times \mathbb{R}^n$$ for some other sc-Banach space $\tilde{\mathbb{W}}$, a sc-germ of neighborhoods
$$U' \text{ around } 0 \in ([0,\infty)^s \times \mathbb{R}^{k-s}) \times (\tilde{\mathbb{W}} \times \mathbb{R}^n),$$
and a strong bundle isomorphism
$$\Psi : U \lhd \mathbb{F} \rightarrow U' \lhd \mathbb{R}^{k'} \times (\tilde{\mathbb{W}} \times \mathbb{R}^n)$$
covering a sc-diffeomorphism $\psi : U \rightarrow U'$ satisfying $\psi(0) = 0$ and of the form
\begin{align*}
\psi : \mathbb{R}^n \times [0,\infty)^s \times \mathbb{E}' \supset U &\rightarrow U' \subset [0,\infty)^s \times \mathbb{R}^{k-s} \times \tilde{\mathbb{W}} \times \mathbb{R}^n\\
(p,v,e) &\mapsto (v, \overline{\psi}(e), \lambda(p,v,e))
\end{align*}
for some linear sc-isomorphism
$$\overline{\psi} : \mathbb{E}' \rightarrow \mathbb{R}^{k-s} \times \tilde{\mathbb{W}}$$
such that, on all levels $\mathbb{R}^n \times [0,\infty)^s \times E'_m$ for $m \geq 0$, the map
\begin{equation} \label{eq:theC1mapinproof}
(p,v,e) \mapsto \lambda(p,v,e) \in \mathbb{R}^n
\end{equation}
is $C^1$. Moreover, the principal part of the section
\begin{equation} \label{eq:basicgermrestrictionproof}
b := \Psi \circ (h - \mathfrak{s}) \circ \psi^{-1} : U' \rightarrow U' \lhd \mathbb{R}^{k'} \times (\tilde{\mathbb{W}} \times \mathbb{R}^n)
\end{equation}
is a basic germ, which means that, for all $a \in [0,\infty)^s \times \mathbb{R}^{k-s}$ and $z \in \tilde{\mathbb{W}} \times \mathbb{R}^n$ such that $(a,z) \in U'$, we have
\begin{equation} \label{eq:projpropertyinducedscfredgerm}
P \circ {\bf b}(a,z) = z - B(a,z),
\end{equation}
where $P : \mathbb{R}^{k'} \times (\tilde{\mathbb{W}} \times \mathbb{R}^n) \rightarrow (\tilde{\mathbb{W}} \times \mathbb{R}^n) $ is projection onto $\tilde{\mathbb{W}} \times \mathbb{R}^n$ and $B : U' \rightarrow \tilde{\mathbb{W}} \times \mathbb{R}^n$ is sc-smooth, satisfies $B(0) = 0$, and satisfies the contraction property \eqref{eq:basicgerm}.

Now, to verify that $\tilde{\s} = \s|_{\tilde{\cO}}$ inherits the local sc-Fredholm germ property, consider the restricted sc$^+$ section
\begin{align*}
\tilde{\mathfrak{s}} &:= \mathfrak{s}|_{\tilde{U}} : \tilde{U} \rightarrow \tilde{U} \lhd \mathbb{F},
\end{align*}
the linear sc-isomorphism given by
\begin{align*}
\tilde{\psi} : [0,\infty)^s \times \mathbb{E}' &\rightarrow [0,\infty)^s \times \mathbb{R}^{k-s} \times \tilde{\mathbb{W}}\\
(v,e) &\mapsto (v, \overline{\psi}(e)),
\end{align*}
the open set
\begin{align*}
\tilde{U'} &:= \tilde{\psi}(\tilde{U}),
\end{align*}
and the strong bundle isomorphism given by
\begin{align*}
\tilde{\Psi} : \tilde{U} \lhd \mathbb{F} &\rightarrow \tilde{U'} \lhd (\mathbb{R}^{k'} \times \mathbb{R}^n) \times \tilde{\mathbb{W}}\\
((v,e) \lhd \xi) &\mapsto (\tilde{\psi}(v,e) \lhd \eta(\Psi((v,e) \lhd \xi)),
\end{align*}
where $\eta : \mathbb{R}^{k'} \times (\tilde{\mathbb{W}} \times \mathbb{R}^n) \rightarrow (\mathbb{R}^{k'} \times \mathbb{R}^n) \times \tilde{\mathbb{W}}$ is the reordering of factors map.

We claim that the principal part of the section
\begin{equation} \label{eq:basicgermtwiddle}
\tilde{b} := \tilde{\Psi} \circ (\tilde{h} - \tilde{\mathfrak{s}}) \circ \tilde{\psi}^{-1} : \tilde{U'} \rightarrow \tilde{U'} \lhd (\mathbb{R}^{k'} \times \mathbb{R}^n) \times \tilde{\mathbb{W}}
\end{equation}
is a basic germ, where now $\tilde{\mathbb{W}}$ plays the role of the sc-Banach space $\mathbb{W}$ in Definition~\ref{dfn:localscFredgerm}(c). Let the maps
\begin{align*}
\tilde{P} &: (\mathbb{R}^{k'} \times \mathbb{R}^n) \times \tilde{\mathbb{W}} \rightarrow \tilde{\mathbb{W}},\\
\pi_{\tilde{\mathbb{W}}} &: \mathbb{R}^n \times \tilde{\mathbb{W}} \rightarrow \tilde{\mathbb{W}},
\end{align*}
be the projections onto the $\tilde{\mathbb{W}}$ factor in their respective domains. Notice that we have $\tilde{P} = \pi_{\tilde{\mathbb{W}}} \circ P \circ \eta^{-1}.$

We write $a \in [0,\infty)^s \times \mathbb{R}^{k-s}$ and $w \in \tilde{\mathbb{W}}$. Define the map
\begin{align*}
\t : \tilde{U}' &\rightarrow U'\\
(a,w) &\mapsto \psi(0,\tilde{\psi}^{-1}(a,w)),
\end{align*}
and from the definitions observe
\begin{equation} \label{eq:raw}
\t(a,w) = (a,w,\lambda(0,\tilde{\psi}^{-1}(a,w))).
\end{equation}
From the definitions and \eqref{eq:projpropertyinducedscfredgerm},\eqref{eq:raw}, we compute
\begin{align*}
\tilde{P} \circ \tilde{\bf b}(a,w) &= \pi_{\tilde{\mathbb{W}}} \circ \tilde{P} \circ  \pi_{\mathbb{R}^{k'} \times \mathbb{R}^n \times \tilde{\mathbb{W}}} \circ \tilde{\Psi} \circ (\tilde{h} - \tilde{\mathfrak{s}}) \circ \tilde{\psi}^{-1}(a,w)\\
&=  \pi_{\tilde{\mathbb{W}}} \circ  P \circ  \pi_{\mathbb{R}^{k'} \times \tilde{\mathbb{W}} \times \mathbb{R}^n} \circ \Psi \circ (h - \mathfrak{s})(0, \tilde{\psi}^{-1}(a,w))\\
&=  \pi_{\tilde{\mathbb{W}}} \circ  P \circ  \pi_{\mathbb{R}^{k'} \times \tilde{\mathbb{W}} \times \mathbb{R}^n} \circ \Psi \circ (h - \mathfrak{s})\\
&\,\,\,\,\,\,\,\,\,\, \circ (\psi^{-1} \circ \psi) \big (0, \tilde{\psi}^{-1}(a,w) \big )\\
&=  \pi_{\tilde{\mathbb{W}}} \circ  P \circ  {\bf b} \circ \psi(0,\tilde{\psi}^{-1} (a,w))\\
&=  \pi_{\tilde{\mathbb{W}}} \circ  P \circ  {\bf b} \circ \t (a,w)\\
&= \pi_{\tilde{\mathbb{W}}}((w,\lambda(0,\tilde{\psi}^{-1}(a,w))) - B \circ \t(a,w))\\
&= w - \pi_{\tilde{\mathbb{W}}} \circ B \circ \t(a,w).
\end{align*}
So, setting 
\begin{align*}
\tilde{B} := \pi_{\tilde{\mathbb{W}}} \circ B \circ \t : \tilde{U'} &\rightarrow \tilde{\mathbb{W}},
\end{align*}
we have $\tilde{P} \circ \tilde{\bf b}(a,w) = w - \tilde{B}(a,w)$. Hence to complete the proof that $\tilde{\bf b}$ is a basic germ it remains to show that  $\tilde{B}$ satisfies the contraction property \eqref{eq:basicgerm}.

Recall from \eqref{eq:raw} the map
\begin{align*}
\tilde{\lambda} : \tilde{U}' &\rightarrow \mathbb{R}^n\\
(a,w) &\mapsto \lambda(0,\tilde{\psi}^{-1}(a,w)).
\end{align*}
For all $m \geq 0$, this map restricted to the $m$-level $\tilde{U}'_m \rightarrow \mathbb{R}^n$ is $C^1$ since the map $\lambda$ from \eqref{eq:theC1mapinproof} is $C^1$ on every level by definition of $\mathbb{R}^n$-sliced sc-Fredholm germ, and since $\tilde{\psi}$ is a linear sc-isomorphism and hence $C^{\infty}$ on every level. So, there exist constants $\delta'_m > 0$ and $C_m > 0$ such that, for all $||(a,w)||_m,||(a,w')||_m < \delta'_m$, we have the $C^1$-estimate
\begin{equation} \label{eq:c1estimate}
||\tilde{\lambda}(a,w) - \tilde{\lambda}(a,w')||_{\mathbb{R}^n} \leq C_m \cdot ||(a,w) - (a,w')||_m = C_m \cdot ||w - w'||_m.
\end{equation}
Then by \eqref{eq:raw} and our convention that the norm on a Cartesian product is the sum norm (which is equivalent to any standard Banach norm on the product), we have the estimate
\begin{align}
||\t(a,w) - \t(a,w')||_m &=  ||(a,w,\tilde{\lambda}(a,w)) - (a,w',\tilde{\lambda}(a,w'))||_m \nonumber\\
&= ||\tilde{\lambda}(a,w) - \tilde{\lambda}(a,w')||_{\mathbb{R}^n} + ||w -w'||_m, \nonumber\\
&\leq (C_m + 1) \cdot ||w - w'||_m. \label{eq:basicproof}
\end{align}

We now verify the contraction property \eqref{eq:basicgerm} for $\tilde{B}$. Let $\epsilon > 0$ and $m \geq 0$. By the contraction property which is satisfied for $B$, there exists $\delta'_m > 0$ such that \eqref{eq:basicgerm} holds for $B$ and
$$\epsilon'_m := \epsilon / (C_m + 1).$$
Shrink $\delta'_m > 0$ to be smaller than the $\delta'_m$ for which \eqref{eq:basicproof} holds. We claim that
$$\delta_m := \delta'_m / (C_m + 1)$$
satisfies \eqref{eq:basicgerm} for $\tilde{B}$ and $\epsilon$. Indeed, consider some $$||(a,w)||_m, ||(a,w')||_m < \delta_m.$$ Then by $\t(0) = \psi(0,\tilde{\psi}^{-1}(0)) = 0$ and \eqref{eq:basicproof}, we have
$$||\t(a,w)||_m, ||\t(a,w')||_m < (C_m + 1) \cdot \delta_m = \delta'_m.$$
We compute, using the contraction property \eqref{eq:basicgerm} for $B$ and the estimate \eqref{eq:basicproof},
\begin{align*}
||\tilde{B}(a,w) - \tilde{B}(a,w')||_m &= ||\pi_{\tilde{\mathbb{W}}} \circ B \circ \t(a,w) - \pi_{\tilde{\mathbb{W}}} \circ B \circ \t(a,w')||_m\\
&\leq ||B \circ \t(a,w) - B \circ \t(a,w')||_m\\
&= ||B(a,w,\tilde{\lambda}(a,w)) - B(a,w',\tilde{\lambda}(a,w'))||_m\\
& \leq \epsilon'_m \cdot ||(w,\tilde{\lambda}(a,w)) - (w',\tilde{\lambda}(a,w'))||_m\\
& = \epsilon'_m \cdot ||(a,w,\tilde{\lambda}(a,w)) - (a,w',\tilde{\lambda}(a,w'))||_m\\
& = \epsilon'_m \cdot ||\t(a,w) - \t(a,w')||_m\\
&\leq (C_m + 1) \cdot \epsilon'_m \cdot ||w - w'||_m\\
&= \epsilon \cdot ||w - w'||_m.
\end{align*}
This completes the proof that \eqref{eq:basicgerm} holds for $\tilde{B}$, and hence that $\tilde{\bf b}$ is a basic germ.

We have shown that $\tilde{\s}$ is a local sc-Fredholm germ. Moreover, $\tilde{\s}$ is a tame sc-Fredholm germ because we have $\tilde{\psi} = id_{[0,\infty)^s} \times \overline{\psi}$ by definition of $\tilde{\psi}$, where $\overline{\psi} : \mathbb{E}' \rightarrow \mathbb{R}^{k-s} \times \tilde{\mathbb{W}}$ is the linear sc-isomorphism given by the $\mathbb{R}^n$-sliced sc-Fredholm germ property of $\s$.

The claimed index formula holds because, by definition of the sc-Fredholm index \eqref{eq:fredholmindexdef} and the form of the basic germs ${\bf b}$ \eqref{eq:basicgermrestrictionproof} and $\tilde{\bf b}$ \eqref{eq:basicgermtwiddle}, we have $ind(\s) = k - k'$ and $ind(\tilde{\s}) = k - (k' + n)$. 
\end{proof}

\section{\texorpdfstring{Slice coordinates for local submersions to $\mathbb{R}^n$}{Slice coordinates for local submersions to R\textasciicircum n}} \label{sec:slicecoordinates}

The purpose of this section is to prove Lemma~\ref{lem:slicecoordinates}, which generalizes the local submersion normal form (Lemma~\ref{lem:sclocalsubmersionnormalform}) for sc-smooth maps $f : U \rightarrow \mathbb{R}^n$ where the domain $U \subset [0,\infty)^s \times \mathbb{E}$ is open to maps $$f : \cO \rightarrow \mathbb{R}^n$$ where the domain is a tame sc-retract $(\cO,[0,\infty)^s \times \mathbb{E}, \mathbb{R}^s \times \mathbb{E})$. This means that the set $\cO = r(U)$ is the image of a tame sc-retraction $r : U \rightarrow U$ (see Definition~\ref{dfn:tameretraction}), which can have much more complicated local structure than the open set $U$, for example locally varying dimension. In this case, the local submersion normal form is obtained by a change of coordinates around any smooth point $x \in \cO_{\infty}$ at which $f$ is submersive and satisfies $f(x) = 0$ such that the sc-retract in the new coordinates is a $\mathbb{R}^n$-sliced sc-retract (Definition~\ref{dfn:sliceretraction}) with induced tame sc-retract (see Lemma~\ref{lem:inducedtameretraction}) identified with a neighborhood of $x$ in $f^{-1}(0) \cap \cO^1$. For this reason, we call the sc-diffeomorphism with this $\mathbb{R}^n$-sliced sc-retract \emph{slice coordinates around $x$}.

Moreover, given a tame local bundle model $K \rightarrow \cO$ and a tame sc-Fredholm germ $\s : \cO \rightarrow K$ (Definition~\ref{dfn:tamescFredgerm}), and assuming that $f$ is compatible with $\s$ as explained in Lemma~\ref{lem:slicecoordinates}(III), we prove that $\s$ in the new slice coordinates around $x$ is a $\mathbb{R}^n$-sliced sc-Fredholm germ (Definition~\ref{dfn:slicescFredholmgerm}) and its induced tame sc-Fredholm germ is identified with the restriction of $\s$ to $f^{-1}(0) \cap \cO^1$. See Section~\ref{sec:evaluationmaps} for a discussion of why evaluation maps $f = ev$ at marked points are compatible with the Cauchy-Riemann section $\s = \overline{\partial}_J$.

For simplicity, in this section we take all partial quadrants $C$ to be in standard form $C = [0,\infty)^s \times \mathbb{E} \subset \mathbb{R}^s \times \mathbb{E}$. There is no loss of generality because all partial quadrants are linearly sc-isomorphic to a standard partial quadrant by definition (see \eqref{eq:linearisotostandardquadrantdef}) and hence all sc-retracts are sc-diffeomorphic to sc-retracts in standard partial quadrants.

The notion of submersion that we use in Lemma~\ref{lem:slicecoordinates}, as in the local submersion normal form in sc-calculus (Lemma~\ref{lem:sclocalsubmersionnormalform}), requires surjectivity of the tangent map relative to the boundary: For a point $x \in \cO_{\infty} \cap (\{0\} \times \mathbb{E})$, we require that $D_xf(T_x\cO \cap (\{0\} \times \mathbb{E})) = \mathbb{R}^n$. Recall that $T^R_x\cO = T_x\cO \cap (\{0\} \times \mathbb{E})$ is the reduced tangent space at $x$ (see \eqref{eq:reducedtangentspacedef}). In the following lemma we interpret $T_x^R\cO$ as the tangent space at $x$ along the corner $\cO \cap (\{0\} \times \mathbb{E})$.
\begin{lem} \label{lem:associatedboundaryretract}
Consider a standard partial quadrant $[0,\infty)^s \times \mathbb{E}$ and a tame sc-retract $(\cO,[0,\infty)^s \times \mathbb{E}, \mathbb{R}^s \times \mathbb{E})$. Then, for any open subset $U \subset [0,\infty)^s \times \mathbb{E}$ and tame sc-retraction $r : U \rightarrow U$ onto $r(U) = \cO$, the restriction of $r$ to $U_{\partial} := U \cap (\{0\} \times \mathbb{E})$ is a tame sc-retraction
\begin{align*}
r_{\partial} &:= r|_{U_{\partial}} : U_{\partial} \rightarrow U_{\partial}
\end{align*}
onto
$$\cO_{\partial} := \cO \cap (\{0\} \times \mathbb{E}).$$
In particular, the tuple
$$(\cO_{\partial},\mathbb{E},\mathbb{E})$$
is a tame sc-retract, which we call the {\bf boundary sc-retract associated to the tame sc-retract $\cO$}.

Moreover, for all $x \in \cO_{\partial}^1$, we have
\begin{equation} \label{eq:boundaryretracttangentspace}
T_x\cO_{\partial} = T_x\cO \cap (\{0\} \times \mathbb{E}) = T_x^R\cO.
\end{equation}
\end{lem}
\begin{proof}
Since $r$ is tame, for $x \in U_{\partial}$ we have $$d_{[0,\infty)^s \times \mathbb{E}}(r(x)) = d_{[0,\infty)^s \times \mathbb{E}}(x) = s,$$ and hence $r(x) \in U_{\partial}$. Thus $r_{\partial}(U_{\partial}) \subset U_{\partial}$ holds. It is then immediate that $r_{\partial}$ is a sc-smooth retraction, as the restriction of the sc-smooth retraction $r$. Moreover, the domain $U_{\partial}$ of $r_{\partial}$ is an open subset of a cornerless partial quadrant, i.e.\ $U_{\partial} \subset \mathbb{E}$, and hence $r$ trivially satisfies the tameness hypotheses.

We now prove that the image of $r_{\partial}$ is $\cO_{\partial}$. Indeed, if $x \in im(r_{\partial})$ then $x \in \cO$ and $x \in U_{\partial} \subset \{0\} \times \mathbb{E}$, hence $x \in \cO_{\partial}$. For the reverse inclusion, if $x \in \cO \cap (\{0\} \times \mathbb{E})$ then $x \in U_{\partial}$ and hence $x = r(x) = r_{\partial}(x)$.

It remains to prove that \eqref{eq:boundaryretracttangentspace} holds. Let $\xi \in T_x\cO_{\partial}.$ Then $\xi \in \{0\} \times \mathbb{E}$ and $D_xr(\xi) = D_xr_{\partial}(\xi) = \xi,$ so $\xi \in T_x\cO.$ To prove the reverse inclusion, let $\xi \in T_x\cO \cap (\{0\} \times \mathbb{E}).$ Then we have $D_xr_{\partial}(\xi) = D_xr(\xi) = \xi$, which implies $\xi \in T_x\cO_{\partial}$, as required.
\end{proof}

In the following lemma, we construct the slice coordinates around a smooth point $x$ in a tame sc-retract $\cO$ at which a sc-smooth map $f : \cO \rightarrow \mathbb{R}^n$ is submersive on the tangent space $T_x\cO_{\partial}$ at $x$ to the associated boundary retract $\cO_{\partial}$ from Lemma~\ref{lem:associatedboundaryretract}, i.e.\ $D_x f(T_x\cO_{\partial}) = \mathbb{R}^n$.

The statement of the lemma is in three parts: (I) provides slice coordinates for $f^{-1}(0)$, (II) provides slice bundle coordinates for the restriction of a bundle retract to $f^{-1}(0)$, and (III) provides slice sc-Fredholm coordinates for the restriction of a sc-Fredholm section.

\begin{lem} \label{lem:slicecoordinates}
\text{}
\begin{enumilistfiber}
\item Consider a tame sc-retract $(\cO,[0,\infty)^s \times \mathbb{E},\mathbb{R}^s \times \mathbb{E})$ and a sc-smooth map $f : \cO \rightarrow \mathbb{R}^n$. Let $\cO_{\partial} = \cO \cap (\{0\}^s \times \mathbb{E})$ denote the boundary sc-retract associated to $\cO$ (see Lemma~\ref{lem:associatedboundaryretract}), and $f_{\partial} := f|_{\cO_{\partial}} : \cO_{\partial} \rightarrow \mathbb{R}^n$ the restriction of $f$. Suppose that, at some $x \in (\cO_{\partial})_{\infty}$ satisfying $f(x) = 0$, the tangent map $D_xf_{\partial} : T_x\cO_{\partial} \rightarrow \mathbb{R}^n$ is surjective.

Then, there exists an open subset
$$\hat{\cO} \subset \cO^1,$$
a sc-Banach space $\mathbb{K}$, a $\mathbb{R}^n$-sliced sc-retract
\begin{equation} \label{eq:normalformsliceretract}
(\cO',\mathbb{R}^n \times [0,\infty)^s \times \mathbb{K}^1, \mathbb{R}^n \times \mathbb{R}^s \times \mathbb{K}^1),
\end{equation}
and a sc-smooth diffeomorphism
$$g : \hat{\cO} \rightarrow \cO'$$
satisfying
\begin{equation} \label{eq:slicecoordinatesbasestatement}
g(f^{-1}(0) \cap \hat{\cO}) = \cO' \cap (\{0\}^n \times [0,\infty)^s \times \mathbb{K}^1) =: \tilde{\cO}'.
\end{equation}
Here, $\tilde{\cO}'$ is the tame sc-retract induced by $\cO'$ (see Lemma~\ref{lem:inducedtameretraction}).

We view $g$ as providing {\bf slice coordinates $\tilde{\cO}' \subset \cO'$ around $x$ with respect to $(f^{-1}(0) \cap \cO^1) \subset \cO^1$}.\\

\item Consider, in addition, a tame strong bundle retract $$(K,[0,\infty)^s \times \mathbb{E} \lhd \mathbb{F}, \mathbb{R}^s \times \mathbb{E} \lhd \mathbb{F})$$ covering the tame sc-retract $\cO$. Let $\hat{K} := \pi^{-1}(\hat{\cO})$ where $\pi : K^1 \rightarrow \cO^1$ is the $1$-shifted local bundle model. Set
$$K' := (g \lhd id_{\mathbb{F}^1})(\hat{K}).$$

Then, the tuple
\begin{equation} \label{eq:normalformslicebundleretract}
(K',\mathbb{R}^n \times [0,\infty)^s \times \mathbb{K}^1 \lhd \mathbb{F}^1,\mathbb{R}^n \times \mathbb{R}^s \times \mathbb{K}^1 \lhd \mathbb{F}^1)
\end{equation}
is a $\mathbb{R}^n$-sliced bundle retract covering the $\mathbb{R}^n$-sliced sc-retract \eqref{eq:normalformsliceretract}. In particular, the induced tame strong bundle retract $\tilde{K}'$ (see Lemma~\ref{lem:inducedtamebundleretraction}) covers the induced tame sc-retract $\tilde{\cO}'$.

Moreover,
\begin{equation} \label{eq:slicecoordinatesbundlestatement}
(g \lhd id_{\mathbb{F}^1})(\pi^{-1}(f^{-1}(0) \cap \hat{\cO})) = \tilde{K}'
\end{equation}
holds, so we view $(g \lhd id_{\mathbb{F}^1})$ as providing {\bf slice bundle coordinates $\tilde{K}' \subset K'$ for $\pi$ around $x$ with respect to $(f^{-1}(0) \cap \cO^1) \subset \cO^1$}.\\

\item Consider, in addition, a tame sc-Fredholm germ $\s : \cO \rightarrow K$ (Definition~\ref{dfn:tamescFredgerm}). Assume that $x = 0$ and that $\s$ has the following property: There exists a choice of sc-Banach space $\mathbb{W}$ and linear sc-isomorphism $\overline{\psi} : \mathbb{E} \rightarrow \mathbb{R}^{k-s} \times \mathbb{W}$ satisfying the conditions in Definition~\ref{dfn:tamescFredgerm} of tame sc-Fredholm germ for $\s$ such that, in addition, there exists a sc-complement $L$ of $\ker D_xf_{\partial}$ in $T_x\cO_{\partial}$ satisfying
\begin{equation} \label{eq:thespecialproperty}
\overline{\psi}(L) \subset (\{0\}^{k-s} \times \mathbb{W}).
\end{equation}

Then, the section
$$\s' := (g \lhd id_{\mathbb{F}^1}) \circ \s \circ g^{-1}: \cO' \rightarrow K'$$
is a $\mathbb{R}^n$-sliced sc-Fredholm germ. In particular, the restriction to the induced tame local bundle model
$$\tilde{\s}' := \s'|_{\tilde{\cO}'} : \tilde{\cO}' \rightarrow \tilde{K}'$$
is a tame sc-Fredholm germ with index satisfying $ind(\tilde{\s}') = ind(\s) - n$.

We view $\tilde{\s}'$ as being in {\bf slice sc-Fredholm coordinates around $x$ with respect to $(f^{-1}(0) \cap \cO^1) \subset \cO^1$}.
\end{enumilistfiber}
\end{lem}
\begin{proof}
\ul{Proof of (I)}: Consider any open subset $U \subset [0,\infty)^s \times \mathbb{E}$, tame sc-retraction $r : U \rightarrow U$ onto $r(U) = \cO$, and the associated retraction on the boundary (see Lemma~\ref{lem:associatedboundaryretract}) denoted by $r_{\partial} : U_{\partial} \rightarrow U_{\partial}$ where $U_{\partial} = U \cap (\{0\}^s \times \mathbb{E}).$

Set
$$\mathbb{K} := \ker D_x(f_{\partial} \circ r_{\partial}) \subset T_xU_{\partial} = \mathbb{E}.$$
Let $L \subset E_{\infty}$ be a sc-complement of $\ker D_xf_{\partial}$ in $T_x\cO_{\partial}$, which exists by Lemma~\ref{lem:complementfinitecodimension} since the surjection $D_xf_{\partial} : T_x\cO_{\partial} \rightarrow \mathbb{R}^n$ induces an isomorphism $T_x\cO_{\partial}/\ker D_xf_{\partial} \cong \mathbb{R}^n$. Note that the restriction
\begin{equation} \label{eq:isoonL}
D_xf_{\partial}|_L : L \rightarrow \mathbb{R}^n
\end{equation}
is an isomorphism.

We claim that $L$ is a sc-complement of $\mathbb{K}$ in $\mathbb{E}$, i.e.
\begin{equation} \label{eq:complementpullsthroughretraction}
\mathbb{E} = \mathbb{K} \oplus L.
\end{equation}
First, note that the sc-splitting $\mathbb{E} = T_x\cO_{\partial} \oplus \ker D_xr_{\partial}$ holds since the map $D_xr_{\partial} : \mathbb{E} \rightarrow \mathbb{E}$ is a linear sc-retraction with image $T_x\cO_{\partial}$. Hence,
$$\mathbb{E} = (L \oplus \ker D_xf_{\partial}) \oplus \ker D_xr_{\partial}.$$
From this description of $\mathbb{E}$ together with the definition of $\mathbb{K}$, we claim that $\ker D_xf_{\partial} \oplus \ker D_xr_{\partial} = \mathbb{K}$ follows, which then implies \eqref{eq:complementpullsthroughretraction}. Indeed, write $\xi \in \mathbb{E}$ as $\xi = l + \eta + \nu$ for $l \in L, \eta \in \ker D_xf_{\partial},$ and $\nu \in \ker D_x r_{\partial}$. Note that $D_xr_{\partial}(\eta) = \eta$ and $D_xr_{\partial}(l) = l$ since $\eta,l \in T_x\cO_{\partial}$. Then if $\xi \in \mathbb{K}$ we have $0 = D_x(f_{\partial} \circ r_{\partial})(\xi) = D_xf_{\partial} (l) + D_xf_{\partial}(\eta)  = D_xf_{\partial} (l)$, which implies $l = 0$ by \eqref{eq:isoonL}, and hence $\xi = \eta + \nu \in \ker D_xf_{\partial} \oplus \ker D_xr_{\partial}$. To prove the reverse inclusion, observe that if $\xi = \eta + \nu$ then we have $D_x(f_{\partial} \circ r_{\partial})(\xi) = D_xf_{\partial}(\eta) = 0$ and so $\xi \in \mathbb{K}$. Hence \eqref{eq:complementpullsthroughretraction} holds.

We now prepare to apply the normal form of a sc-smooth local submersion (Lemma~\ref{lem:sclocalsubmersionnormalform}) to the map
$$f \circ r : U \rightarrow \mathbb{R}^n.$$
The restriction of $f \circ r$ to $U_{\partial} = U \cap (\{0\}^s \times \mathbb{E})$ is equal to $f_{\partial} \circ r_{\partial}$. So, by the hypothesized surjectivity of the map $D_xf_{\partial} : T_x\cO_{\partial} \rightarrow \mathbb{R}^n$ and the definition of tangent space $T_x\cO_{\partial} = D_xr_{\partial}(T_xU_{\partial})$, it follows that the map
$$F := D_x(f \circ r|_{U \cap \{0\}^s \times \mathbb{E}}) : \mathbb{E} \rightarrow \mathbb{R}^n$$
is surjective with kernel $\mathbb{K} = \ker F$. Hence Lemma~\ref{lem:sclocalsubmersionnormalform} applies to the map $f \circ r$ and the sc-complement $L$ of $\mathbb{K}$ in $\mathbb{E}$. The lemma provides an open neighborhood
$$\hat{\ul{U}} \subset U^1 \subset [0,\infty)^s \times \mathbb{E}^1$$
of $x$ in $U^1$, an open set
$$\ul{U}' \subset \mathbb{R}^n \times [0,\infty)^s \times \mathbb{K}^1,$$
and a sc-smooth diffeomorphism of the form
\begin{align*}
g : \hat{\ul{U}} &\rightarrow \ul{U}'\\
(v,e) &\mapsto (f \circ r(v,e),v,pr(e)),
\end{align*}
where $pr : \mathbb{E} = L \oplus \mathbb{K} \rightarrow \mathbb{K}$ is the projection along $L$, such that on every level $m \geq 0$ the map $g : \hat{\ul{U}}_m \rightarrow \ul{U}'_m$ is a $C^{m+1}$-diffeomorphism. In particular,
\begin{equation} \label{eq:isprojectionontorn}
f \circ r \circ g^{-1} : \ul{U}' \rightarrow \mathbb{R}^n \text{ is projection onto the } \mathbb{R}^n\text{-coordinate}
\end{equation}
and also $g$ preserves the $[0,\infty)^s$-coordinate.

We shrink the open sets $\hat{\ul{U}}$ and $\ul{U}'$ in the following way, so that the smaller open set $\hat{U} \subset \hat{\ul{U}}$ has the property $r(\hat{U}) \subset \hat{U}$, making $r|_{\hat{U}} : \hat{U} \rightarrow \hat{U}$ into a sc-retraction: Set
\begin{align*}
\hat{U} &:= r^{-1}(\hat{\ul{U}} \cap \cO) \cap \hat{\ul{U}},\\
U' &:= g(\hat{U}).
\end{align*}
The set $\hat{U}$ is open in $U^1$ because $\hat{\ul{U}} \cap \cO$ is open in $\cO^1$ and $r : U \rightarrow \cO$ is sc-continuous. We claim that $r(\hat{U}) \subset \hat{U}$, so $r$ restricts to a tame sc-retraction $r : \hat{U} \rightarrow \hat{U}$ onto the tame sc-retract
$$\hat{\cO} := r(\hat{U}).$$
Indeed, if $y \in \hat{U}$ then $r(r(y)) = r(y) \in \hat{\ul{U}} \cap \cO$, and so it follows that $r(y) \in \hat{U}$. Note that
$$U' \subset \mathbb{R}^n \times [0,\infty)^s \times \mathbb{K}^1$$
is open since $g$ is a sc-diffeomorphism, and moreover the restricted map
$$g : \hat{U} \rightarrow U'$$
is a sc-diffeomorphism.

The map
$$r' := g \circ r \circ g^{-1} : U' \rightarrow U'$$
is a sc-smooth retraction since $g$ is a sc-diffeomorphism and $r$ is a sc-smooth retraction. Indeed, $r \circ r = r$ implies that $r' \circ r' = r'.$ Set
$$\cO' := r'(U').$$
Notice that
$$\hat{\cO} = r(\hat{U}) = r \circ g^{-1}(U') = g^{-1} \circ r'(U') = g^{-1}(\cO')$$
holds, so $g$ restricts to a bijection
$$g : \hat{\cO} \rightarrow \cO'.$$
This restriction is a sc-smooth diffeomorphism, as required, since $$g \circ r : \hat{U} \rightarrow U'$$ is sc-smooth and $g^{-1} \circ r' : U' \rightarrow \hat{U}$ is sc-smooth. Indeed, the definition of a map between sc-retracts being sc-smooth is that these compositions with the sc-retractions are sc-smooth, where the domain and codomain are now open subsets of sc-Banach spaces and sc-smoothness is defined as usual.

We now prove that $r'$ is a $\mathbb{R}^n$-sliced sc-retraction, and so $\cO'$ is the required $\mathbb{R}^n$-sliced sc-retract \eqref{eq:normalformsliceretract}. We must show that $r$ is tame and satisfies the defining property \eqref{eq:sliceretraction} of a $\mathbb{R}^n$-sliced retraction, i.e.\ that $r$ preserves the $\mathbb{R}^n$-coordinate. To prove \eqref{eq:sliceretraction}, notice from the definitions and \eqref{eq:isprojectionontorn} that we have
$$(f \circ g^{-1}) \circ r' = f \circ r \circ g^{-1} = \pi_{\mathbb{R}^n}.$$
Since $r' \circ r' = r'$, we have
$$\pi_{\mathbb{R}^n} \circ r' = (f \circ g^{-1}) \circ r' \circ r' = (f \circ g^{-1}) \circ r' = \pi_{\mathbb{R}^n},$$
completing the proof of \eqref{eq:sliceretraction}.

We now verify that $r'$ is a tame sc-retraction. Definition~\ref{dfn:tameretraction}(1) for $r' = g \circ r \circ g^{-1}$ holds because it holds for $r$ and since $g$ preserves the $[0,\infty)^s$-coordinate. To verify Definition~\ref{dfn:tameretraction}(2) for $r'$, let $y' \in \cO'_{\infty}$ and set $y := g^{-1}(y') \in \hat{\cO}_{\infty}$. Since $r$ is a tame sc-retraction, there exists a sc-subspace $A \subset \mathbb{R}^s \times \mathbb{E}^1$ such that $\mathbb{R}^s \times \mathbb{E}^1 = T_y\hat{\cO} \oplus A$ and $A \subset (\mathbb{R}^s \times \mathbb{E}^1)_y$. We claim that $A' := D_yg(A)$ is the required sc-subspace of $\mathbb{R}^n \times \mathbb{R}^s \times \mathbb{K}^1$. First, since $g$ is the identity on $[0,\infty)^s$, the tangent map $D_yg$ is the identity on $\mathbb{R}^s$, and thus $A' \subset (\mathbb{R}^n \times \mathbb{R}^s \times \mathbb{K}^1)_{y'}$ by \eqref{eq:Exstandard}. Moreover, by definition of $r'$ and the tangent space of a retract we have $T_{y'}\cO' = D_{y'}r'(T_{y'}U') = D_yg(D_yr(T_y\hat{U})) = D_yg(T_y\hat{\cO})$. Hence, since $D_yg$ is a linear sc-isomorphism, we have $$\mathbb{R}^n \times \mathbb{R}^s \times \mathbb{K}^1 = D_yg(\mathbb{R}^s \times \mathbb{E}^1) = D_yg(T_y\hat{\cO} \oplus A) = T_{y'}\cO' \oplus A',$$ as required. This completes the verification that $r'$ is a $\mathbb{R}^n$-sliced sc-retraction and hence \eqref{eq:normalformsliceretract} is indeed a $\mathbb{R}^n$-sliced sc-retract.

To complete the proof of (I), it remains to show that \eqref{eq:slicecoordinatesbasestatement} holds. First, recall from above that $g(\hat{\cO}) = \cO'$.  Let $y \in g(f^{-1}(0) \cap \hat{\cO}).$ Since $g^{-1}(y) \in \hat{\cO}$ is in the image of the retraction $r$, it follows that $0 = f \circ g^{-1}(y) = f \circ r \circ g^{-1}(y)$, and hence \eqref{eq:isprojectionontorn} implies $$y \in \{0\}^n \times [0,\infty)^s \times \mathbb{K}^1,$$ proving the forwards inclusion in \eqref{eq:slicecoordinatesbasestatement}. For the reverse inclusion, let $y \in \cO' \cap (\{0\}^n \times [0,\infty)^s \times \mathbb{K}^1)$. Then, again by \eqref{eq:isprojectionontorn}, we have $$f \circ g^{-1}(y) = f \circ r \circ g^{-1}(y) = \pi_{\mathbb{R}^n}(y) = 0.$$ This proves the reverse inclusion and hence the equality \eqref{eq:slicecoordinatesbasestatement}.
\\

\ul{Proof of (II)}: After shrinking the original open set $U \subset [0,\infty)^s \times \mathbb{E}$, we can assume that there exists a tame bundle retraction
\begin{align*}
R : U \lhd \mathbb{F} &\rightarrow U \lhd \mathbb{F}\\
 (y,\xi) &\mapsto (r(y),\Gamma(y,\xi))
\end{align*}
onto $R(U \lhd \mathbb{F}) = K$ covering $r$. Since $r$ restricts to a retraction $$r : \hat{U} \rightarrow \hat{U}$$ onto $$r(\hat{U}) = \hat{\cO}$$ as discussed above, it follows that $R$ restricts to a tame bundle retraction $$R : \hat{U} \lhd \mathbb{F}^1 \rightarrow \hat{U} \lhd \mathbb{F}^1$$ with image the tame bundle retract
$$\hat{K} = R(\hat{U} \lhd \mathbb{F}^1) = \pi^{-1}(\hat{\cO}),$$
where $\pi : K^1 \rightarrow \cO^1$ is the $1$-shifted tame local bundle model.

Set
\begin{align*}
\G' : U' \lhd \mathbb{F}^1 &\rightarrow \mathbb{F}^1\\
(y',\xi) &\mapsto \Gamma(g^{-1}(y'),\xi)
\end{align*}
We claim that the map
\begin{align*}
R' := U' \lhd \mathbb{F}^1 &\rightarrow U' \lhd \mathbb{F}^1\\
(y',\xi) &\mapsto (r'(y'),\G'(y',\xi))
\end{align*}
is a $\mathbb{R}^n$-sliced bundle retraction onto
$$K' := (g \lhd id_{\mathbb{F}^1})(\hat{K}),$$
proving that \eqref{eq:normalformslicebundleretract} is a $\mathbb{R}^n$-sliced bundle retract covering the $\mathbb{R}^n$-sliced sc-retract \eqref{eq:normalformsliceretract}. It is then immediate from Lemma~\ref{lem:inducedtamebundleretraction} that the induced bundle retract $\tilde{K}'$ covers the induced sc-retract $\tilde{\cO}'$, as claimed. Since $r' = g \circ r \circ g^{-1}$, we can write $R'$ as a composition of strong bundle maps
$$R' = (g \lhd id_{\mathbb{F}^1}) \circ R \circ (g^{-1} \lhd id_{\mathbb{F}^1}),$$
and hence $R'$ is a strong bundle map. In addition, $R' \circ R' = R'$ follows from $R \circ R = R$. Thus $R'$ is a bundle retraction covering the $\mathbb{R}^n$-sliced sc-retraction $r'$, which implies that $R'$ is a $\mathbb{R}^n$-sliced bundle retraction, as desired. Moreover, by definition of $K'$ we have $$R'(U' \lhd \mathbb{F}^1) = (g \lhd id_{\mathbb{F}^1}) \circ R(\hat{U} \lhd \mathbb{F}^1) = (g \lhd id_{\mathbb{F}^1})(\hat{K}) = K'.$$ Thus $K'$ is indeed the desired $\mathbb{R}^n$-sliced bundle retract \eqref{eq:normalformslicebundleretract}.

We now prove \eqref{eq:slicecoordinatesbundlestatement}. Let $\pi' : K' \rightarrow \cO'$ denote the $\mathbb{R}^n$-sliced local bundle model, which restricts to the induced tame local bundle model $\tilde{\pi} = \pi'|_{\tilde{K}'} : \tilde{K}' \rightarrow \tilde{\cO'}$ (see \eqref{eq:inducedbundlemodel}). Here $\tilde{K}' = \pi'^{-1}(\tilde{\cO}')$ is the tame bundle retract induced by the $\mathbb{R}^n$-sliced bundle retract $K'$ (see \eqref{eq:inducedscretractbundle}). To prove \eqref{eq:slicecoordinatesbundlestatement}, let $y \in f^{-1}(0) \cap \hat{\cO}$ and let $\xi \in \pi^{-1}(y).$ Then, by \eqref{eq:slicecoordinatesbasestatement}, we have $g(y) \in \tilde{\cO}'$, and hence $(g \lhd id_{\mathbb{F}^1})(\xi) \in \pi'^{-1}(\tilde{\cO}') = \tilde{K}'$, proving the forwards inclusion. To see the reverse inclusion, let $\xi \in \tilde{K}'.$ Set $y = \pi'(\xi) \in \tilde{\cO}'$. Then, again by \eqref{eq:slicecoordinatesbasestatement}, we have $g^{-1}(y) \in f^{-1}(0) \cap \hat{\cO}$ and hence $(g^{-1} \lhd id_{\mathbb{F}^1})(\xi) \in \pi^{-1}(f^{-1}(0) \cap \hat{\cO}).$ This proves \eqref{eq:slicecoordinatesbundlestatement}, completing the proof of the statements in (II).
\\

\ul{Proof of (III)}: To prove that the section
$$\s' : = (g \lhd id_{\mathbb{F}^1}) \circ \s \circ g^{-1} : \cO' \rightarrow K'$$
is a $\mathbb{R}^n$-sliced sc-Fredholm germ (Definition~\ref{dfn:slicescFredholmgerm}), we must first show that it is a local sc-Fredholm germ (Definition~\ref{dfn:localscFredgerm}). Since the given section $\s$ is a tame sc-Fredholm germ (Definition~\ref{dfn:tamescFredgerm}) by hypothesis, by shrinking the original open set $U \subset [0,\infty)^s \times \mathbb{E}$ we can assume that the tame sc-retraction $r : U \rightarrow U$ satisfies the properties guaranteed by the definition of a tame sc-Fredholm germ. Moreover, the $1$-shifted section $\s|_{\cO^1} : \cO^1 \rightarrow K^1$ is also a tame sc-Fredholm germ (see Remark~\ref{rmk:levelshift}), and since $\hat{U} \subset U^1$ is an open subset of the $1$-shift of $U$, the retraction $r : \hat{U} \rightarrow \hat{U}$ onto $r(\hat{U}) = \hat{\cO}$ and the tame bundle retraction $R$ covering $r$ also satisfy the properties required in the definition of tame sc-Fredholm germ for $\s|_{\hat{\cO}} : \hat{\cO} \rightarrow \hat{K}$. We explain these properties as we use them.

There exists a sc-germ of neighborhoods  $\hat{U}([0,\infty)^s \times \mathbb{E}^1,0)$ around $0$ such that the $0$-level $\hat{U}(0)$ in the germ is equal to $\hat{U}$ and such that the composition $\mathbold{\s} \circ r : \hat{U}([0,\infty)^s \times \mathbb{E}^1,0) \rightarrow \mathbb{F}^1$ possesses a filling (see Definition~\ref{dfn:localscFredgerm}(b))
$$h : \hat{U}([0,\infty)^s \times \mathbb{E}^1,0) \rightarrow \hat{U}([0,\infty)^s \times \mathbb{E}^1,0) \lhd \mathbb{F}^1,$$
where $\mathbold{\s}$ means the principal part in $\mathbb{F}^1$ of $\s$. By hypothesis and the construction of $\hat{\cO}$, we have $x = 0 \in f^{-1}(0) \cap \hat{\cO}$, and so it follows from the form of $g$ above that
\begin{equation} \label{eq:g0equals0}
g(0) = 0 \in \cO'.
\end{equation}
The sc-germ of neighborhoods around $0$ given by
$$U'(\mathbb{R}^n \times [0,\infty)^s \times \mathbb{K}^1,0) := g \big (\hat{U}([0,\infty)^s \times \mathbb{E}^1,0) \big )$$
has $0$-level $U'(0)$ equal to the set open set $U' = g(\hat{U})$ above. From now on, we abbreviate the sc-germs of neighborhoods simply by their $0$-level, i.e.\ by $\hat{U}$ and $U'$.

We claim that the map
$$h' := (g \lhd id_{\mathbb{F}^1}) \circ h \circ g^{-1} : U' \rightarrow U' \lhd \mathbb{F}^1.$$
is a filling of $\mathbold{\s}' \circ r' : U' \rightarrow \mathbb{F}^1$. The required properties Definition~\ref{dfn:localscFredgerm}(b).(i)-(iii) of a filling hold for $h'$ because they hold for $h$, as we explain now. To verify $(i)$ for $h'$, let $y' \in \cO'$. Then $y := g^{-1}(y') \in \hat{\cO}$, so by $(i)$ for $h$ we have $\s(y) = h(y)$, which implies $\s'(y') = h'(y')$ by definition of $\s'$ and $h'$, as required. To verify (ii), let $y' \in U'$ and assume ${\bf h}'(y') = \Gamma'(r'(y'),{\bf h}'(y'))$. Then, setting $y := g^{-1}(y') \in \hat{U}$, we have ${\bf h}(y) = {\bf h}'(y') = \Gamma(r(y),{\bf h}(y))$. Hence $y \in \hat{\cO}$ by property $(ii)$ for $h$, so $y' \in g(\hat{\cO}) = \cO'$, as required. We now verify property (iii) for $h'$. Property $(iii)$ for $h$ says that the linearization $D_0A$ at $0$ of the map
\begin{align*}
A : \hat{U} &\rightarrow \mathbb{F}^1\\
y &\mapsto (id_{\mathbb{F}^1} - \Gamma(r(y),\cdot)) \circ {\bf h}(y)
\end{align*}
restricts to a linear sc-isomorphism $\ker D_0r \rightarrow \ker \Gamma(0,\cdot).$ We must show that the map
\begin{align*}
A' : U' &\rightarrow \mathbb{F}^1\\
y' &\mapsto (id_{\mathbb{F}^1} - \Gamma'(r'(y'),\cdot)) \circ {\bf h}'(y')
\end{align*}
restricts to a linear sc-isomorphism $\ker D_0r' \rightarrow \ker \G'(0,\cdot).$ Note that, by definition of $\G'$ and by \eqref{eq:g0equals0}, we have $$\ker \G'(0,\cdot) = \ker \G(g^{-1}(0),\cdot) = \ker \G(0,\cdot).$$
Moreover, from the definition of $h'$ we have
$${\bf h}'(y') = {\bf h} \circ g^{-1}(y'),$$
and so by definition of $r'$ we have
$$A' = A \circ g^{-1} : U' \rightarrow \mathbb{F}^1.$$
Hence $D_0A'$ restricts to the desired isomorphism since $D_0g^{-1}$ restricts to an isomorphism $\ker D_0r' \rightarrow \ker D_0r$ by definition of $r'$ and \eqref{eq:g0equals0}. Hence $h'$ is indeed a filling of $\mathbold{\s}' \circ r'$, as claimed.

To verify that $\s'$ is a local sc-Fredholm germ, it remains to check the properties in Definition~\ref{dfn:localscFredgerm}(c). Since $$\s : \hat{\cO} \rightarrow \hat{K}$$ is a tame sc-Fredholm germ, we have the following: There exists a sc$^+$-section $$\mathfrak{s} : \hat{U} \rightarrow \hat{U} \lhd \mathbb{F}^1$$ satisfying $\mathfrak{s}(0) = h(0)$, a sc-Banach space $\mathbb{W}$, a sc-germ of neighborhoods
$$\hat{U}' \text{ around } 0 \text{ in } [0,\infty)^s \times \mathbb{R}^{k-s} \times \mathbb{W}$$
for some $k \geq s$, and a strong bundle isomorphism
$$\Psi : \hat{U} \lhd \mathbb{F}^1 \rightarrow \hat{U}' \lhd \mathbb{R}^{k'} \times \mathbb{W},$$
for some $k' \geq 0$, covering a linear sc-isomorphism 
$$\psi : \hat{U} \rightarrow \hat{U}'$$
of the form
$$\psi = id_{[0,\infty)^s} \times \overline{\psi} : [0,\infty)^s \times \mathbb{E}^1 \rightarrow [0,\infty)^s \times \mathbb{R}^{k-s} \times \mathbb{W}$$
for some linear sc-isomorphism
$$\overline{\psi} : \mathbb{E}^1 \rightarrow \mathbb{R}^{k-s} \times \mathbb{W},$$
such that the principal part of the section
\begin{equation} \label{eq:isbasicgermconstraintproof}
b := \Psi \circ (h - \mathfrak{s}) \circ \psi^{-1} : \hat{U}' \rightarrow \hat{U}' \lhd \mathbb{R}^{k'} \times \mathbb{W}
\end{equation}
is a basic germ.

By the hypothesis \eqref{eq:thespecialproperty}, we can assume that the subspace $L \subset E_{\infty}$ satisfies $\overline{\psi}(L) \subset \{0\}^{k-s} \times \mathbb{W}$.  We consider $\overline{\psi}(L)$ as a sc-subspace of $\mathbb{W}$ and let $\tilde{\mathbb{W}}$ be any sc-complement $\tilde{\mathbb{W}} \oplus \overline{\psi}(L) = \mathbb{W}$ which exists by \cite[Prop.~1.1]{HWZbook} since $L$ is finite dimensional. After fixing any linear isomorphism $\overline{\psi}(L) \rightarrow \mathbb{R}^n$, we obtain a linear sc-isomorphism
$$\t : \mathbb{W} = \tilde{\mathbb{W}} \oplus \overline{\psi}(L) \rightarrow \tilde{\mathbb{W}} \times \mathbb{R}^n.$$
Define
\begin{align*}
\psi' &:= (id_{[0,\infty)^s \times \mathbb{R}^{k-s}} \times \tau) \circ \psi  \circ g^{-1} : U' \rightarrow [0,\infty)^s \times \mathbb{R}^{k-s} \times \tilde{\mathbb{W}} \times \mathbb{R}^n,\\
\mathfrak{s}' &:= (g \lhd id_{\mathbb{F}^1}) \circ \mathfrak{s} \circ g^{-1} : U' \rightarrow U' \lhd \mathbb{F}^1,\\
\hat{U}'' &:= \psi'(U') = (id_{[0,\infty)^s \times \mathbb{R}^{k-s}} \times \tau)(\hat{U}'),\\
\Psi' &:= ((id_{[0,\infty)^s \times \mathbb{R}^{k-s}} \times \tau) \lhd (id_{\mathbb{R}^{k'}} \times \tau))\\
&\,\,\,\, \circ \Psi \circ (g^{-1} \lhd id_{\mathbb{F}^1}) : U' \lhd \mathbb{F}^1 \rightarrow \hat{U}'' \lhd \mathbb{R}^{k'} \times \tilde{\mathbb{W}} \times \mathbb{R}^n.
\end{align*}

We claim that the principal part of the section
\begin{equation} \label{eq:rnslicedgermproof}
b' := \Psi' \circ (h' - \mathfrak{s}') \circ \psi'^{-1} : \hat{U}'' \rightarrow \hat{U}'' \lhd \mathbb{R}^{k'} \times (\tilde{\mathbb{W}} \times \mathbb{R}^n)
\end{equation}
is a basic germ, where here $\tilde{\mathbb{W}} \times \mathbb{R}^n$ plays the role of the space $\mathbb{W}$ in Definition~\ref{dfn:localscFredgerm}(c). By hypothesis, the principal part ${\bf b}$ of the map \eqref{eq:isbasicgermconstraintproof} is a basic germ, which means that, letting
$$P : \mathbb{R}^{k'} \times \mathbb{W} \rightarrow \mathbb{W}$$
denote projection onto $\mathbb{W}$, we have, for $a \in [0,\infty)^s \times \mathbb{R}^{k-s}$ and $w \in \mathbb{W}$,
$$P \circ {\bf b}(a,w) = w - B(a,w),$$
where $B : \hat{U}' \rightarrow \mathbb{W}$ is a sc-smooth germ satisfying $B(0) = 0$ and the contraction property \eqref{eq:basicgerm}. Let
$$P' : \mathbb{R}^{k'} \times (\tilde{\mathbb{W}} \times \mathbb{R}^n) \rightarrow \tilde{\mathbb{W}} \times \mathbb{R}^n$$
be projection onto $\tilde{\mathbb{W}} \times \mathbb{R}^n$. Define
\begin{align*}
B' := \t \circ B \circ (id_{[0,\infty)^s \times \mathbb{R}^{k-s}} \times \tau)^{-1} : \hat{U}'' &\rightarrow \tilde{\mathbb{W}} \times \mathbb{R}^n.
\end{align*}
Notice from the definitions that we have
$$P' = \t \circ P \times (id_{\mathbb{R}^{k'}} \times \tau)^{-1}.$$
Then, for $a \in [0,\infty)^s \times \mathbb{R}^{k-s}$ and $z \in \tilde{\mathbb{W}} \times \mathbb{R}^n$ such that $(a,z) \in \hat{U}''$, we compute
\begin{align*}
P' \circ {\bf b}'(a,z) &= \tau \circ P \circ (id_{\mathbb{R}^{k'}} \times \tau)^{-1} \circ \pi_{\mathbb{R}^{k'} \times \tilde{\mathbb{W}} \times \mathbb{R}^n}\\
&\,\,\,\,\,\,\,\,\,\, \circ \Psi' \circ (h'-\mathfrak{s}') \circ \psi'^{-1}(a,z)\\
&= \tau \circ P \circ \pi_{\mathbb{R}^{k'} \times \mathbb{W}} \circ \Psi \circ (h-\mathfrak{s}) \circ \psi^{-1}(a,\t^{-1}(z))\\
&=\tau \circ P \circ {\bf b}(a,\t^{-1}(z))\\
&= \tau(\t^{-1}(z) - B(a,\t^{-1}(z)))\\
&= z - B'(a,z),
\end{align*}
so to prove that ${\bf b}'$ is a basic germ, it remains to verify that $B'$ satisfies the contraction property \eqref{eq:basicgerm}. Since $\t$ is a linear sc-isomorphism, for every integer $m \geq 0$ there exist constants $c_m, C_m > 0$ such that, for all $w \in \mathbb{W}$, we have
$$||\t(w)||_m \leq C_m \cdot ||w||_m,$$
and for all $(a,z) \in ([0,\infty)^s \times \mathbb{R}^{k-s}) \times (\tilde{\mathbb{W}} \times \mathbb{R}^n)$ we have
$$||(id_{\mathbb{R}^k} \times \t)^{-1}(a,z)||_m \leq c_m \cdot ||(a,z)||_m.$$
Fix $\epsilon > 0$ and $m \geq 0$. Then since $B$ satisfies the contraction property \eqref{eq:basicgerm}, there exists $\delta_m' > 0$ that suffices for $B$ and $\epsilon_m' := \epsilon / (C_m \cdot c_m)$. We claim that $\delta_m := \delta_m' / c_m$ suffices for $B'$ and $\epsilon$. Indeed, for any elements $a \in [0,\infty)^s \times \mathbb{R}^{k-s}$ and $z,z' \in \tilde{\mathbb{W}} \times \mathbb{R}^n$ satisfying $$||(a,z)||_m,||(a,z')||_m < \delta_m,$$ we have $||(a,\t^{-1}(z))||_m,||(a,\t^{-1}(z'))||_m < c_m \cdot \delta_m = \delta'_m.$ So, we compute
\begin{align*}
||B'(a,z) - B'(a,z')||_m &= ||\t(B(a,\t^{-1}(z)) - B(a,\t^{-1}(z')))||_m\\
&\leq C_m \cdot ||B(a,\t^{-1}(z)) - B(a,\t^{-1}(z'))||_m\\
&\leq C_m \cdot \epsilon'_m \cdot ||\t^{-1}(z) - \t^{-1}(z')||_m\\
&= \epsilon / c_m \cdot ||\t^{-1}(z - z')||_m\\
&\leq \epsilon \cdot ||z - z'||_m,
\end{align*}
as required. This completes the proof that ${\bf b}'$ is a basic germ, and hence that $\s'$ is a local sc-Fredholm germ.

We claim that the additional conditions required for $\s'$ to be a $\mathbb{R}^n$-sliced sc-Fredholm germ are satisfied by construction. We must check that $\psi' : U' \rightarrow \hat{U}''$ is of the required form \eqref{eq:slicedgermform}. Consider the projection
$$T : \mathbb{W} = \tilde{\mathbb{W}} \oplus \overline{\psi}(L) \rightarrow \tilde{\mathbb{W}}$$
along $\overline{\psi}(L) \subset \{0\}^{k-s} \times \mathbb{W}$. We have $L \oplus \mathbb{K}^1 = \mathbb{E}^1$ by \eqref{eq:complementpullsthroughretraction} and since $L \subset E_{\infty}$. Hence the map
$$\overline{\psi}' := (id_{\mathbb{R}^{k-s}} \times T) \circ \overline{\psi}|_{\{0\} \oplus \mathbb{K}^1} : \mathbb{K}^1 \rightarrow \mathbb{R}^{k-s} \times \tilde{\mathbb{W}}$$
is a linear sc-isomorphism. Let $(p,v,\k) \in U' \subset \mathbb{R}^n \times [0,\infty)^s \times \mathbb{K}^1$. By the form of the map $g$ above, we can write
$$g^{-1}(p,v,\k) =: (v,e_{p,v,\k}) \in [0,\infty)^s \times \mathbb{E}^1$$
for some $e_{p,v,\k} \in \mathbb{E}^1$. Moreover, we have $pr(e_{p,v,\k}) = \k$, where $pr$ is the projection $pr : \mathbb{E}^1 = L \oplus \mathbb{K}^1 \rightarrow \mathbb{K}^1$ along $L$. In particular, we have $e_{p,v,\k} - \k \in L$, from which we obtain a map to $\{0\} \times \mathbb{R}^n \subset \tilde{\mathbb{W}} \times \mathbb{R}^n$ defined by
\begin{align*}
\lambda : U' &\rightarrow \mathbb{R}^n\\
(p,v,\k) &\mapsto \t(\overline{\psi}(e_{p,v,\k} - \k)).
\end{align*}
This map is $C^1$ on every level $(U')_m \rightarrow \mathbb{R}^n$ for $m \geq 0$ because $g$ is a $C^1$-diffeomorphism on every level and the linear sc-isomorphisms $\t$ and $\overline{\psi}$ are $C^{\infty}$ on every level. 
Then we compute
\begin{align*}
\psi'(p,v,\k) &= (id_{[0,\infty)^s \times \mathbb{R}^{k-s}} \times \t) \circ \psi \circ (v,e_{p,v,\k})\\
&= (id_{[0,\infty)^s \times \mathbb{R}^{k-s}} \times \t) \circ (id_{[0,\infty)^s} \times \overline{\psi}) \circ (v,e_{p,v,\k})\\
&= (v,(id_{\mathbb{R}^{k-s}} \times \t) \circ \overline{\psi}(e_{p,v,\k}))\\
&= (v,(id_{\mathbb{R}^{k-s}} \times \t) (\overline{\psi}(\k) + \overline{\psi}(e_{p,v,\k} - \k)))\\
&= (v,(id_{\mathbb{R}^{k-s}} \times T) (\overline{\psi}(\k)), \t(\overline{\psi}(e_{p,v,\k} - \k)))\\
&= (v, \overline{\psi}'(\k),\lambda(p,v,\k)).
\end{align*}
This proves that $\s'$ is a $\mathbb{R}^n$-sliced sc-Fredholm germ. Hence it follows from Lemma~\ref{lem:inducedlocalscFredholmgerm} that the restriction $\tilde{\s}' = \s'|_{\tilde{\cO}'} : \tilde{\cO}' \rightarrow \tilde{K}'$ is a tame sc-Fredholm germ with sc-Fredholm index $ind(\tilde{\s}') = ind(\s') - n$. Moreover, by definition \eqref{eq:fredholmindexdef} of index and by the forms of the basic germs \eqref{eq:isbasicgermconstraintproof} and \eqref{eq:rnslicedgermproof}, we have $ind(\s') = ind(\s)$. Hence $ind(\tilde{\s}') = ind(\s) - n$. This completes the proof of the lemma.
\end{proof}

\begin{rmk}
We note some additional properties of the objects constructed in the proof of Lemma~\ref{lem:slicecoordinates}. These properties are required for the construction of quotients of polyfolds by group actions in \cite{ZQuotient}.

The sc-Banach space $\mathbb{K}$ is given by
$$\mathbb{K} = \ker D_x(f_{\partial} \circ r_{\partial}) \subset \mathbb{E},$$
where $r_{\partial} : U \cap (\{0\} \times \mathbb{E}) \rightarrow U \cap (\{0\} \times \mathbb{E})$ is the boundary sc-retraction associated to a tame sc-retraction $r: U \rightarrow U$ onto $r(U) = \cO$ from Lemma~\ref{lem:associatedboundaryretract}. There is a sc-complement $L$ of $\ker D_xf_{\partial}$ in $T_x\cO_{\partial}$, and moreover $L$ is a sc-complement of $\mathbb{K}$ in $\mathbb{E}$. The change of coordinates $g : [0,\infty)^s \times \mathbb{E}^1 \rightarrow \mathbb{R}^n \times [0,\infty)^s \times \mathbb{K}^1$, which is defined in a neighborhood of $x$, is of the form
$$g(v,e) = (f \circ r(v,e),v,pr(e))$$
where $pr : \mathbb{E} = L \oplus \mathbb{K} \rightarrow \mathbb{K}$ is projection along $L$. In particular, in the case $x = 0$, the sc-differential of $g$ satisfies $$D_0g(\{0\} \times L) = \mathbb{R}^n \times \{0\} \times \{0\}$$
and $$D_0g(\{0\} \times \mathbb{K}^1) = \{0\} \times \{0\} \times \mathbb{K}^1.$$
\end{rmk}

\section{Slicing tame sc-Fredholm sections with transverse constraints} \label{sec:globalslicing}

The theory of sc-Fredholm sections $\s : \cB \rightarrow \cE$ (Definition~\ref{dfn:scfredholmsection}) of tame strong bundles $\rho : \cE \rightarrow \cB$ (Definition~\ref{dfn:tamebundle}) over tame $M$-polyfolds $\cB$ (Definition~\ref{dfn:tameMpolyfold}) is developed in \cite{HWZbook}. In this section, we introduce the new stronger notion of a tame sc-Fredholm section (Definition~\ref{dfn:tamescfredholmsection}), which is the same as a standard sc-Fredholm section with the stronger condition that the local sc-Fredholm germs (Definition~\ref{dfn:localscFredgerm}) can be chosen to be the tame sc-Fredholm germs introduced in Definition~\ref{dfn:tamescFredgerm}: roughly, we require that the change of coordinates that brings the local fillers into basic germ form can be chosen to be a \emph{linear sc-isomorphism} on the base. We also introduce the new notion of slices $\tilde{\cB} \subset \cB$ (Definition~\ref{dfn:globalslice}) of sc-Fredholm sections $\s$, which are subspaces of $\cB$ in which the local sc-Fredholm germs are $\mathbb{R}^n$-sliced (Definition~\ref{dfn:slicescFredholmgerm}): roughly, the dependence of the change of coordinates to basic germ form on the normal directions to $\tilde{\cB}$ in $\cB$ splits off into a $\mathbb{R}^n$-factor in the codomain.

Tame sc-Fredholm sections and slices are the main objects of interest in this paper. In this section, we prove our main Theorems~\ref{thm:slicestructure}, \ref{thm:globalimplicit}. Roughly, given a tame sc-Fredholm section $\s : \cB \rightarrow \cE$ and a sc-smooth map $f : \cB \rightarrow Y$ to a finite dimensional manifold $Y$ that is $\s$-compatibly transverse (Definition~\ref{dfn:transverse}) to a submanifold $N \subset Y$, then $f^{-1}(N)$ is a slice of $\s$. Moreover, given a slice $\tilde{\cB} \subset \cB$ of a sc-Fredholm section $\s : \cB \rightarrow \cE$, the restriction $\s|_{\tilde{\cB}}$ is tame sc-Fredholm. The generalizations of the concepts and theorems in this section to the ep-groupoid case, which are required to handle isotropy in applications, are developed in Section~\ref{sec:ep-groupoids}.

In Section~\ref{sec:evaluationmaps}, we explain why the Cauchy-Riemann section $$\overline{\partial}_J : \cB \rightarrow \cE$$ is tame sc-Fredholm and why evaluation maps $$ev : \cB \rightarrow Y$$ at marked points are $\overline{\partial}_J$-compatibly transverse to every submanifold $N \subset Y$.

We begin with a review of standard $M$-polyfold notions.

\begin{dfn} \label{dfn:tameMpolyfold} \cite[Defs.~2.6, 2.19]{HWZbook}\text{}
\begin{itemlist}
\item A {\bf tame $\mathbold{M}$-polyfold chart} on a topological space $\cB$ is a tuple
$$
(V,\varphi,(\cO,C,\mathbb{E})),
$$
where $V \subset \cB$ is an open set,  $\varphi : V \rightarrow \cO$ is a homeomorphism, and $(\cO,C,\mathbb{E})$ is a tame sc-retract (Definition~\ref{dfn:tameretraction}).
\item A {\bf tame $\mathbold{M}$-polyfold} $\cB$ is a Hausdorff paracompact topological space together with an equivalence class of tame sc-smooth atlases. A {\bf tame sc-smooth atlas} consists of a cover of $\cB$ by tame $M$-polyfold charts whose transition maps are sc-smooth. Atlases are equivalent if and only if their union is an atlas.
\end{itemlist}
\end{dfn}

A tame $M$-polyfold $\cB$ carries a filtration by a sequence of topological spaces
$$\cB = \big ( \cB_0 \hookleftarrow \cB_1 \hookleftarrow \cdots \big )$$
where the maps $\cB_m \rightarrow \cB_{m+1}$ are continuous injections with dense image. This filtration is provided locally by the sc-structure on the sc-retracts in charts, and it is independent of the choice of charts due to sc-continuity of the transition maps. The smooth points $\cB_{\infty}$ of $\cB$ are the subset
$$\cB_{\infty} := \cap_{m \geq 0} \cB_m.$$
By forgetting about the first $m \geq 0$ levels of $\cB$, we obtain the \emph{$m$-shift} $\cB^m$ of $\cB$, which is the $M$-polyfold
$$\cB^m := \big ( \cB_m \hookleftarrow \cB_{m+1} \hookleftarrow \cdots \big )$$
with charts given by restricting charts on $\cB$ to the $m$-shift $$\varphi : V_m = V \cap \cB_m \rightarrow \cO_m = \cO \cap E_m.$$
Notice that a $m$-shift does not affect the smooth points, i.e.
$$\cB_{\infty} = \cB^m_{\infty}.$$

The \emph{tangent space} $T_x\cB$ at a point $x \in \cB_1$ is defined as the set of equivalence classes of tangent vectors in all charts around $x$, with the equivalence given by the tangent map of the transition map between charts, which is a linear isomorphism on every tangent space; see \cite[Def.~2.11]{HWZbook} for a precise treatment. The result is a Banach space structure on $T_x\cB$ and, for any chart $(V,\varphi,(\cO,C,\mathbb{E}))$ containing $x \in V$, a tangent map $D_x\varphi : T_xV = T_x\cB \rightarrow T_{\varphi(x)}\cO$ which is a bounded linear isomorphism. If $x \in \cB_{\infty}$ is a smooth point then $\varphi(x) \in \cO_{\infty}$ is a smooth point and $T_x\cB$ inherits the structure of a sc-Banach space from the sc-Banach space $T_{\varphi(x)}\cO$.

The \emph{reduced tangent space} $T_x^R\cB$ \cite[Def.~2.20]{HWZbook} is defined only for smooth points $x \in \cB_{\infty}$ as the subspace of $T_x\cB$ such that 
\begin{equation} \label{eq:defglobalredcuedtangent}
D_x\varphi(T^R_x\cB) = T_{\varphi(x)}^R\cO,
\end{equation}
where $T_{\varphi(x)}^R\cO$ is the reduced tangent space to a sc-retract \eqref{eq:reducedtangentspacedef}. Note that, since the reduced tangent space to a sc-retract is invariant under sc-diffeomorphisms of sc-retracts only at smooth points, this global notion of reduced tangent space $T^R_x\cB$ is well-defined only for smooth points $x \in \cB_{\infty}$.

The \emph{degeneracy index} $d_{\cB} : \cB \rightarrow \mathbb{N}_0$ is defined for $x \in \cB$ as the minimum over all charts $(V,\varphi,(\cO,C,\mathbb{E}))$ containing $x \in V$ of the degeneracy index in the chart $d_{C}(\varphi(x))$; see \cite[Def.~2.13]{HWZbook}. For a smooth point $x \in \cB_{\infty}$, it is explained in \cite[Rmk.~2.4]{HWZbook} that there is a global interpretation of this quantity:
\begin{equation} \label{eq:degindexdef}
d_{\cB}(x) = \dim(T_x\cB/T_x^R\cB) \text{ for } x \in \cB_{\infty}.
\end{equation}

We now recall the definition of a tame strong bundle $\rho : \cE \rightarrow \cB$.

\begin{dfn} \label{dfn:tamebundle}  \cite[Defs.~2.25, 2.26]{HWZbook} \text{}
\begin{itemlist}
\item Consider a continuous surjective map $\rho : \cE \rightarrow \cB$ from a topological space $\cE$ onto a topological space $\cB$ such that, for every $x \in \cB$, the fiber $\rho^{-1}(x) = \cE_x$ has the structure of a Banach space. A {\bf tame strong bundle chart} for $\rho : \cE \rightarrow \cB$ is a tuple
$$
(\rho^{-1}(V), \Phi, (K,C\lhd \mathbb{F}, \mathbb{E} \lhd \mathbb{F}))
$$
in which $(K,C\lhd \mathbb{F}, \mathbb{E} \lhd \mathbb{F})$ is a tame strong bundle retract (Section~\ref{subsec:slicebundleretracts}) covering a tame sc-retract $(\cO,C,\mathbb{E})$, i.e.\ $K \rightarrow \cO$ is a tame strong local bundle model. In addition, $V \subset \cB$ is the domain of a tame $M$-polyfold chart $(V,\varphi,(\cO,C,\mathbb{E}))$ and
$$\Phi : \rho^{-1}(V) \rightarrow K$$
is a homeomorphism onto $K$ that covers $\varphi$, i.e.
\begin{equation} \label{eq:bundlechartcoverbasechart}
\pi_{\cO} \circ \Phi = \varphi \circ \rho
\end{equation}
holds on $\rho^{-1}(V)$. Moreover, $\Phi$ is linear on the fibers over $V$, i.e.\ for all $x \in V$ the map $\Phi : \rho^{-1}(x) \rightarrow \pi_{\cO}^{-1}(\varphi(x))$ is a bounded linear isomorphism of Banach spaces.

\item A {\bf tame strong bundle} $\rho : \cE \rightarrow \cB$ is a continuous surjective map from a paracompact Hausdorff space $\cE$ onto a tame $M$-polyfold $\cB$ such that every fiber has the structure of a Banach space, together with an equivalence class of tame bundle atlases. A {\bf tame bundle atlas} consists of a cover of $\cB$ by tame strong bundle charts whose transition maps are sc-smooth. Atlases are equivalent if and only if their union is an atlas.
\end{itemlist}
\end{dfn}

The tame strong bundle charts on a tame strong bundle $\rho : \cE \rightarrow \cB$ induce a double filtration $\cE_{m,k}$ for $m \geq 0$ and $0 \leq k \leq m + 1$ from the double filtration \eqref{eq:doublefiltrationlocal} on the bundle retracts in the charts. From this we distinguish the $M$-polyfolds $\cE[i]$ for $i = 0,1,$ by the filtrations
\begin{equation} \label{eq:doublefiltrationglobal}
\cE[i]_m = \cE_{m,m+i}, \,\, m \geq 0.
\end{equation}
Both projections
$$\rho[i] : \cE[i] \rightarrow \cB$$
are sc-smooth maps.

There are two notions of smoothness for a section of $\rho$ (see \cite[Def.~2.27]{HWZbook}), which correspond to the two projections $\rho[i]$ for $i = 0,1$. A section $s : \cB \rightarrow \cE$ is called a \emph{sc-smooth section} of $\rho$ if $s : \cB \rightarrow \cE[0]$ is a sc-smooth section of the bundle $\rho[0].$ The section $s$ is called a \emph{sc$^+$-section} of $\rho$ if $s(\cB) \subset \cE[1]$ and $s : \cB \rightarrow \cE[1]$ is a sc-smooth section of the bundle $\rho[1].$ The sc$^+$-sections become important when perturbating sc-Fredholm sections. They are analogous to compact operators in classical Fredholm theory. In particular, adding sc$^+$-sections to sc-Fredholm sections preserves the sc-Fredholm property \cite[Thm.~3.2]{HWZbook}.

We now recall the definition of a sc-Fredholm section $\s : \cB \rightarrow \cE$.

\begin{dfn} \label{dfn:scfredholmsection} \cite[Def.~3.8]{HWZbook} \text{}
\begin{itemlist}
\item Consider a section $\s : \cB \rightarrow \cE$ of a tame strong bundle $\rho : \cE \rightarrow \cB$. Let $x \in \cB_{\infty}$ be a smooth point. Then a tame strong bundle chart $(\rho^{-1}(V), \Phi, (K,C\lhd \mathbb{F}, \mathbb{E} \lhd \mathbb{F}))$  for $\rho$ covering a tame $M$-polyfold chart  $(V,\varphi,(\cO,C,\mathbb{E}))$ on $\cB$ is called a {\bf sc-Fredholm chart for $\mathbold{\s}$ at $\mathbold{x}$} if it satisfies $x \in V$ and $\varphi(x) = 0 \in \cO$, and if the section
$$\Phi \circ \s \circ \varphi^{-1} : \cO \rightarrow K$$
of the local bundle model $K \rightarrow \cO$ is a local sc-Fredholm germ (Definition~\ref{dfn:localscFredgerm}).

\item A {\bf sc-Fredholm section} $\s : \cB \rightarrow \cE$ of $\rho$ is a section that has the following properties:

(1) $\s$ is sc-smooth.

(2) $\s$ is {\bf regularizing}: $x \in \cB_m$, $\s(x) \in \cE_{m,m+1}$ implies $x \in \cB_{m+1}$.

(3) For every smooth point $x \in \cB_{\infty}$, there exists a sc-Fredholm chart for $\s$ at $x$. The {\bf index $\mathbold{ind_x(\s)}$ of $\mathbold{\s}$ at $\mathbold{x}$} is the index \eqref{eq:fredholmindexdef} of the local sc-Fredholm germ in any such chart.
\end{itemlist}
\end{dfn}

We now introduce the new stronger notion of a tame sc-Fredholm section $\s : \cB \rightarrow \cE$ of a tame $M$-polyfold bundle $\rho : \cE \rightarrow \cB$.  A tame sc-Fredholm section is a sc-Fredholm section that is locally tame in the sense of Definition~\ref{dfn:tamescFredgerm}.

\begin{dfn} \label{dfn:tamescfredholmsection} \text{}
\begin{itemlist}
\item Consider a section $\s : \cB \rightarrow \cE$ of a tame strong bundle $\rho : \cE \rightarrow \cB$. Let $x \in \cB_{\infty}$ be a smooth point. Then a {\bf tame sc-Fredholm chart for $\mathbold{\s}$ at $\mathbold{x}$} is a sc-Fredholm chart (Definition~\ref{dfn:scfredholmsection}) such that the local sc-Fredholm germ $\Phi \circ \s \circ \varphi^{-1}$ is in addition a tame sc-Fredholm germ (Definition~\ref{dfn:tamescFredgerm}).

\item A {\bf tame sc-Fredholm section} $\s : \cB \rightarrow \cE$ of $\rho$ is a sc-Fredholm section (Definition~\ref{dfn:scfredholmsection}) such that there exists a tame sc-Fredholm chart for $\s$ at every $x \in \cB_{\infty}$.
\end{itemlist}
\end{dfn}

\begin{rmk} \label{rmk:levelshiftglobal}
Given a strong bundle $\rho : \cE \rightarrow \cB$ and a sc-Fredholm section $\s : \cB \rightarrow \cE$, we obtain the $m$-shift $\rho : \cE^m \rightarrow \cB^m$ and similarly $\s : \cB^m \rightarrow \cE^m$ by restricting domains and codomains to the $m$-shifts $\cE^m$ and $\cB^m$. Performing this shift preserves all essential features of the polyfold setup. In particular, $\s : \cB^m \rightarrow \cE^m$ is a sc-Fredholm section of the strong bundle $\rho : \cE^m \rightarrow \cB^m$ and the sc-Fredholm index is preserved \cite[Cor.~5.1]{HWZbook}.

Crucially, the zero set of a sc-Fredholm section is preserved under $m$-shifts. Indeed, since $\s^{-1}(0) \subset \cB_{\infty}$ by the regularizing property (Definition~\ref{dfn:scfredholmsection}(2)) of $\s$, it follows that $$\s^{-1}(0) = \s^{-1}(0) \cap \cB^m = \s|_{\cB^m}^{-1}(0).$$ Moreover, compactness of $\s^{-1}(0)$ in the topology of $\cB_{m}$ for any fixed $m \geq 0$ is equivalent to compactness in the topology of $\cB_{\infty}$ (see \cite[Thm.~5.3,~Cor.~5.1]{HWZbook}). For this reason, we can unambiguously refer to compactness of $\s^{-1}(0)$ without reference to the levels, and $m$-shifting preserves this compactness.
\end{rmk}

We now introduce the new notions of a slice $\tilde{\cB} \subset \cB$ of a tame $M$-polyfold $\cB$, a slice of a tame strong bundle $\rho : \cE \rightarrow \cB$, and a slice of a sc-Fredholm section $\s : \cB \rightarrow \cE$ (Definition~\ref{dfn:globalslice}). Roughly, a slice is a subspace $\tilde{\cB}$ such that for every point $x \in \tilde{\cB}$ there is a tame $M$-polyfold chart that identifies a neighborhood of $x$ in $\cB$ with a $\mathbb{R}^n$-sliced sc-retract (Definition~\ref{dfn:sliceretraction}), and such that the induced tame sc-retract \eqref{eq:inducedscretract} is identified with a neighborhood of $x$ in $\tilde{\cB}$. To be a slice of the sc-Fredholm section $\s$, we additionally require that, at smooth points $x \in \tilde{\cB}_{\infty} = \tilde{\cB} \cap \cB_{\infty}$, we can choose the chart on $\cB$ such that it is covered by a bundle chart for $\rho$ in which $\s$ is a $\mathbb{R}^n$-sliced sc-Fredholm germ (Definition~\ref{dfn:slicescFredholmgerm}). In such a chart, the induced tame sc-Fredholm germ (see Lemma~\ref{lem:inducedlocalscFredholmgerm}) is identified with the restriction of $\s$ to $\tilde{\cB}$ near $x$. Globally, the result is that a slice of $\s$ is a finite codimension $M$-polyfold $\tilde{\cB}$ embedded in $\cB$ that is compatible with the local fillers for $\s$ in such a way that $\s$ restricts to a tame sc-Fredholm section over $\tilde{\cB}$ (Theorem~\ref{thm:slicestructure}(III)).

We first introduce the slice notions in charts.

\begin{dfn} \label{dfn:slicecharts}
\text{}
\begin{itemlist}
\item Consider a subspace $\tilde{\cB} \subset \cB$ of a tame $M$-polyfold $\cB$. A {\bf $\mathbold{\mathbb{R}^n}$-sliced chart with respect to $\tilde{\cB} \subset \cB$} is a tame $M$-polyfold chart on $\cB$ of the form
\begin{equation}  \label{eq:slicechart}
(V,\varphi,(\cO,\mathbb{R}^n \times C, \mathbb{R}^n \times \mathbb{E}))
\end{equation}
where $(\cO,\mathbb{R}^n \times C, \mathbb{R}^n \times \mathbb{E})$ is a $\mathbb{R}^n$-sliced sc-retract (Definition~\ref{dfn:sliceretraction}) with induced tame sc-retract $\tilde{\cO} = \cO \cap (\{0\} \times C)$ (see \eqref{eq:inducedscretract}) satisfying
\begin{equation} \label{eq:slicechartcondition}
\tilde{\cO} = \varphi(\tilde{\cB} \cap V).
\end{equation}

\item Consider a tame strong bundle $\rho : \cE \rightarrow \cB$. A {\bf $\mathbold{\mathbb{R}^n}$-sliced bundle chart for $\rho$ with respect to $\tilde{\cB} \subset \cB$} is a tame strong bundle chart for $\rho$ of the form
\begin{equation}  \label{eq:slicechartbundledef}
(\rho^{-1}(V),\Phi,(K,\mathbb{R}^n \times C \lhd \mathbb{F}, \mathbb{R}^n \times \mathbb{E} \lhd \mathbb{F}))
\end{equation}
covering a $\mathbb{R}^n$-sliced chart with respect to $\tilde{\cB} \subset \cB$ as in \eqref{eq:slicechart}. In particular, the tame bundle retract $(K,\mathbb{R}^n \times C \lhd \mathbb{F}, \mathbb{R}^n \times \mathbb{E} \lhd \mathbb{F})$ is a $\mathbb{R}^n$-sliced bundle retract (Definition~\ref{dfn:slicebundleretraction}) covering the sc-retract $\cO$ from \eqref{eq:slicechart}.

\item Consider a section $\s : \cB \rightarrow \cE$ of $\rho$ and a smooth point $$x \in \tilde{\cB}_{\infty} = \tilde{\cB} \cap \cB_{\infty}.$$ Then a {\bf $\mathbold{\mathbb{R}^n}$-sliced sc-Fredholm chart for $\s$ at $x$ with respect to $\tilde{\cB} \subset \cB$} is a $\mathbb{R}^n$-sliced bundle chart with respect to $\tilde{\cB} \subset \cB$ as in \eqref{eq:slicechartbundledef} that satisfies $x \in V$ and $\varphi(x) = 0$, and such that the section
$$\Phi \circ \s \circ \varphi^{-1} : \cO \rightarrow K$$
is a $\mathbb{R}^n$-sliced sc-Fredholm germ (Definition~\ref{dfn:slicescFredholmgerm}).
\end{itemlist}
\end{dfn}

The following global notions of slices of tame $M$-polyfolds, tame strong bundles, and sc-Fredholm sections require covers of $\tilde{\cB}$ by $\mathbb{R}^n$-sliced charts.

\begin{dfn} \label{dfn:globalslice}
\text{}
\begin{itemlist}
\item Consider a tame $M$-polyfold $\cB$. A subspace $$\tilde{\cB} \subset \cB$$ is called a {\bf slice of $\mathbold{\cB}$} if for every $x \in \tilde{\cB}$ there exists an integer $$n_x = codim_x(\tilde{\cB} \subset \cB) \geq 0$$ and a $\mathbb{R}^{n_x}$-sliced chart with respect to $\tilde{\cB} \subset \cB$ that contains $x$. For each $x \in \tilde{\cB}_1 = \tilde{\cB} \cap \cB_1$, the integer $codim_x(\tilde{\cB} \subset \cB)$ is called the {\bf codimension of the slice at $\mathbold{x}$} (and it is well-defined at $x$ and locally constant in $\tilde{\cB}$, by Theorem~\ref{thm:slicestructure}).

\item Consider a tame strong bundle $\rho : \cE \rightarrow \cB$. A slice $\tilde{\cB} \subset \cB$ of $\cB$ is called a {\bf slice of the bundle $\mathbold{\rho}$} if for every $x \in \tilde{\cB}$ there exists a $\mathbb{R}^{n_x}$-sliced bundle chart for $\rho$ with respect to $\tilde{\cB} \subset \cB$ that contains $x$.

\item Consider a sc-Fredholm section $\s : \cB \rightarrow \cE$ of $\rho$. A slice $\tilde{\cB} \subset \cB$ of the bundle $\rho$ is called a {\bf slice of the sc-Fredholm section $\mathbold{\s}$} if for every $x \in \tilde{\cB}_{\infty} = \tilde{\cB} \cap \cB_{\infty}$ there exists a $\mathbb{R}^{n_x}$-sliced sc-Fredholm chart for $\s$ at $x$ with respect to $\tilde{\cB} \subset \cB$.
\end{itemlist}
\end{dfn}

We are now equipped to prove one of the main theorems (Theorem~\ref{thm:slicestructure}), which says that slices of tame $M$-polyfolds are tame $M$-polyfolds with locally constant codimension, bundles restrict to slices, and sc-Fredholm sections restrict to slices with a drop in Fredholm index given by the codimension. The essential point in the proof is that sliced charts with respect to $\tilde{\cB} \subset \cB$ induce tame $M$-polyfold charts on $\tilde{\cB}$, sliced bundle charts induce tame bundle charts on the restricted bundle map, and sliced sc-Fredholm charts induce tame sc-Fredholm charts on the restricted section. This is the $M$-polyfold analog of the classical Fact~\ref{fac:restrictbanach} (restrictions of Fredholm sections to sub-Banach manifolds).

\begin{thm} \label{thm:slicestructure} {\bf (Restrictions of sc-Fredholm sections to slices)}
\begin{enumilistfiber}
\item Consider a tame $M$-polyfold $\cB$ and a slice $\tilde{\cB} \subset \cB$ (Definition~\ref{dfn:globalslice}). Then, $\tilde{\cB}$ is a tame $M$-polyfold with atlas induced by the sliced charts with respect to $\tilde{\cB} \subset \cB$. For $x \in \tilde{\cB}_1$, the codimension $codim_x(\tilde{\cB} \subset \cB)$ is well-defined and locally constant in $\tilde{\cB}$, i.e.\ it equals $codim_{x'}(\tilde{\cB} \subset \cB)$ for every $x'$ in an open neighborhood of $x$ in $\tilde{\cB}$. For $x \in \tilde{\cB}_{\infty}$, the degeneracy index satisfies $d_{\tilde{\cB}}(x) = d_{\cB}(x)$.\\

\item Consider, in addition, a tame strong bundle $\rho : \cE \rightarrow \cB$. If $\tilde{\cB} \subset \cB$ is a slice of $\rho$, then the restriction $\tilde{\rho} := \rho|_{\tilde{\cE}} : \tilde{\cE} :=  \rho^{-1}(\tilde{\cB}) \rightarrow \tilde{\cB}$ is a tame strong bundle with atlas induced by the sliced bundle charts for $\rho$ with respect to $\tilde{\cB} \subset \cB$.\\

\item Consider, in addition, a sc-Fredholm section $\s : \cB \rightarrow \cE$. If $\tilde{\cB} \subset \cB$ is a slice of $\s$, then the restriction $\tilde{\s} = \s|_{\tilde{\cB}} : \tilde{\cB} \rightarrow \tilde{\cE}$ is a tame sc-Fredholm section (Definition~\ref{dfn:tamescfredholmsection}) of $\tilde{\rho}$ with tame sc-Fredholm charts induced by the sliced sc-Fredholm charts for $\s$ with respect to $\tilde{\cB} \subset \cB$. For $x \in \tilde{\cB}_{\infty}$, the index satisfies $ind_x(\tilde{\s}) = ind_x(\s) - codim_x(\tilde{\cB} \subset \cB)$. If $\s^{-1}(0)$ is compact and $\tilde{\cB}_{\infty} \subset \cB_{\infty}$ is closed, then $\tilde{\s}^{-1}(0)$ is compact.
\end{enumilistfiber}
\end{thm}

\begin{proof}
\ul{Proof of (I)}: By Definition~\ref{dfn:globalslice} of a slice, for every $x \in \tilde{\cB}$ there is an integer $n_x \geq 0$ and a $\mathbb{R}^{n_x}$-sliced chart with respect to $\tilde{\cB} \subset \cB$ that contains $x$. We will show that each of these sliced charts induces a tame $M$-polyfold chart on $\tilde{\cB}$. The transition maps between the tame charts on $\tilde{\cB}$ constructed in this way are sc-smooth because they are restrictions of the sc-smooth transition maps between the sliced charts on $\cB$. Hence we have a tame sc-smooth atlas on $\tilde{\cB}$, providing the claimed tame $M$-polyfold structure.

Given a $\mathbb{R}^n$-sliced chart
\begin{equation} \label{eq:slicedchatglobalsliceproof}
(V,\varphi,(\cO,\mathbb{R}^n \times C, \mathbb{R}^n \times \mathbb{E}))
\end{equation}
with respect to $\tilde{\cB} \subset \cB$, we must construct the claimed induced tame $M$-polyfold chart on $\tilde{\cB}$. Use the notation
\begin{align*}
\tilde{V} := \tilde{\cB} \cap V \text{ and } \tilde{\varphi} := \varphi|_{\tilde{V}} : \tilde{V} \rightarrow \varphi(\tilde{V}),
\end{align*}
and recall the induced tame sc-retract $(\tilde{\cO},C,\mathbb{E})$ from Lemma~\ref{lem:inducedtameretraction}. We claim that the tuple
\begin{equation} \label{eq:Mchartslice}
(\tilde{V},\tilde{\varphi},(\tilde{\cO}, C,\mathbb{E}))
\end{equation}
is a tame $M$-polyfold chart on $\tilde{\cB}$. Indeed, the set $\tilde{V}$ is open in $\tilde{\cB}$ since $V \subset \cB$ is open and $\tilde{\cB} \subset \cB$ has the subspace topology. Since $\varphi$ is a homeomorphism, it follows that its restriction $\tilde{\varphi}$ is a homeomorphism onto its image, which is $\tilde{\cO}$ by \eqref{eq:slicechartcondition} and definition of $\tilde{V}$. This completes the construction of the tame $M$-polyfold structure on $\tilde{\cB}$.

Let $x \in \tilde{\cB}_1$. We now verify that the codimension
$$n_x = codim_x(\tilde{\cB}~\subset~\cB)$$
of the slice $\tilde{\cB} \subset \cB$ at $x$ is well-defined and is locally constant in $\tilde{\cB}$. There exists a $\mathbb{R}^{n_x}$-sliced chart as in \eqref{eq:slicedchatglobalsliceproof} with $x \in V$. It suffices to show that, for all $y \in \tilde{V}_1$, we have $T_y{\cB}/T_y\tilde{\cB} \cong \mathbb{R}^{n_x}$. Indeed, this proves that the codimension $n_x$ is independent of the sliced chart around $x$, hence is well-defined, and moreover that it is constant in the open neighborhood $\tilde{V}$ of $x$ in $\tilde{\cB}$. For every $y \in \tilde{V}_1$, the tangent map $$D_y\varphi : T_yV \rightarrow T_{\varphi(y)}\cO$$ of the chart map $\varphi$ at $y$ is a linear isomorphism. Moreover, $D_y\varphi$ restricts to a linear isomorphism $D_y\tilde{\varphi} : T_x\tilde{V} \rightarrow T_{\varphi(y)}\tilde{\cO}$. Hence $D_y\varphi$ induces an isomorphism $$T_yV/T_y\tilde{V} \cong T_{\varphi(y)}\cO/T_{\varphi(y)}\tilde{\cO}.$$ Thus $$T_y\cB / T_y\tilde{\cB} = T_yV/T_y\tilde{V} \cong \mathbb{R}^{n_x}$$ holds by \eqref{eq:codimentangentspacesslicedretract}. Hence $codim_y(\tilde{\cB} \subset \cB) = n_x$ holds for all $y \in \tilde{V}_1$.

We now prove the claimed degeneracy index formula. Let $x \in \tilde{\cB}_{\infty}$. Since $x$ is a smooth point, the reduced tangent space $T_x^R\cB$ is well-defined and $D_x\varphi$ restricts to a sc-isomorphism $T_x^R\cB \rightarrow T_{\varphi(x)}^R\cO$. Hence $D_x\varphi$ induces a sc-isomorphism $T_x\cB/T_x^R\cB \cong T_{\varphi(x)}\cO/T_{\varphi(x)}^R\cO$. Similarly, $D_x\tilde{\varphi}$ induces a sc-isomorphism $T_x\tilde{\cB}/T_x^R\tilde{\cB} \cong T_{\varphi(x)}\tilde{\cO}/T_{\varphi(x)}^R\tilde{\cO}$. Then by \eqref{eq:inducedreducatedtangentspacequotientrelationship}, the inclusion $\tilde{\cO} \subset \cO$ induces an isomorphism $$T_{\varphi(x)}\tilde{\cO}/T_{\varphi(x)}^R\tilde{\cO} \cong T_{\varphi(x)}\cO/T_{\varphi(x)}^R\cO.$$ It follows that the inclusion $\tilde{\cB} \subset \cB$ induces an isomorphism $$T_x\tilde{\cB}/T_x^R\tilde{\cB} \cong T_x\cB/T_x^R\cB.$$ So by the global description \eqref{eq:degindexdef} of degeneracy index at smooth points, we conclude $d_{\tilde{\cB}}(x) = \dim(T_x\tilde{\cB}/T_x^R\tilde{\cB}) = \dim(T_x\cB/T_x^R\cB) = d_{\cB}(x)$. This completes the proof of the statements in (I).
\\

\ul{Proof of (II)}: By definition of a slice of a bundle, there is a cover of $\tilde{\cB}$ by sliced bundle charts with respect to $\tilde{\cB} \subset \cB$. We will show that each of these sliced bundle charts induces a tame strong bundle chart for $\tilde{\rho} = \rho|_{\tilde{\cE}} : \tilde{\cE} := \rho^{-1}(\tilde{\cB}) \rightarrow \tilde{\cB}$. The transition maps between the tame charts for $\tilde{\rho}$ constructed in this way are sc-smooth because they are restrictions of the transition maps between the sliced bundle charts for $\rho$. Hence we have a bundle atlas for $\tilde{\rho}$, providing the claimed tame strong bundle structure.

Given a $\mathbb{R}^n$-sliced bundle chart
\begin{equation} \label{eq:slicedbundlechartglobalrestrictionproof}
(\rho^{-1}(V),\Phi,(K,\mathbb{R}^n \times C \lhd \mathbb{F}, \mathbb{R}^n \times \mathbb{E} \lhd \mathbb{F}))
\end{equation}
for $\rho$ with respect to $\tilde{\cB} \subset \cB$ that covers a $\mathbb{R}^n$-sliced chart with respect to $\tilde{\cB} \subset \cB$ as in \eqref{eq:slicedchatglobalsliceproof}, we must construct the claimed induced tame strong bundle chart on $\tilde{\rho}$. Let $(\tilde{K},C \lhd \mathbb{F}, \mathbb{E} \lhd \mathbb{F})$ denote the induced tame bundle retract from Lemma~\ref{lem:inducedtamebundleretraction} and use the notation
\begin{align*}
\tilde{\Phi} &:= \Phi|_{\tilde{\rho}^{-1}(\tilde{V})} : \tilde{\rho}^{-1}(\tilde{V})\rightarrow \Phi(\tilde{\rho}^{-1}(\tilde{V})).
\end{align*}
We claim that the tuple
\begin{equation} \label{eq:Mchartbundleslice}
(\tilde{\rho}^{-1}(\tilde{V}),\tilde{\Phi},(\tilde{K},C \lhd \mathbb{F}, \mathbb{E} \lhd \mathbb{F}))
\end{equation}
is a tame strong bundle chart for $\tilde{\rho}$ that covers the induced $M$-polyfold chart \eqref{eq:Mchartslice}. Since $\pi_{\cO} \circ \Phi = \varphi \circ \rho$ by \eqref{eq:bundlechartcoverbasechart} it follows by restriction that $\pi_{\tilde{\cO}} \circ \tilde{\Phi} = \tilde{\varphi} \circ \tilde{\rho}$ on $\tilde{\rho}^{-1}(\tilde{V})$. It follows that $im(\tilde{\Phi}) = \pi_{\cO}^{-1}(\tilde{\cO}) = \tilde{K}$, where the second equality holds by \eqref{eq:inducedscretractbundle}. Hence $\tilde{\Phi}$ is a homeomorphism onto its image $\tilde{K}$ because it is the restriction of the homeomorphism $\Phi$, and similarly $\tilde{\Phi}$ is linear on fibers as the restriction of $\Phi$. This completes the proof of (II).
\\

\ul{Proof of (III)}: The final statement about compactness of the zero sets holds because $\tilde{\s}^{-1}(0) = \s^{-1}(0) \cap \tilde{\cB}_{\infty}$ is the intersection of a compact set and a closed subset of $\cB_{\infty}$.

We proceed to verify that $\tilde{\s} = \s|_{\tilde{\cB}} : \tilde{\cB} \rightarrow \tilde{\cE}$ is a tame sc-Fredholm section of $\tilde{\rho}$. The section $\tilde{\s}$ is sc-smooth and regularizing because it is the restriction of the section $\s$. So to prove that $\tilde{\s}$ is a tame sc-Fredholm section it remains to show that there exists a tame sc-Fredholm chart (Definition~\ref{dfn:tamescfredholmsection}) for $\tilde{\s}$ at every smooth point $x \in \tilde{\cB}_{\infty}$. Since $\tilde{\cB}$ is a slice of the sc-Fredholm section $\s$, there exists a $\mathbb{R}^{n_x}$-sliced sc-Fredholm chart for $\s$ at $x$ with respect to $\tilde{\cB} \subset \cB$. This is a $\mathbb{R}^{n_x}$-sliced bundle chart as in \eqref{eq:slicedbundlechartglobalrestrictionproof} such that the local sc-Fredholm germ $\Phi \circ \s \circ \varphi^{-1}$ is $\mathbb{R}^{n_x}$-sliced and such that $\varphi(x) = 0$.

We claim that the tame bundle chart \eqref{eq:Mchartbundleslice} is a tame sc-Fredholm chart for $\tilde{\s}$ at $x$. This requires that the local section $\tilde{\Phi} \circ \tilde{\s} \circ \tilde{\varphi}^{-1} : \tilde{\cO} \rightarrow \tilde{K}$ is a tame sc-Fredholm germ, which follows from Lemma~\ref{lem:inducedlocalscFredholmgerm} since $$\Phi \circ \s \circ \varphi^{-1} : \cO \rightarrow K$$ is a $\mathbb{R}^{n_x}$-sliced sc-Fredholm germ. Moreover, Lemma~\ref{lem:inducedlocalscFredholmgerm} and Definition~\ref{dfn:scfredholmsection}(3) provide the sc-Fredholm index formula $$ind_x(\tilde{\s}) = ind(\tilde{\Phi} \circ \tilde{\s} \circ \tilde{\varphi}^{-1}) = ind(\Phi \circ \s \circ \varphi^{-1}) - n_x = ind_x(\s) - n_x.$$ This completes the proof that $\tilde{\s}$ is a tame sc-Fredholm section satisfying the claimed index formula.
\end{proof}

We now prepare for the proof of our main theorem (Theorem~\ref{thm:globalimplicit}), which is the $M$-polyfold analog of the classical Fact~\ref{fac:globalimplicitbanachfredholm} (transverse preimages are sub-Banach manifolds). The result is roughly as follows. Consider a tame sc-Fredholm section $\s : \cB \rightarrow \cE$ of a tame $M$-polyfold bundle $\rho : \cE \rightarrow \cB$. Given a sc-smooth map $f : \cB \rightarrow Y$ to a finite dimensional smooth manifold $Y$ that is $\s$-compatibly transverse (Definition~\ref{dfn:transverse}) to a submanifold $N \subset Y$, we prove that the restriction of $\s$ to the preimage $f^{-1}(N)$ is a tame sc-Fredholm section. The notion of $\s$-compatible transversality requires compatibility between the map $f$, the submanifold $N$, and the coordinate changes that bring the local fillers of $\s$ into basic germ form.

All necessary compatibility is satisfied in applications. In particular, evaluation maps $f = ev$ at marked points are $\overline{\partial}_J$-compatibly transverse to every submanifold, where $\overline{\partial}_J$ is the Cauchy-Riemann section; see Section~\ref{sec:evaluationmaps}. We now introduce the precise definition of $\s$-compatibly transverse.

\begin{dfn} \label{dfn:transverse}
Consider a tame $M$-polyfold $\cB$, a smooth manifold $Y$ together with a codimension-$n$ submanifold $N \subset Y$, and a sc-smooth map $f : \cB \rightarrow Y$. Then $f$ is called {\bf transverse to $\mathbold{N}$} if
\begin{equation} \label{eq:transversedefiningproperty}
D_xf(T_x^R\cB) + T_{f(x)}N = T_{f(x)}Y
\end{equation}
holds for all $x \in f^{-1}(N) \cap \cB_{\infty},$

Suppose that $f$ is transverse to $N$, and consider in addition a tame sc-Fredholm section $\s : \cB \rightarrow \cE$ (Definition~\ref{dfn:tamescfredholmsection}) of a tame strong bundle $\rho : \cE \rightarrow \cB$. Then $f$ is called {\bf $\mathbold{\s}$-compatibly transverse to $\mathbold{N}$} if, for all points $x \in f^{-1}(N) \cap \cB_{\infty}$, there exists a tame sc-Fredholm chart (Definition~\ref{dfn:tamescfredholmsection}) for $\s$ at $x$, denoted
$${(\rho^{-1}(V), \Phi, (K,[0,\infty)^s \times \mathbb{E} \lhd \mathbb{F}, \mathbb{R}^s \times \mathbb{E} \lhd \mathbb{F}))},$$ covering a tame $M$-polyfold chart ${(V,\varphi,(\cO,[0,\infty)^s \times \mathbb{E}, \mathbb{R}^s \times \mathbb{E}))}$ on $\cB$ such that the tame sc-Fredholm germ (Definition~\ref{dfn:tamescFredgerm})
$${\Phi \circ \s \circ \varphi^{-1} : \cO \rightarrow K}$$ satisfies the following property: There exists a choice of sc-Banach space $\mathbb{W}$ and linear sc-isomorphism
$$\overline{\psi} : \mathbb{E} \rightarrow \mathbb{R}^{k-s} \times \mathbb{W}$$
satisfying the conditions in Definition~\ref{dfn:tamescFredgerm} of tame sc-Fredholm germ such that, in addition, there exists a sc-complement $L$ of the subspace ${(D_xf)^{-1}(T_{f(x)}N) \cap T_x^R\cB} \subset T_x^R\cB$ satisfying
\begin{equation} \label{eq:thespecialpropertyglobalversion}
\overline{\psi} \circ D_x\varphi(L) \subset \{0\}^{k-s} \times \mathbb{W}.
\end{equation}
\end{dfn}

We are now equipped to prove the main theorem.

\begin{thm} \label{thm:globalimplicit} {\bf (Transverse preimages are slices of sc-Fredholm sections)}
\begin{enumilistfiber}
\item Consider a tame $M$-polyfold $\cB$, a smooth manifold $Y$ together with a codimension-$n$ submanifold $N \subset Y$, and a sc-smooth map $$f : \cB \rightarrow Y.$$ Assume that $f$ is transverse to $N$ (Definition~\ref{dfn:transverse}).

Then, there exists an open neighborhood
$$\tilde{\cB} \subset f^{-1}(N) \cap \cB_1$$
of ${f^{-1}(N) \cap \cB_{\infty}}$ such that $\tilde{\cB}$ is a slice of $\cB^1$ with $codim_x(\tilde{\cB} \subset \cB^1) = n$ for every $x \in \tilde{\cB}_1 = \tilde{\cB} \cap \cB_2$. In particular, $\tilde{\cB}$ is a tame $M$-polyfold with degeneracy index satisfying $d_{\tilde{\cB}}(x) = d_{\cB}(x)$ for all $x \in \tilde{\cB}_{\infty}$.\\

\item Consider, in addition, a tame strong bundle $\rho : \cE \rightarrow \cB$. Then, there exists a possibly smaller neighborhood $\tilde{\cB}$ in (I) that is a slice of the bundle $\rho|_{\cE^1} : \cE^1 \rightarrow \cB^1$. In particular, the restriction
$$\tilde{\rho} := \rho|_{\tilde{\cE}} : \tilde{\cE} := (\rho|_{\cE^1})^{-1}(\tilde{\cB}) \rightarrow \tilde{\cB}$$
is a tame strong bundle.\\

\item Consider, in addition, a tame sc-Fredholm section $\s : \cB \rightarrow \cE$ (Definition~\ref{dfn:tamescfredholmsection}) of $\rho$. Assume that $f$ is $\s$-compatibly transverse to $N$ (Definition~\ref{dfn:transverse}). Then, there exists a possibly smaller neighborhood $\tilde{\cB}$ in (II) that is a slice of the tame sc-Fredholm section $\s|_{\cB^1} : \cB^1 \rightarrow \cE^1$. In particular, the restriction
$$\tilde{\s} := \s|_{\tilde{\cB}} : \tilde{\cB} \rightarrow \tilde{\cE}$$
is a tame sc-Fredholm section of $\tilde{\rho}$ with index satisfying $$ind_x(\tilde{\s}) = ind_x(\s) - n$$ for all $x \in \tilde{\cB}_{\infty}$. If $N$ is closed as a subset of $Y$ and $\s^{-1}(0)$ is compact, then $\tilde{\s}^{-1}(0)$ is compact.
\end{enumilistfiber}
\end{thm}

\begin{proof}
\ul{Proof of (I)}: For every $x \in f^{-1}(N) \cap \cB_{\infty}$ we will construct a $\mathbb{R}^n$-sliced chart (Definition~\ref{dfn:slicecharts}) with respect to $f^{-1}(N) \cap \cB^1 \subset \cB^1$ that contains $x$. Then we define $\tilde{\cB}$ to be the union of the domains of these charts intersected with $f^{-1}(N) \cap \cB^1$.

Let $x \in f^{-1}(N) \cap \cB_{\infty}$. Let $Z \subset Y$ be the domain of a manifold chart $Z \xrightarrow{\sim} \mathbb{R}^{n+\dim N}$ on $Y$ containing $f(x) \in N$ that identifies $N \cap Z$ with $\{0\}^n \times \mathbb{R}^{\dim N}$. Let $\g : Z \rightarrow \mathbb{R}^n$ be the smooth map given in the chart by the projection $\mathbb{R}^n \times \mathbb{R}^{\dim N} \rightarrow \mathbb{R}^n$. Then $0 = \g(f(x))$ is a regular value of $\g$ and
\begin{equation} \label{eq:transverse1}
\g^{-1}(0) = N \cap Z.
\end{equation}
In particular, we have
\begin{equation} \label{eq:gammatransverse}
T_{f(x)}N = \ker D_{f(x)} \g.
\end{equation}

Then $f^{-1}(Z)$ is an open neighborhood of $x$ in $\cB$. Choose $V \subset f^{-1}(Z)$ so that, in addition, there is a tame $M$-polyfold chart
\begin{equation} \label{eq:tamechartinglobalconstraint}
(V,\varphi,(\cO,[0,\infty)^s \times \mathbb{E}, \mathbb{R}^s \times \mathbb{E}))
\end{equation}
on $\cB$ satisfying $\varphi(x) \in \{0\}^s \times \mathbb{E}$. Set
$$\ul{x} := \varphi(x).$$
By definition \eqref{eq:defglobalredcuedtangent} of the reduced tangent space $T_x^R\cB$, we have $$D_x\varphi(T_x^R\cB) = T_{\ul{x}}^R\cO.$$ So, since $f$ is transverse to $N$, we conclude from \eqref{eq:transversedefiningproperty} that we have
\begin{equation} \label{eq:transverseinchart}
D_xf \circ D_{\ul{x}}\varphi^{-1}(T_{\ul{x}}^R\cO) + T_{f(x)}N = T_{f(x)}Y.
\end{equation}

We claim that the map
$$\ul{f} := \g \circ f \circ \varphi^{-1} : \cO \rightarrow \mathbb{R}^n$$
satisfies the hypotheses of Lemma~\ref{lem:slicecoordinates}(I) at $\ul{x}$. Let $\cO_{\partial} := \cO \cap (\{0\}^s \times \mathbb{E})$ denote the boundary sc-retract associated to $\cO$ (see Lemma~\ref{lem:associatedboundaryretract}) and consider the restriction $\ul{f}_{\partial} := \ul{f}|_{\cO_{\partial}}$. Recall from \eqref{eq:boundaryretracttangentspace} that we have $T_{\ul{x}}^R\cO = T_{\ul{x}}\cO_{\partial}.$ Then the tangent map $D_{\ul{x}}\ul{f}_{\partial} : T_{\ul{x}}\cO_{\partial} \rightarrow T_0\mathbb{R}^n = \mathbb{R}^n$ is surjective by \eqref{eq:transverseinchart}, \eqref{eq:gammatransverse}, and since $0$ is a regular value of $\g$. Hence, the hypotheses of Lemma~\ref{lem:slicecoordinates}(I) are indeed satisfied, yielding slice coordinates around $\ul{x}$ with respect to $(\ul{f}^{-1}(0) \cap \cO^1) \subset \cO^1$. Precisely, this means that we obtain a sc-Banach space $\mathbb{K}$, a $\mathbb{R}^n$-sliced sc-retract (Definition~\ref{dfn:sliceretraction})
\begin{equation} \label{eq:slicescretractpreimageglobal}
(\cO',\mathbb{R}^n \times [0,\infty)^s \times \mathbb{K}^1,\mathbb{R}^n \times \mathbb{R}^s \times \mathbb{K}^1),
\end{equation}
an open set $\hat{\cO} \subset \cO^1$, and a sc-smooth diffeomorphism
$$g :  \hat{\cO} \rightarrow \cO'$$
such that the induced tame sc-retract
$$\tilde{\cO}' = \cO' \cap (\{0\}^n \times [0,\infty)^s \times \mathbb{K}^1)$$
satisfies
\begin{equation} \label{eq:preimageglobalsliceretractnormalforminduced}
g(\ul{f}^{-1}(0) \cap \hat{\cO}) = \tilde{\cO}'.
\end{equation}

Set
$$\hat{V} := \varphi^{-1}(\hat{\cO}) \text{ and } \hat{\varphi} := g \circ \varphi|_{\hat{V}} : \hat{V} \rightarrow \cO'.$$
Then
\begin{equation} \label{eq:slicechartglobal}
(\hat{V},\hat{\varphi},(\cO',\mathbb{R}^n \times [0,\infty)^s \times \mathbb{K}^1,\mathbb{R}^n \times \mathbb{R}^s \times \mathbb{K}^1))
\end{equation}
is an $M$-polyfold chart on $\cB^1$.

At this point, for every $x \in f^{-1}(N) \cap \cB_{\infty}$, we have constructed an $M$-polyfold chart on $\cB^1$ with domain $\hat{V}_x \subset \cB^1$ containing $x \in \hat{V}_x$. Define 
\begin{equation} \label{eq:sliceunion}
\tilde{\cB} := \bigcup_{x \in f^{-1}(N) \cap \cB_{\infty}} \hat{V}_x \cap f^{-1}(N).
\end{equation}
We claim that $\tilde{\cB} \subset \cB^1$ is a slice. We must show that the chart \eqref{eq:slicechartglobal} is a $\mathbb{R}^n$-sliced chart (Definition~\ref{dfn:slicecharts}) with respect to $\tilde{\cB} \subset \cB^1$. The sc-retract \eqref{eq:slicescretractpreimageglobal} is given as $\mathbb{R}^n$-sliced by Lemma~\ref{lem:slicecoordinates}(I), so it remains to show that \eqref{eq:slicechartcondition} holds, which in the notation of this proof is the statement $\tilde{\cO}' = \hat{\varphi}(\tilde{\cB} \cap \hat{V}).$ Let $p \in \tilde{\cO}'$. By \eqref{eq:preimageglobalsliceretractnormalforminduced} and the definition of $\ul{f}$, we have $\hat{\varphi}^{-1}(p) = \varphi^{-1} \circ g^{-1}(p) \in (\g \circ f)^{-1}(0) \subset f^{-1}(N).$ Hence we have $\hat{\varphi}^{-1}(p) \in \hat{V} \cap f^{-1}(N) \subset \tilde{\cB} \cap \hat{V}.$ To see the reverse inclusion, let $p \in \hat{\varphi}(\tilde{\cB} \cap \hat{V}).$ By \eqref{eq:preimageglobalsliceretractnormalforminduced}, we must show that $\ul{f}(g^{-1}(p)) = 0$. From the definitions, we compute
$$\ul{f} \circ g^{-1}(p) = \g \circ f \circ \hat{\varphi}^{-1}(p) \in \g(f(\tilde{\cB} \cap \hat{V})) \subset \g(N \cap Z) = \{0\},$$
as required. Hence \eqref{eq:slicechartglobal} is indeed a $\mathbb{R}^n$-sliced chart with respect to $\tilde{\cB} \subset \cB^1$. This proves that $\tilde{\cB} \subset \cB^1$ is a slice (Definition~\ref{dfn:globalslice}). Hence $\tilde{\cB}$ is a tame $M$-polyfold with the claimed degeneracy index by Theorem~\ref{thm:slicestructure}. Moreover, the existence of a $\mathbb{R}^n$-sliced chart around every $y \in \tilde{\cB}$ proves the claim that $codim_y(\tilde{\cB} \subset \cB^1) = n$ for every $y \in \tilde{\cB}_1$. This completes the proof of the statements (I).
\\

\ul{Proof of (II)}: To prove that $\tilde{\cB}$ is a slice of the bundle $\rho|_{\cE^1} : \cE^1 \rightarrow \cB^1$, for every $x \in \tilde{\cB}$ we must construct a sliced bundle chart for $\rho|_{\cE^1}$ with respect to $\tilde{\cB} \subset \cB^1$. We can choose the tame $M$-polyfold chart \eqref{eq:tamechartinglobalconstraint} so that it is covered by a tame strong bundle chart 
\begin{equation} \label{eq:tamebundlechartinglobalconstraint}
(\rho^{-1}(V), \Phi, (K,C\lhd \mathbb{F}, \mathbb{E} \lhd \mathbb{F}))
\end{equation}
for $\rho$. Note that this may require shrinking the open neighborhood $V \subset \cB$ of $x$, and so the resulting slice $\tilde{\cB}$ of $\cB^1$ defined by the formula \eqref{eq:sliceunion} may be smaller than in part (I).

The tame bundle retract $(K, C\lhd \mathbb{F}, \mathbb{E} \lhd \mathbb{F})$ covers the tame sc-retract $\cO$. Hence Lemma~\ref{lem:slicecoordinates}(II) yields a $\mathbb{R}^n$-sliced bundle retract
$$
(K',\mathbb{R}^n \times [0,\infty)^s \times \mathbb{K}^1 \lhd \mathbb{F}^1,\mathbb{R}^n \times \mathbb{R}^s \times \mathbb{K}^1 \lhd \mathbb{F}^1)
$$
that covers the sliced sc-retract \eqref{eq:slicescretractpreimageglobal} and satisfies
$$
(g \lhd id_{\mathbb{F}^1})(\hat{K}) = K',
$$
where $\hat{K} \subset K$ is the preimage of $\hat{\cO}$ in the local bundle model $K^1 \rightarrow \cO^1$. Set
$$\hat{\Phi} := (g \lhd id_{\mathbb{F}^1}) \circ \Phi|_{\rho|_{\cE^1}^{-1}(\hat{V})} : \rho|_{\cE^1}^{-1}(\hat{V}) \rightarrow K'.$$
Then the desired sliced bundle chart for $\rho|_{\cE^1}$ with respect to $\tilde{\cB} \subset \cB^1$ is the tuple
\begin{equation} \label{eq:slicebundlechartglobal}
(\rho|_{\cE^1}^{-1}(\hat{V}),\hat{\Phi},(K',\mathbb{R}^n \times [0,\infty)^s \times \mathbb{K}^1 \lhd \mathbb{F}^1,\mathbb{R}^n \times \mathbb{R}^s \times \mathbb{K}^1 \lhd \mathbb{F}^1)).
\end{equation}
This proves that $\tilde{\cB} \subset \cB^1$ is a slice of the bundle $\rho|_{\cE^1}$. The restriction $\tilde{\rho} : \rho|_{\cE^1}^{-1}(\tilde{\cB}) \rightarrow \tilde{\cB}$ is then a tame strong bundle by Theorem~\ref{thm:slicestructure}(II).
\\

\ul{Proof of (III)}: To prove that $\tilde{\cB}$ is a slice of the tame sc-Fredholm section $\s|_{\cB^1} : \cB^1 \rightarrow \cE^1$, for every $x \in \tilde{\cB}_{\infty}$ we must construct a sliced sc-Fredholm chart for $\s|_{\cB^1}$ at $x$ with respect to $\tilde{\cB} \subset \cB^1$. Since $f$ is $\s$-compatibly transverse to $N$, we can assume that the tame strong bundle chart \eqref{eq:tamebundlechartinglobalconstraint} is also a tame sc-Fredholm chart for $\s$ at $x$ such that we can choose $\overline{\psi}, \mathbb{W},$ and $L \subset T_x^R\cB$ satisfying \eqref{eq:thespecialpropertyglobalversion}. Note that this may require shrinking the open neighborhood $V$ of $x$, and so the resulting slice $\tilde{\cB}$ of $\rho|_{\cE^1}$ defined by the formula \eqref{eq:sliceunion} may be smaller than in parts (I) and (II).

We now show that $\mathbb{W}$, $\overline{\psi}$, and $L' := D_x\varphi(L)$ satisfy the hypotheses of Lemma~\ref{lem:slicecoordinates}(III). Indeed, by \eqref{eq:thespecialpropertyglobalversion}, $L$ is a sc-complement of the sc-subspace
$$\mathbb{A} := (D_xf)^{-1}(T_{f(x)}N) \cap T_x^R\cB$$
in $T_x^R\cB$, and so we have the sc-splitting
\begin{equation} \label{eq:LcomplementA}
D_x\varphi(L) \oplus  D_x\varphi(\mathbb{A}) = T_{\ul{x}}^R\cO.
\end{equation}
Recall from \eqref{eq:boundaryretracttangentspace} that we have $T_{\ul{x}}^R\cO = T_{\ul{x}}\cO_{\partial}.$ Then from the definitions and \eqref{eq:gammatransverse}, we compute
\begin{align*}
D_x\varphi(\mathbb{A}) &= D_x\varphi \big ( (D_xf)^{-1}(T_{f(x)}N) \cap T_x^R\cB \big )\\
&= D_x\varphi \big ( (D_xf)^{-1}(\ker D_{f(x)} \g) \big ) \cap T_{\ul{x}}^R\cO\\
&= \ker D_{\ul{x}} \ul{f} \cap T_{\ul{x}}^R\cO\\
&= \ker D_{\ul{x}}\ul{f}_{\partial},
\end{align*}
and so by \eqref{eq:LcomplementA} the space $L' = D_x\varphi(L)$ is indeed a sc-complement of $\ker D_{\ul{x}}\ul{f}_{\partial}$ in $T_{\ul{x}}\cO_{\partial}$, as required. Moreover, the required condition \eqref{eq:thespecialproperty} holds for $L'$ by \eqref{eq:thespecialpropertyglobalversion}.

Hence, Lemma~\ref{lem:slicecoordinates}(III) asserts that the section
$$\s' := (g \lhd id_{\mathbb{F}^1}) \circ \s \circ g^{-1} : \cO' \rightarrow K'$$
is a $\mathbb{R}^n$-sliced sc-Fredholm germ, which means that  \eqref{eq:slicebundlechartglobal} is a $\mathbb{R}^n$-sliced sc-Fredholm chart. This proves that $\tilde{\cB}$ is a slice of $\s|_{\cB^1}$. The restriction $\tilde{\s} : \tilde{\cB} \rightarrow \tilde{\cE}$ is then a tame sc-Fredholm section satisfying the claimed index formula by Theorem~\ref{thm:slicestructure}(III).

Since $N \subset Y$ is closed it follows that $f^{-1}(N) \subset \cB$ is closed, so $\tilde{\cB}_{\infty} = f^{-1}(N) \cap \cB_{\infty}$ is closed in $\cB_{\infty}$. Hence if $\s^{-1}(0)$ is compact then it follows from the final statement of Theorem~\ref{thm:slicestructure}(III) that $\tilde{\s}^{-1}(0)$ is compact.
\end{proof}

\subsection{Example: The Cauchy-Riemann section and evaluation maps at marked points} \label{sec:evaluationmaps}

Consider a symplectic manifold $(Y,\omega),$
a codimension-$n$ submanifold $N \subset Y,$
and an $\omega$-compatible almost complex structure $J$ on $Y$.

In applications, the $M$-polyfold $\cB$ in Theorem~\ref{thm:globalimplicit} (or, in the presence of isotropy, the ep-groupoid $X$ in Corollary~\ref{cor:globalimplicitepgroupoid}) is a space of maps $\Si \rightarrow Y$ modulo reparameterization of the domain $\Si$, where $\Si$ varies in some Deligne-Mumford space of Riemann surfaces, and
$$\s = \overline{\partial}_J : \cB \rightarrow \cE$$
is the Cauchy-Riemann section associated to $J$.

For examples of Cauchy-Riemann sections constructed in the polyfold setting, see the Gromov-Witten polyfolds in \cite{MR3683060}, the Symplectic Field Theory polyfolds in \cite{HWZsymplecticfieldtheory}\cite{PolyfoldConstructions}\cite{FHLecturesOnSFT}\cite{SFTPrimerI}\cite{SFTPrimerII}, the simplified model for Hamiltonian Floery theory polyfolds in \cite{wehrheimfredholm}, and the polyfolds for pseudoholomorphic disks with boundary on a Lagrangian in \cite{jiayong}.

We consider the evaluation map
$$f = ev : \cB \rightarrow Y$$
at a marked point\footnote{The evaluation map at $k$ marked points $\cB \rightarrow Y^k$ can be treated similarly; we consider a single marked point here for simplicity.} that varies in the domains $\Si$. The purpose of this section is to explain the following properties of this setup:
\begin{itemize}
\item $\s$ is a tame sc-Fredholm section (Definition~\ref{dfn:tamescfredholmsection}),
\item $f$ is $\s$-compatibly transverse to $N$ (Definition~\ref{dfn:transverse}).
\end{itemize}

To see that $\s$ is a tame sc-Fredholm section, we consider a smooth point $x \in \cB_{\infty}$ and explain why $\s$ is a tame sc-Fredholm germ (Definition~\ref{dfn:tamescFredgerm}) in the chart constructed around $x$ when building the polyfold $\cB$. The tame sc-retract
$$\cO \subset [0,\infty)^s \times \mathbb{R}^{t} \times D \times \mathbb{E}$$
is homeomorphic to the image of the pregluing map near $x$. Here the $[0,\infty)^s$-factor is gluing parameters near the Morse-type breakings of $x$, the $\mathbb{R}^{t}$-factor is gluing parameters near the nodal points of $x$, the space $D$ is variations of the complex structure on the domain $\Si$ of $x$ (i.e.\ tangent directions to the Deligne-Mumford space), and the sc-Banach space $\mathbb{E}$ corresponds to varying the map $x$ while preserving the matching conditions at nodes and breaking points. The tame bundle retract
$$K \subset [0,\infty)^s \times \mathbb{R}^t \times D \times \mathbb{E} \lhd \mathbb{F}$$
is homeomorphic to the image of the pregluing map in the fibers near $x$. From this local bundle model $K \rightarrow \cO$ we obtain a tame $M$-polyfold chart on $\cB$ with domain $V \subset \cB$ and chart map $\varphi : V \rightarrow \cO$ where $x \in V$ is identified with $\varphi(x) = 0 \in \cO$. It is covered by a tame strong bundle chart $\Phi : \cE|_{V} \rightarrow K$. 

The section
$$\Phi \circ \s \circ \varphi^{-1} : \cO \rightarrow K$$
is a local sc-Fredholm germ (Definition~\ref{dfn:localscFredgerm}), and moreover the natural change of coordinates of the filling to basic germ form satisfies the required conditions of a tame sc-Fredholm germ (Definition~\ref{dfn:tamescFredgerm}). Indeed, a filling
$$h : [0,\infty)^s \times \mathbb{R}^t \times D \times \mathbb{E} \rightarrow [0,\infty)^s \times \mathbb{R}^t \times D \times \mathbb{E} \lhd \mathbb{F}$$
of $\Phi \circ \s \circ \varphi^{-1}$ is constructed as in \cite[Sec.~4.5]{MR3683060}. The change of coordinates $\psi$ on the domain of the filling $h$ that brings $h$ into basic germ form (Definition~\ref{dfn:localscFredgerm}(c)) is obtained as follows; see \cite[Prop.~4.8]{MR3683060} for details. Let 
$$P := \ker (D_0 h) \cap (\{0\}^{s+t} \times \{0\} \times  \mathbb{E})$$
denote the kernel of the linearization of $h$ at $0$ in the directions $\mathbb{E}$. Note that $P$ is finite dimensional since $D_0h$ is a sc-Fredholm operator. Then, the sc-Banach space $\mathbb{W}$ and linear sc-isomorphism
$$\overline{\psi}: \mathbb{E}' := \mathbb{R}^t \times D 
\times \mathbb{E} \rightarrow \mathbb{R}^{k-s} \times \mathbb{W}$$
required in Definition~\ref{dfn:tamescFredgerm} are obtained by choosing any sc-splitting $\mathbb{E} = P \oplus \mathbb{W}$ and linear isomorphism $D \times P \cong \mathbb{R}^{k-t-s}$ to obtain the linear sc-isomorphism $\overline{\psi} : \mathbb{E}' = \mathbb{R}^t \times D \times P \oplus \mathbb{W} \rightarrow \mathbb{R}^t \times \mathbb{R}^{k-t-s} \times \mathbb{W}.$ That is, the linear sc-isomorphism $\psi = id_{[0,\infty)^s} \times \overline{\psi}$ is a suitable choice of the sc-diffeomorphism in Definition~\ref{dfn:localscFredgerm}(c) of local sc-Fredholm germ, so in particular $\Phi \circ \s \circ \varphi^{-1}$ is a tame sc-Fredholm germ. It follows that $\s = \overline{\partial}_J$ is a tame sc-Fredholm section.

\begin{prp} \label{prp:examplefiscompatibletransverse}
$f$ is $\s$-compatibly transverse to $N$.
\end{prp}
\begin{proof}
Suppose $x \in f^{-1}(N) \cap \cB_{\infty}.$ To show that $f$ is $\s$-compatibly transverse to $N$ at $x$, we now construct a sc-subspace $L \subset T_x^R\cB$ which satisfies the required conditions, in particular condition \eqref{eq:thespecialpropertyglobalversion}.

Let $Z$ be any complement of $T_{f(x)}N$ in $T_{f(x)}Y$. First we claim that it suffices to construct $L \subset T_x^R\cB$ satisfying
\begin{enumerate} \label{eq:Lrequirements}
\item $D_xf(L) = Z$ and $D_xf : L \rightarrow Z$ is an isomorphism,
\item $D_x\varphi(L) \subset \{0\}^{s+t} \times \{0\} \times \mathbb{E}$,
\item $D_x\varphi(L) \cap P = \{0\}.$
\end{enumerate}
Observe that $f$ is transverse to $N$ at $x$ by (i); indeed, the required spanning property \eqref{eq:transversedefiningproperty} holds since $D_xf(L) + T_{f(x)}N = T_{f(x)}Y$. Moreover, we claim that $L $ is a sc-complement of $\mathbb{A} := (D_xf)^{-1}(T_{f(x)}N) \cap T_x^R\cB$ in $T_x^R\cB$. Note that the linear isomorphism $D_xf|_L : L \rightarrow Z$ is automatically sc-continuous because all norms on finite dimensional spaces are equivalent. Consider the projection $\pi : T_{f(x)}Y = Z \oplus T_{f(x)}N \rightarrow Z$ and the composition $\pi \circ D_xf : T_x^R\cB \rightarrow Z$. Then we have $\mathbb{A} = \ker(\pi \circ D_xf)$, and moreover $\pi \circ D_xf$ maps $L$ isomorphically onto $Z$. Hence the sc-splitting $L \oplus \mathbb{A} = T_x^R\cB$ holds because the coordinate projections are given by the sc-operators $\Pi_L := (D_xf)|_L^{-1} \circ (\pi \circ D_xf) : T_x^R\cB \rightarrow L$ and $(id_{T_x^R\cB} - \Pi_L) : T_x^R\cB \rightarrow \mathbb{A}.$

Assuming properties (ii) and (iii) of $L$, we now explain how to choose $\mathbb{W}$ and $\overline{\psi}$ so that the required property \eqref{eq:thespecialpropertyglobalversion} for $\s$-compatible transversality holds. By (ii) and the definition of $P$, we view both $D_x\varphi(L)$ and $P$ as subspaces of $\mathbb{E}$. By (iii) and since both $D_x\varphi(L)$ and $P$ are finite dimensional sc-subspaces of $\mathbb{E}$, their span has a sc-splitting $D_x\varphi(L) \oplus P$. Let $\mathbb{W}'$ be any sc-complement of $D_x\varphi(L) \oplus P$ in $\mathbb{E}$. Then choose any linear isomorphism $D \times P \cong \mathbb{R}^{k-t-s}$ and set $\mathbb{W} = D_x\varphi(L) \oplus \mathbb{W}'$. Then, the linear sc-isomorphism $\overline{\psi} : \mathbb{E}' = \mathbb{R}^t \times D \times P \oplus \mathbb{W} \rightarrow \mathbb{R}^t \times \mathbb{R}^{k-t-s} \times \mathbb{W}$ indeed satisfies $\overline{\psi} \circ D_x\varphi(L) \subset \{0\}^{k-s} \times \mathbb{W}$, as required.

It remains to construct the sc-subspace $L \subset T_x^R\cB$ that satisfies the conditions (i)-(iii). We begin by applying Lemma~\ref{lem:avoid} as follows. Note that Lemma~\ref{lem:avoid} is deferred to after this proof. The space $\mathbb{E}$ consists of sections of the pullback tangent bundle of $Y$ along $x$ that have matching asymptotic conditions at nodes and breaking points. Since $P \subset \mathbb{E}$ is a finite dimensional subspace, by Lemma~\ref{lem:avoid} (applied to the smooth component of $\Si$ on which the marked point lies) there exists a neighborhood $U$ of the marked point in $\Si$ small enough such that the following holds: if $\xi \in P$ is a section supported in $U$, then $\xi = 0$. Moreover, the marked point is always in the complement of the nodes and breakings, so we can choose $U$ disjoint from all such special points.

Choose any basis $\{z_1,\ldots,z_n\}$ of the complement $Z$ of $T_{f(x)}N$ in $T_{f(x)}Y$. For each $i = 1,\ldots,n$, consider a sc-smooth path of the form $\g_i : (-\epsilon,\epsilon) \rightarrow \cB$ through $x = \g_i(0)$ obtained by deforming $x$ to move the image $f(x)$ of the marked point in the direction $z_i$, i.e.\ $D_xf(\g_i'(0)) = z_i$, while only changing $x$ in the neighborhood $U$. The result is that all special points are preserved along the path, i.e.\ nodes and breakings do not get glued, and moreover the complex structure on the domain of $x$ is not varying along the path. Hence we have $$D_x\varphi(\g_i'(0)) \in \{0\}^{s+t} \times \{0\} \times \mathbb{E}.$$ Define
$$L := \text{Span}(\{\g_1'(0),\ldots,\g_n'(0)\}).$$
Then it is clear from the construction that (i) and (ii) hold. Moreover (iii) holds because every $\xi \in D_x\varphi(L)$ is supported in $U$ by construction of the $\g_i$ and hence if $\xi \in D_x\varphi(L) \cap P$ then $\xi = 0$ by our choice of $U$. Note that $L \subset T_x^R\cB$ holds by (ii), since the reduced tangent space $T_x^R\cB$ is defined by $D_x\varphi(T_x^R\cB) = T_0^R\cO \cap (\{0\}^s \times \mathbb{R}^t \times D \times \mathbb{E})$. The proof is complete.
\end{proof}

The following lemma was used in the proof of Proposition~\ref{prp:examplefiscompatibletransverse}.

\begin{lem} \label{lem:avoid}
Consider a smooth manifold $\Si$ and a finite rank vector bundle $V \rightarrow \Si$. Let $\G(V)$ denote smooth sections and let $q \in \Si$. Then, if $P \subset \G(V)$ is a finite dimensional subspace, there exists an open neighborhood $U$ of $q$ in $\Si$ such that, if $\xi \in P$ is supported in $U$, then $\xi = 0$.
\end{lem}

\begin{proof}
We assume that $U$ does not exist and prove that $P$ must then be infinite dimensional. We will construct a countable set of linearly independent elements of $P$.

Consider any open neighborhood $U_0$ of $q$. Then there exists some $\xi_0 \in P$ supported in $U_0$ such that $\xi_0 \neq 0$. Since $\xi_0 \neq 0$, there exists an open neighborhood $U_1$ of $q$ such that $\xi_0$ is not supported in $U_1$ (because if a smooth section is supported in every neighborhood of a point then it is identically $0$). Note that $U_0 \supset U_1$ necessarily holds.

Inductively, assume that for some $n \geq 0$ we have constructed a nested sequence of open neighborhoods $U_0 \supset \cdots \supset U_{n+1}$ and sections $\xi_0,\ldots,\xi_n$ such that, for all $0 \leq i \leq n$, the section $\xi_i$ is supported in $U_i$ but is not supported in $U_{i+1}$. Then by our assumption that the claimed open set $U$ does not exist, there must exist some $\xi_{n+1} \in P$ supported in $U_{n+1}$ such that $\xi_{n+1} \neq 0$. Then let $U_{n+2}$ be a neighborhood of $q$ on which $\xi_{n+1}$ is not supported. Hence the inductive hypothesis holds for $n +1$.

This inductive process constructs a section $\xi_n \in P$ for all $n \geq 0$ with support lying in $U_n$ but not in $U_{n+1}$. We claim that the collection $\{ \xi_n \,\, | \,\, n \geq 0 \}$ is linearly independent, proving the lemma. Indeed, suppose $\sum_{n=0}^\infty c_n \cdot \xi_n = 0$ for some $c_n \in \mathbb{R}$. Then for every $n \geq 1$ the section $\xi_n$ is supported in $U_1$, but $\xi_0$ is not supported in $U_1$, hence $c_0 = 0$. Hence $\sum_{n=1}^\infty c_n \cdot \xi_n = 0$. Inductively, we conclude that $c_n = 0$ for all $n \geq 0$.
\end{proof}

\section{Handling isotropy: the ep-groupoid case} \label{sec:ep-groupoids}

In this section, we generalize the main theorems (Theorem~\ref{thm:slicestructure} and Theorem~\ref{thm:globalimplicit}) to the case of ep-groupoids in Corollary~\ref{cor:slicestructureepgroupoid} and Corollary~\ref{cor:globalimplicitepgroupoid}. All of the results in this section follow from the results in the $M$-polyfold case in the previous sections.

We first review tame ep-groupoids (Definition~\ref{dfn:tameepgroupoid}), bundles over them (Definition~\ref{dfn:tamestrongbundleepgroupoid}), and their sc-Fredholm section functors (Definition~\ref{dfn:scfredholmepgroupoid}), before we introduce the new notion of tame sc-Fredholm section functors (Definition~\ref{dfn:tamescfredholmepgroupoid}) and prove Corollary~\ref{cor:slicestructureepgroupoid} and Corollary~\ref{cor:globalimplicitepgroupoid}.

A \emph{groupoid} $\cX = (X,{\bf X})$ is a small category with object set $X$ and morphism set ${\bf X}$ such that all morphisms are invertible. Associated to any groupoid are the following structure maps. For a detailed description, see for example \cite[Def.~7.1]{HWZbook}. The \emph{source map}
$$s : {\bf X} \rightarrow X$$
and the \emph{target map}
$$t : {\bf X} \rightarrow X$$
send a morphism to its source and target, respectively. The \emph{multiplication map}
$$m : {\bf X} \tensor[_s]{\times}{_t} {\bf X} \rightarrow {\bf X}$$
composes any pair of morphisms such that the source of the first is the target of the second, and hence $m$ is defined on the fiber product ${\bf X} \tensor[_s]{\times}{_t} {\bf X}$. The \emph{unit map} 
$$u : X \rightarrow {\bf X}$$
sends an object to the identity morphism from that object to itself, which exists and is unique since each self-morphism set is a group. The \emph{inverse map}
$$\iota : {\bf X} \rightarrow {\bf X}$$
inverts morphisms.

\begin{dfn} \cite[Defs.~7.3, 7.6]{HWZbook} \label{dfn:tameepgroupoid}
A {\bf tame ep-groupoid} $\cX = (X,{\bf X})$ is a groupoid equipped with tame $M$-polyfold structures on the object space $X$ and on the morphism space ${\bf X}$ satisfying the following properties:
\\

\begin{enumerate}
\item {\bf (\'etale)} The source $s$ and target $t$ maps are surjective local sc-diffeomorphisms,\\

\item The unit map $u$ and the inverse map $\iota$ are sc-smooth.\\

\item {\bf (proper)} Every $x \in X$ possesses an open neighborhood $${V(x) \subset X}$$ of $x$ such that the map
$$t : s^{-1}(\overline{V(x)}) \rightarrow X$$
is proper.\\

\item The multiplication map $m$ is sc-smooth, where the fiber product ${\bf X} \tensor[_s]{\times}{_t} {\bf X}$ is equipped with the $M$-polyfold structure given by \cite[Prop.~2.15]{HWZbook}. (See Remark~\ref{rmk:hwzfiberproduct} to compare with the fiber product results in this paper.)
\end{enumerate}
\end{dfn}

\begin{rmk} \label{rmk:hwzfiberproduct}
There is an essential difference between the fiber product results in this paper and the fiber product result in \cite[Prop.~2.15]{HWZbook} that is used to give ${\bf X} \tensor[_s]{\times}{_t} {\bf X}$ an $M$-polyfold structure in Definition~\ref{dfn:tameepgroupoid}(iv).

The $M$-polyfold structure constructed in \cite[Prop.~2.15]{HWZbook} requires one of the maps in the fiber product to be a local sc-diffeomorphism. In the case of the fiber product ${\bf X} \tensor[_s]{\times}{_t} {\bf X}$ from Definition~\ref{dfn:tameepgroupoid}, both the source map $s$ and the target map $t$ are local sc-diffeomorphisms.

In this paper, we construct polyfold structures on fiber products over maps to a finite dimensional smooth manifold (Corollary~\ref{cor:ep-fiberFredholm}). In applications, these maps will never be local sc-diffeomorphisms because the $M$-polyfolds have infinite dimensional tangent spaces.
\end{rmk}

For an ep-groupoid $\cX = (X, {\bf X})$, the \emph{orbit space}
\begin{equation} \label{eq:orbitspacedef}
|X| = X / \sim
\end{equation}
is the quotient of the object space $X$ by the equivalence relation defined by $x \sim y$ if and only if there exists $\phi \in {\bf X}$ such that $s(\phi) = x$ and $t(\phi) = y$. That is, to obtain the orbit space, we identify any two objects that have a morphism between them.

The degeneracy indices $d_X : X \rightarrow \mathbb{N}_0$ and $d_{\bf X} : {\bf X} \rightarrow \mathbb{N}_0$ are defined as usual on the $M$-polyfolds $X$ and ${\bf X}$. The \emph{induced degeneracy index} \cite[Def~7.5]{HWZbook}
$$d_{|X|} : |X| \rightarrow \mathbb{N}_0$$
on the orbit space $|X|$ is defined by $d_{|X|}(|x|) = d_X(x)$, and is well-defined by \cite[Prop.~2.7]{HWZbook} as discussed above the definition \cite[Def~7.5]{HWZbook}.

Now we review the notion of a strong bundle over an ep-groupoid $\cX = (X, {\bf X})$; see \cite[Sec.~8.3]{HWZbook} for more detail. Consider a strong $M$-polyfold bundle
$$P : E \rightarrow X$$
over the object space $X$ of $\cX$. Since the source map $s$ is by definition a local sc-diffeomorphism, an $M$-polyfold structure on the fiber product ${\bf X} \tensor[_s]{\times}{_P} E$ is provided by \cite[Prop.~2.15]{HWZbook}. Moreover, the projection onto the first factor
$$\pi_1 : {\bf E} := {\bf X} \tensor[_s]{\times}{_P} E \rightarrow {\bf X}$$
is a strong $M$-polyfold bundle.

\begin{dfn} \label{dfn:tamestrongbundleepgroupoid} \cite[Def.~8.4]{HWZbook} A {\bf tame strong bundle} over an ep-groupoid $\cX = (X, {\bf X})$ is a pair $(P,\mu)$ of a tame strong bundle
$$P : E \rightarrow X$$
over the object $M$-polyfold $X$ and a strong bundle map
$$\mu : {\bf X} \tensor[_s]{\times}{_P} E \rightarrow E$$
covering the target map $t : {\bf X} \rightarrow X$, i.e.
$$P \circ \mu = t \circ \pi_1,$$
and which satisfies
\begin{enumerate}
\item $\mu(1_x,e) = e$ for all $x \in X$ and $e \in E_x$,\\

\item $\mu(g \circ h,e) = \mu(g,\mu(h,e))$ for all $g, h \in {\bf X}$ and $e \in E$ satisfying $s(h) = P(e)$ and $t(h) = s(g) = P(\mu(h,e)).$
\end{enumerate}
We call $\mu$ the {\bf strong bundle structure map}.
\end{dfn}

\begin{rmk}
The standard definition \cite[Def.~8.4]{HWZbook} of a tame strong bundle requires, in addition to the conditions in Definition~\ref{dfn:tamestrongbundleepgroupoid}, that the structure map $\mu$ is a surjective local sc-diffeomorphism. However, this condition is automatically satisfied, as noted in \cite{MR3683060} on page 37.
\end{rmk}

Let $(P : E \rightarrow X,\mu)$ be a strong bundle over the ep-groupoid $(X,{\bf X})$. A sc-smooth \emph{section functor} $\s$ \cite[Def.~8.7]{HWZbook} of $(P,\mu)$ is a sc-smooth section $\s : X \rightarrow E$ of the strong bundle $P : E \rightarrow X$ over the object $M$-polyfold $X$ satisfying the following compatibility with $\mu$: For all morphisms $\phi \in {\bf X}$,
\begin{equation} \label{eq:sectionfunctorcompatibility}
\s(t(\phi)) = \mu(\phi,\s(s(\phi)))
\end{equation}
holds.

\begin{dfn} \label{dfn:scfredholmepgroupoid} \cite[Def.~8.7]{HWZbook}
A {\bf sc-Fredholm section functor} of a strong bundle $(P : E \rightarrow X,\mu)$ over an ep-groupoid $(X, {\bf X})$ is a sc-smooth section functor $\s : X \rightarrow E$ such that, when viewed as a section of the strong bundle $M$-polyfold bundle $P$, it is sc-Fredholm in the $M$-polyfold sense (Definition~\ref{dfn:scfredholmsection}).
\end{dfn}

We introduce the following new class of sc-Fredholm section functors.

\begin{dfn} \label{dfn:tamescfredholmepgroupoid}
A {\bf tame sc-Fredholm section functor} of a strong bundle $(P : E \rightarrow X,\mu)$ over an ep-groupoid $\cX = (X, {\bf X})$ is a sc-Fredholm section functor such that, when viewed as a section of the strong bundle $M$-polyfold bundle $P$, it is tame sc-Fredholm in the $M$-polyfold sense (Definition~\ref{dfn:tamescfredholmsection}).
\end{dfn}

We now generalize Theorem~\ref{thm:slicestructure} to the ep-groupoid case.

\begin{cor} \label{cor:slicestructureepgroupoid}
\text{}
\begin{enumilistfiber}
\item Consider a tame ep-groupoid $\cX = (X, {\bf X})$ and a slice $\tilde{X} \subset X$ of the object $M$-polyfold $X$, in the sense of Definition~\ref{dfn:globalslice}. Assume that $\tilde{X}$ is closed under morphisms, i.e.\ for all $\phi \in {\bf X}$ we have
\begin{equation} \label{eq:sliceclosedundermorphisms}
s(\phi) \in \tilde{X} \iff t(\phi) \in \tilde{X},
\end{equation}
or equivalently,
\begin{equation*} \label{eq:sliceclosedundermorphismsequivalent}
s^{-1}(\tilde{X}) = t^{-1}(\tilde{X}).
\end{equation*}

Then, denoting the subset of morphisms ${\bf X}$ that have source and target in $\tilde{X}$ by
\begin{equation} \label{eq:defmorphismsslice}
\tilde{\bf X} := s^{-1}(\tilde{X}) = t^{-1}(\tilde{X}),
\end{equation}
the tuple $\tilde{\cX} = (\tilde{X}, \tilde{\bf X})$ is a tame ep-groupoid with the tame $M$-polyfold atlas on $\tilde{X}$ induced by the sliced charts (Definition~\ref{dfn:slicecharts}) with respect to $\tilde{X} \subset X$. Note that $\tilde{\cX}$ is the full subcategory of $\cX$ with object space $\tilde{X}$.

For $x \in \tilde{X}_1$, the codimension $codim_x(\tilde{X} \subset X)$ (Definition~\ref{dfn:globalslice}) is well-defined and locally constant in $\tilde{X}$, i.e.\ it  equals $codim_{x'}(\tilde{X} \subset X)$ for every $x'$ in an open neighborhood of $x$ in $\tilde{X}$. For $x \in \tilde{X}_{\infty}$, the degeneracy index satisfies $d_{\tilde{X}}(x) = d_{X}(x)$.

\item Consider, in addition, a tame strong bundle  $(P: E \rightarrow X, \mu)$ over $\cX$ and suppose that $\tilde{X}$ is a slice of the bundle $P$ in the sense of Definition~\ref{dfn:globalslice}.

Then, the tuple $(\tilde{P}, \tilde{\mu})$ consisting of the restrictions
\begin{align*}
\tilde{E} &:= P^{-1}(\tilde{X})\\
\tilde{P} &:= P|_{\tilde{E}} : \tilde{E} \rightarrow \tilde{X},\\
\tilde{\mu} &:= \mu|_{\tilde{{\bf X}} \tensor[_s]{\times}{_P} \tilde{E}} : \tilde{{\bf X}} \tensor[_s]{\times}{_P} \tilde{E} \rightarrow \tilde{E},
\end{align*}
is a tame strong bundle over $\tilde{\cX}$, where the bundle atlas for $\tilde{P}$ is induced by the sliced bundle charts for $P$ with respect to $\tilde{X} \subset X$.\\

\item Consider, in addition, a sc-Fredholm section functor $\s : X \rightarrow E$ of $(P,\mu)$ and suppose that $\tilde{X}$ is a slice of $\s$ in the sense of Definition~\ref{dfn:globalslice}.

Then, the restriction $\tilde{\s} := \s|_{\tilde{X}} : \tilde{X} \rightarrow \tilde{E}$ is a tame sc-Fredholm section functor of the bundle $(\tilde{P},\tilde{\mu})$ with tame sc-Fredholm charts induced by the sliced sc-Fredholm charts for $\s$ with respect to $\tilde{X} \subset X$. For $x \in \tilde{X}_{\infty}$, the index satisfies $ind_x(\tilde{\s}) = ind_x(\s) - codim_x(\tilde{X} \subset X)$. If $|\s^{-1}(0)|$ is compact and $|\tilde{X}_{\infty}| \subset |X_{\infty}|$ is closed, then $|\tilde{\s}^{-1}(0)|$ is compact.
\end{enumilistfiber}
\end{cor}
\begin{proof}
\ul{Proof of (I)}: First we claim that $\tilde{X} = (\tilde{X}, \tilde{\bf X})$ is a groupoid. Given any subset $A$ of the object set of a groupoid, say $A \subset X$, we obtain a sub-groupoid $(A, {\bf A})$ of $(X, {\bf X})$ by defining the morphism set to be ${\bf A} := s^{-1}(A) \cap t^{-1}(A)$. Hence, with the set of morphisms $\tilde{\bf X}$ as defined in \eqref{eq:defmorphismsslice}, the tuple $(\tilde{X}, \tilde{\bf X})$ is a groupoid.

Since $\tilde{X} \subset X$ is a slice of the tame $M$-polyfold $X$, Theorem~\ref{thm:slicestructure}(I) provides a tame $M$-polyfold structure on $\tilde{X}$ with the claimed degeneracy index for $x \in \tilde{X}_{\infty}$ and the claimed locally constant codimension $codim_x(\tilde{X} \subset X)$ for $x \in \tilde{X}_1$.

We now equip $\tilde{\bf X}$ with a tame $M$-polyfold structure by pulling back the tame $M$-polyfold charts on $\tilde{X}$ through the source and target maps $s, t : {\bf X} \rightarrow X$ on $\cX$. That is, since $s$ and $t$ are local sc-diffeomorphisms by the \'etale property (Definition~\ref{dfn:tameepgroupoid}(i)) of $\cX$, for every $z \in \tilde{\bf X}$ there exists a neighborhood $\mathbold{V} \subset {\bf X}$ of $z$ such that $s|_{\mathbold{V}} : \mathbold{V} \rightarrow V$ is a sc-diffeomorphism, where $V := s(\mathbold{V})$. Then $V$ is a neighborhood of $s(z)$ in $X$. Since $\tilde{X} \subset X$ is a slice and $s(z) \in \tilde{X}$, after shrinking $V$ we can assume that it is the domain of a 
sliced chart with respect to $\tilde{X} \subset X$ (Definition~\ref{dfn:slicecharts}). This induces a tame $M$-polyfold chart on $\tilde{X}$ with domain $V \cap \tilde{X}$, as in the proof of Theorem~\ref{thm:slicestructure}. Pulling this back through $s$, we obtain a sliced chart with respect to $\tilde{\bf X} \subset {\bf X}$ with domain $\mathbold{V}$, and since $s({\bf V} \cap \tilde{\bf X}) = V \cap \tilde{X}$, its induced tame $M$-polyfold chart has domain $\mathbold{V} \cap \tilde{\bf X}$. So, we have constructed a tame $M$-polyfold chart with domain an open neighborhood of $z$ in $\tilde{\bf X}$. The analogous construction using the target map $t$ instead of the source map $s$ produces another such chart. On overlaps between any sliced charts with respect to $\tilde{\bf X} \subset {\bf X}$ produced in this way, the transition maps are sc-smooth because they are compositions of the sc-smooth chart maps on $X$ with the sc-smooth source and target maps (and their local inverses). Hence the transitions between the induced tame $M$-polyfold charts on $\tilde{\bf X}$ are also sc-smooth, as restrictions of the transitions on ${\bf X}$. Hence covering $\tilde{\bf X}$ with these $M$-polyfold charts provides the claimed $M$-polyfold structure. Moreover, observe from the construction of the charts that the source $\tilde{s} = s|_{\tilde{\bf X}}$ and target $\tilde{t} = t|_{\tilde{\bf X}}$ maps on $\tilde{\cX} = (\tilde{X}, \tilde{{\bf X}})$ are local sc-diffeomorphisms, i.e.\ $\tilde{\cX}$ satisfies the \'etale property (Definition~\ref{dfn:tameepgroupoid}(i)). Similarly, the unit map, the inverse map, and the multiplication map on $\tilde{\cX}$ are sc-smooth since they are restrictions of the corresponding maps on $\cX$. This verifies properties Definition~\ref{dfn:tameepgroupoid}(ii) and Definition~\ref{dfn:tameepgroupoid}(iv) for $\tilde{\cX}$.

To verify that $\tilde{\cX}$ is a tame ep-groupoid, it remains to verify properness (Definition~\ref{dfn:tameepgroupoid}(iii)). Let $x \in \tilde{X}$. By properness of $\cX$, there exists an open neighborhood $V(x)$ of $x$ in $X$ such that the mapping $t : s^{-1}(cl_X(V(x))) \rightarrow X$ is proper. Shrink $V(x)$ so that it is the domain of a sliced chart (Definition~\ref{dfn:slicecharts}) with respect to $\tilde{X} \subset X$. Then ${V(x) \cap \tilde{X}}$ is closed in $V(x)$
because the chart map that homeomorphically identifies $V(x)$ with the sliced sc-retract $(\cO, \mathbb{R}^n \times C, \mathbb{R}^n \times \mathbb{E})$ sends $V(x) \cap \tilde{X}$ to the induced tame sc-retract $\cO \cap (\{0\} \times C)$, which is closed in $\cO$. Let $V'(x) \subset V(x)$ be a smaller open neighborhood of $x$ such that $cl_X(V'(x)) \subset V(x)$. Then we have $$cl_{X}(V'(x) \cap \tilde{X}) = cl_{V(x)}(V'(x) \cap \tilde{X}) \subset cl_{V(x)}(V(x) \cap \tilde{X}) = V(x) \cap \tilde{X}.$$ In particular, the closure of $V'(x) \cap \tilde{X}$ in $\tilde{X}$ and in $X$ agree, i.e.
\begin{equation} \label{eq:closuresagree}
cl_{X}(V'(x) \cap \tilde{X}) = cl_{\tilde{X}}(V'(x) \cap \tilde{X}).
\end{equation}
We claim that the open neighborhood $V'(x) \cap \tilde{X}$ of $x$ in $\tilde{X}$ satisfies the required condition, i.e.\ the map $t|_{\tilde{\bf X}} : s|_{\tilde{\bf X}}^{-1}(cl_{\tilde{X}}(V'(x) \cap \tilde{X})) \rightarrow \tilde{X}$ is proper. Let $K \subset \tilde{X}$ be compact. We must show that the set ${A := t|_{\tilde{\bf X}}^{-1}(K) \cap s|_{\tilde{\bf X}}^{-1}(cl_{\tilde{X}}(V'(x) \cap \tilde{X}))}$ is compact. We have
$$A = t^{-1}(K) \cap s^{-1}(cl_{X}(V'(x) \cap \tilde{X}))$$
by \eqref{eq:closuresagree} and since by definition $\tilde{\bf X} = t^{-1}(\tilde{X}) = s^{-1}(\tilde{X})$. Observe that $A$ is closed in ${\bf X}$ since $K \subset X$ is closed as a compact subset of the Hausdorff space $X$. By our choice of $V(x)$, the set $t^{-1}(K) \cap s^{-1}(cl_X(V(x)))$ is compact. Hence $A$ is compact as a closed subset of the compact set $t^{-1}(K) \cap s^{-1}(cl_X(V(x)))$. This completes the proof that $\tilde{\cX}$ is a tame ep-groupoid.
\\

\ul{Proof of (II)}: Since $\tilde{X}$ is a slice of the tame strong bundle $P : E \rightarrow X$ in the $M$-polyfold sense, Theorem~\ref{thm:slicestructure}(II) provides a tame strong bundle structure on the restriction $\tilde{P}$. To prove that the tuple $(\tilde{P}, \tilde{\mu})$ is a tame strong bundle over the ep-groupoid $\tilde{\cX}$, we must show that the map $\tilde{\mu}$ has the required properties. Recall from the lemma statement that $\tilde{\mu}$ is the restriction of $\mu$, i.e.
$$\tilde{\mu} := \mu|_{\tilde{{\bf X}} \tensor[_s]{\times}{_P} \tilde{E}} : \tilde{{\bf X}} \tensor[_s]{\times}{_P} \tilde{E} \rightarrow \tilde{E}.$$
First of all, $\tilde{\mu}$ indeed takes values in $\tilde{E}$ because, since $\mu$ covers $t$, for any $(\phi,e) \in \tilde{{\bf X}} \tensor[_s]{\times}{_P} \tilde{E}$ we have $P \circ \mu(\phi,e) = t \circ \pi_1(\phi,e) = t(\phi) \in \tilde{X}$ and hence $\mu(\phi,e) \in P^{-1}(\tilde{X}) = \tilde{E}$. Moreover, $\tilde{\mu}$ covers the target map on $\tilde{\cX}$ since $\tilde{\mu}$ is the restriction of $\mu$ which covers the target map $t$ on $\cX$. Finally, the required properties Definition~\ref{dfn:tamestrongbundleepgroupoid}(i)(ii) of $\tilde{\mu}$ follow immediately from the corresponding properties of $\mu$, since $\tilde{\mu}$ is the restriction of $\mu$. This completes the proof that $(\tilde{P},\tilde{\mu})$ is a tame strong bundle over $\tilde{\cX}$.
\\

\ul{Proof of (III)}: Since $\tilde{X}$ is a slice of the sc-Fredholm section $$\s : X \rightarrow E$$ in the $M$-polyfold sense, Theorem~\ref{thm:slicestructure}(III) provides the structure of a tame sc-Fredholm section on the restricted section $$\tilde{\s} := \s|_{\tilde{X}} : \tilde{X} \rightarrow \tilde{E}$$ with the claimed sc-Fredholm index. Moreover, the section $\tilde{\s}$ is a section functor because the required compatibility \eqref{eq:sectionfunctorcompatibility} with $\tilde{\mu}$ is immediate from the compatibility of $\s$ with $\mu$. Hence $\tilde{\s}$ is a tame sc-Fredholm section functor.

The final statement about compactness of the zero sets holds because $|\tilde{\s}^{-1}(0)| = |\s^{-1}(0)| \cap |\tilde{X}_{\infty}|$ is the intersection of a compact set and a closed subset of $|X_{\infty}|$.
\end{proof}

The generalization of Theorem~\ref{thm:globalimplicit} to the ep-groupoid case now easily follows by combining Theorem~\ref{thm:globalimplicit} with Corollary~\ref{cor:slicestructureepgroupoid}.

\begin{cor}\label{cor:globalimplicitepgroupoid}
\text{}
\begin{enumilistfiber}
\item Consider a tame ep-groupoid $\cX = (X, {\bf X})$, a smooth manifold $Y$ together with a codimension-$n$ submanifold $N \subset Y$, and a sc-smooth map $f : X \rightarrow Y$ that satisfies the compatibility with morphisms
\begin{equation} \label{eq:mapcompatibilitymorphisms}
f(s(\phi)) = f(t(\phi)) \text{ for all } \phi \in {\bf X}.
\end{equation}
Assume that $f$ is transverse to $N$ (Definition~\ref{dfn:transverse}).

Then, there exists an open neighborhood
$$\tilde{X} \subset f^{-1}(N) \cap X_1$$
of $f^{-1}(N) \cap X_{\infty}$ such that $\tilde{X}$ is a slice of $X^1$ in the $M$-polyfold sense (Definition~\ref{dfn:globalslice}) satisfying $codim_x(\tilde{X} \subset X^1) = n$ for every $x \in \tilde{X}_1 = \tilde{X} \cap X_2$. In particular, the full subcategory $\tilde{\cX} = (\tilde{X}, \tilde{\bf X})$ of $\cX^1$ with object space $\tilde{X}$ is a tame ep-groupoid with degeneracy index satisfying $d_{\tilde{X}}(x) = d_{X}(x)$ for all $x \in \tilde{X}_{\infty}$.
\\

\item Consider, in addition, a tame strong bundle $(P : E \rightarrow X, \mu)$ over $\cX$. Then, there exists a possibly smaller neighborhood $\tilde{X}$ in (I) that is a slice of the bundle $P|_{E^1} : E^1 \rightarrow X^1$ in the $M$-polyfold sense (Definition~\ref{dfn:globalslice}). In particular, the tuple $(\tilde{P}, \tilde{\mu})$ consisting of the restrictions
\begin{align*}
\tilde{E} &:= P|_{E^1}^{-1}(\tilde{X})\\
\tilde{P} &:= P|_{\tilde{E}} : \tilde{E} \rightarrow \tilde{X},\\
\tilde{\mu} &:= \mu|_{\tilde{{\bf X}} \tensor[_s]{\times}{_P} \tilde{E}} : \tilde{{\bf X}} \tensor[_s]{\times}{_P} \tilde{E} \rightarrow \tilde{E},
\end{align*}
is a tame strong bundle over $\tilde{\cX}$.
\\

\item Consider, in addition, a tame sc-Fredholm section functor $$\s : X \rightarrow E$$ of $(P, \mu)$. Assume that $f$ is $\s$-compatibly transverse to $N$ (Definition~\ref{dfn:transverse}).

Then, there exists a possibly smaller neighborhood $\tilde{X}$ in (II) that is a slice of the tame sc-Fredholm section $\s|_{X^1} : X^1 \rightarrow E^1$ in the $M$-polyfold sense (Definition~\ref{dfn:globalslice}). In particular, the restriction
$$\tilde{\s} := \s|_{\tilde{X}} : \tilde{X} \rightarrow \tilde{E}$$
is a tame sc-Fredholm section functor of the bundle $(\tilde{P},\tilde{\mu})$ with index satisfying $ind_x(\tilde{\s}) = ind_x(\s) - n$ for all $x \in \tilde{X}_{\infty}$. If $N$ is closed as a subset of $Y$ and $|\s^{-1}(0)|$ is compact, then $|\tilde{\s}^{-1}(0)|$ is compact.
\end{enumilistfiber}
\end{cor}
\begin{proof}
We first prove the statements in (I). Since $f$ is transverse to $N$, Theorem~\ref{thm:globalimplicit}(I) provides an open neighborhood $\tilde{X} \subset f^{-1}(N) \cap X_1$ of $f^{-1}(N) \cap X_{\infty}$ that is a slice of $X^1$ with the claimed codimension at every point in $\tilde{X}_1$. The compatibility \eqref{eq:mapcompatibilitymorphisms} of $f$ with the morphisms ${\bf X}$ implies that $s^{-1}(\tilde{X}) = t^{-1}(\tilde{X})$ holds. Hence applying Corollary~\ref{cor:slicestructureepgroupoid}(I) to the tame ep-groupoid $\cX^1 = (X^1, {\bf X}^1)$ and the slice $\tilde{X} \subset X^1$, we conclude that the full subcategory $\tilde{\cX} = (\tilde{X}, \tilde{\bf X})$ of $\cX^1$ with object space $\tilde{X}$ is a tame ep-groupoid with the claimed degeneracy index.

We now prove the statements in (II). Theorem~\ref{thm:globalimplicit}(II) provides a choice of $\tilde{X}$ in (I) that is in addition a slice of the $1$-shifted bundle $P|_{E^1} : E^1 \rightarrow X^1$. Hence applying Corollary~\ref{cor:slicestructureepgroupoid}(II), we conclude that the tuple $(\tilde{P},\tilde{\mu})$ is indeed a tame strong bundle over $\tilde{\cX}$, as claimed.

We now prove the statements in (III). Since $f$ is $\s$-compatibly transverse to $N$, Theorem~\ref{thm:globalimplicit}(III) provides a choice of $\tilde{X}$ in (II) that is in addition a slice of the tame sc-Fredholm section $\s|_{X^1} : X^1 \rightarrow E^1$. Hence applying Corollary~\ref{cor:slicestructureepgroupoid}(III), we conclude that the restriction $\tilde{\s} = \s|_{\tilde{X}}$ is indeed a tame sc-Fredholm section functor of the bundle $(\tilde{P},\tilde{\mu})$ with the claimed index.

It remains to prove the final statement about compactness of the orbit spaces of the zero sets. Assume that $|\s^{-1}(0)|$ is compact. It suffices to show that the inclusion of orbit spaces $$|\tilde{X}_{\infty}| \subset |(X^1)_{\infty}| = |X_{\infty}|$$ is closed, because then Corollary~\ref{cor:slicestructureepgroupoid}(III) implies compactness of the space $|\tilde{\s}^{-1}(0)|$, as required. Notice that, by \eqref{eq:mapcompatibilitymorphisms}, the map $f$ descends to a map on the orbit space $|f| : |X_{\infty}| \rightarrow Y$. Then $|\tilde{X}_{\infty}| = |f|^{-1}(N)$ is closed in $|X_{\infty}|$ by continuity of $|f|$, as required.
\end{proof}

\section{Fiber products of tame sc-Fredholm sections} \label{sec:fiberproducts}

The main result of this section is the construction of fiber products of tame sc-Fredholm section functors (Corollary~\ref{cor:ep-fiberFredholm}). This result is used in the application described in Section~\ref{subsub:pss}. It is a corollary of the construction of restrictions of tame sc-Fredholm section functors to transverse preimages of sc-smooth maps, which is the main result of this paper; see Theorem~\ref{thm:globalimplicit} for the $M$-polyfold case and Corollary~\ref{cor:globalimplicitepgroupoid} for the ep-groupoid generalization.

In this section, we index the $M$-polyfolds and ep-groupoids with parenthesis around the subscript, i.e.\ $\cB_{(i)}$ for $i = 1,2,$ to avoid confusion with the standard notation $\cB_m$ for the $m$-level of an $M$-polyfold $\cB$.

We first describe the Cartesian product of tame sc-Fredholm sections over $M$-polyfolds.

\begin{lem} \label{lem:M-CartesianPolyfold} \text{}
\begin{enumilistfiber}
\item Consider tame $M$-polyfolds $\cB_{(i)}$ for $i = 1,2$. Then, the Cartesian product $$\cB_{(1)} \times \cB_{(2)}$$ is a tame $M$-polyfold with charts given by products of charts on the factors, and with degeneracy index satisfying $$d_{\cB_{(1)} \times \cB_{(2)}}(x_1,x_2) = d_{\cB_{(1)}}(x_1) + d_{\cB_{(2)}}(x_2)$$ for all $(x_1,x_2) \in (\cB_{(1)})_{\infty} \times (\cB_{(2)})_{\infty}$.\\

\item Consider tame strong bundles $\rho_i : \cE_{(i)} \rightarrow \cB_{(i)}$ over $\cB_{(i)}$ for $i = 1,2.$ Then, the product map $\rho_1 \times \rho_2 : \cE_{(1)} \times \cE_{(2)} \rightarrow \cB_{(1)} \times \cB_{(2)}$ is a tame strong bundle over $\cB_{(1)} \times \cB_{(2)}$ with bundle charts given by products of bundle charts on the factors.\\

\item Consider tame sc-Fredholm sections $\s_i : \cB_{(i)} \rightarrow \cE_{(i)}$ (Definition~\ref{dfn:tamescfredholmsection}) of $\rho_i$ for $i = 1,2.$ Then, the product map $$\s_1 \times \s_2 : \cB_{(1)} \times \cB_{(2)} \rightarrow \cE_{(1)} \times \cE_{(2)}$$ is a tame sc-Fredholm section of $\rho_1 \times \rho_2$ with tame sc-Fredholm charts given by products of tame sc-Fredholm charts on the factors (after reordering factors in the charts), and with index satisfying $$ind_{(x_1,x_2)}(\s_1 \times \s_2) = ind_{x_1}(\s_1) + ind_{x_2}(\s_2)$$ for all $(x_1,x_2) \in (\cB_{(1)})_{\infty} \times (\cB_{(2)})_{\infty}$. If $\s_i^{-1}(0)$ is compact for $i = 1,2,$ then $(\s_1 \times \s_2)^{-1}(0)$ is compact.
\end{enumilistfiber}
\end{lem}
\begin{proof}
\ul{Proof of (I)}: We first show that products of tame $M$-polyfold charts (Definition~\ref{dfn:tameMpolyfold}) on the factors $\cB_{(i)}, i = 1,2,$ are indeed tame $M$-polyfold charts on the product ${\cB_{(1)} \times \cB_{(2)}}$, as claimed. For $i = 1,2$, consider a tame $M$-polyfold chart
\begin{equation} \label{eq:tamechartcartesian}
(V_i,\varphi_i,(\cO_i,C_i,\mathbb{E}_i))
\end{equation}
on $\cB_{(i)}$, an open subset $U_i \subset C_i$, and a tame sc-retraction $r_i : U_i \rightarrow U_i$ with image $\cO_i = r_i(U_i)$. We claim that the tuple
\begin{equation} \label{eq:cartesianproductchart}
(V_1 \times V_2, \varphi_1 \times \varphi_2,(\cO_1 \times \cO_2, C_1 \times C_2, \mathbb{E}_1 \times \mathbb{E}_2))
\end{equation}
is a tame $M$-polyfold chart on $\cB_{(1)} \times \cB_{(2)}$. We equip $\cB_{(1)} \times \cB_{(2)}$ with the product topology on every level. Recall that the product sc-structure is given by $(\mathbb{E}_1 \times \mathbb{E}_2)_m = (\mathbb{E}_1)_m \times (\mathbb{E}_2)_m$ for all $m \geq 0$. Then the product set $V_1 \times V_2$ is open in $\cB_{(1)} \times \cB_{(2)}$ and the map $\varphi_1 \times \varphi_2 : V_1 \times V_2 \rightarrow \cO_1 \times \cO_2$ is a homeomorphism. It remains to show that $\cO_1 \times \cO_2$ is a tame sc-retract. The product map $r_1 \times r_2 : U_1 \times U_2 \rightarrow U_1 \times U_2$ is a sc-smooth retraction with image $\cO_1 \times \cO_2$. We now show that $r_1 \times r_2$ is tame. To check Definition~\ref{dfn:tameretraction}(1) for $r_1 \times r_2$, let $(x_1 \times x_2) \in U_1 \times U_2$ and then compute using the corresponding property of each $r_i$ together with \eqref{eq:degcartesian} that the degeneracy index satisfies
\begin{align*}
d_{C_1 \times C_2}((r_1 \times r_2)(x_1,x_2)) &= d_{C_1 \times C_2}(r_1(x_1),r_2(x_2))\\
&= d_{C_1}(r_1(x_1)) + d_{C_2}(r_2(x_2))\\
&= d_{C_1}(x_1) + d_{C_2}(x_2)\\
&= d_{C_1 \times C_2}(x_1,x_2),
\end{align*}
as required. To check Definition~\ref{dfn:tameretraction}(2) for $r_1 \times r_2$, let
$$(x_1,x_2) \in (\cO_1 \times \cO_2)_{\infty} = (\cO_1)_{\infty} \times (\cO_2)_{\infty}$$
be a smooth point. Then by the corresponding property of each $r_i$, there exist sc-subspaces $A_i \subset \mathbb{E}_i$ such that $\mathbb{E}_i = T_{x_i}\cO_i \oplus A_i$ and such that $A_i \subset (\mathbb{E}_i)_{x_i}$, where $(\mathbb{E}_i)_{x_i} \subset \mathbb{E}_i$ is the sc-subspace \eqref{eq:Ex}. From the definition \eqref{eq:Ex} we have
\begin{equation}\label{eq:Exsplittingcartesian}
(\mathbb{E}_1)_{x_1} \times (\mathbb{E}_2)_{x_2} = (\mathbb{E}_1 \times \mathbb{E}_2)_{(x_1,x_2)}.
\end{equation}
So we have $A_1 \times A_2 \subset (\mathbb{E}_1 \times \mathbb{E}_2)_{(x_1,x_2)}$. Moreover, we have the sc-splitting
\begin{align*}
\mathbb{E}_1 \times \mathbb{E}_2 &= (T_{x_1}\cO_1 \oplus A_1) \times (T_{x_2}\cO_2 \oplus A_2)\\
&= T_{(x_1,x_2)}(\cO_1 \times \cO_2) \oplus (A_1 \times A_2)
\end{align*}
because, by definition \eqref{eq:tangentspacedef} of tangent space to a sc-retract, the tangent space to a product satisfies
\begin{align}
T_{(x_1,x_2)}(\cO_1 \times \cO_2) &= D_{(x_1,x_2)}(r_1 \times r_2)(T_{(x_1,x_2)}(U_1 \times U_2)) \label{eq:producttangentspacesplit}\\
&= D_{x_1}r_1(T_{x_1}U_1) \times  D_{x_2}r_2(T_{x_2}U_2) \nonumber\\
&= T_{x_1}\cO_1 \times T_{x_2}\cO_2. \nonumber
\end{align}
This completes the proof that $r_1 \times r_2$ is a tame sc-retraction and hence $\cO_1 \times \cO_2$ is a tame sc-retract.

We have shown that the product chart \eqref{eq:cartesianproductchart} is indeed a tame $M$-polyfold chart on $\cB_{(1)} \times \cB_{(2)}$. Since product charts of this form cover $\cB_{(1)} \times \cB_{(2)}$ and have sc-smooth transition maps due to the transitions on each factor being sc-smooth, the collection of these product charts forms a tame atlas on $\cB_{(1)} \times \cB_{(2)}$. Then $\cB_{(1)} \times \cB_{(2)}$ equipped with this atlas is a tame $M$-polyfold.

We now prove the claimed degeneracy index formula. Consider any $(x_1,x_2) \in (\cB_{(1)})_{\infty} \times (\cB_{(2)})_{\infty}$. From the splitting \eqref{eq:producttangentspacesplit} of the tangent spaces in the product retract, we conclude $$T_{(x_1,x_2)}(\cB_{(1)} \times \cB_{(2)}) \cong T_{x_1}\cB_{(1)} \times T_{x_2}\cB_{(2)}.$$ Furthermore, by definition of reduced tangent space in a retract \eqref{eq:reducedtangentspacedef} together with \eqref{eq:Exsplittingcartesian} and \eqref{eq:producttangentspacesplit}, we conclude $$T_{(\varphi_1(x_1),\varphi_2(x_2))}^R(\cO_1 \times \cO_2) =  T_{\varphi_1(x_1)}^R\cO_1 \times T_{\varphi_2(x_2)}^R\cO_2,$$ which implies by the global definition of reduced tangent space \eqref{eq:defglobalredcuedtangent} that $$T_{(x_1,x_2)}^R(\cB_{(1)} \times \cB_{(2)}) \cong T^R_{x_1}\cB_{(1)} \times T^R_{x_2}\cB_{(2)}$$ by examining any product chart. Hence, by the global description \eqref{eq:degindexdef} of degeneracy index at smooth points, we have
\begin{align*}
d_{\cB_{(1)} \times \cB_{(2)}}(x_1,x_2) &= \dim \big (T_{(x_1,x_2)}(\cB_{(1)} \times \cB_{(2)})/T^R_{(x_1,x_2)}(\cB_{(1)} \times \cB_{(2)}) \big )\\
&= \dim(T_{x_1}\cB_{(1)}/T^R_{x_1}\cB_{(1)}) + \dim(T_{x_2}\cB_{(2)}/T^R_{x_2}\cB_{(2)})\\
&= d_{\cB_{(1)}}(x_1) + d_{\cB_{(2)}}(x_2).
\end{align*}
This completes the proof of (I).
\\

\ul{Proof of (II)}: For $i = 1,2,$ consider a tame $M$-polyfold chart \eqref{eq:tamechartcartesian} on $\cB_{(i)}$ covered by a tame strong bundle chart (Definition~\ref{dfn:tamebundle})
\begin{equation} \label{eq:tamebundlechartcartesian}
(\rho_i^{-1}(V_i), \Phi_i, (K_i, C_i \lhd \mathbb{F}_i, \mathbb{E}_i \lhd \mathbb{F}_i))
\end{equation}
for $\rho_i$. We now construct a tame strong bundle chart for $\rho_1 \times \rho_2$ over the open set $V_1 \times V_2$. For $i = 1,2,$ consider a tame strong bundle retraction \eqref{eq:bundleretractionwrittentout}
\begin{align*}
R_i : U_i \lhd \mathbb{F}_i &\rightarrow U_i \lhd \mathbb{F}_i\\
(x,\xi) &\mapsto (r_i(x),\Gamma_i(x,\xi))
\end{align*}
with image $K_i = R_i(U_i \lhd \mathbb{F}_i)$, where
$$\G_i : U_i \lhd \mathbb{F} \rightarrow \mathbb{F}$$
is a linear projection $\G_i(x,\cdot)$ for every $x \in U_i$. Denote the reordering of factors map by
$$l : (\mathbb{E}_1 \times \mathbb{F}_1) \times (\mathbb{E}_2 \times \mathbb{F}_2) \rightarrow (\mathbb{E}_1 \times \mathbb{E}_2) \times (\mathbb{F}_1 \times \mathbb{F}_2),$$
and note that $l$ is a linear sc-isomorphism. We claim that the tuple
\begin{align} \label{eq:cartesianproductbundlechart}
((\rho_1 \times &\rho_2)^{-1}(V_1 \times V_2), l  \circ (\Phi_1 \times \Phi_2),\\
&(l(K_1 \times K_2), C_1 \times C_2 \lhd \mathbb{F}_1 \times \mathbb{F}_2, \mathbb{E}_1 \times \mathbb{E}_2 \lhd \mathbb{F}_1 \times \mathbb{F}_2)) \nonumber
\end{align}
is a tame strong bundle chart for $\rho_1 \times \rho_2$ covering the product $M$-polyfold chart \eqref{eq:cartesianproductchart}. Since for $i = 1,2,$ the map $\Phi_i : \rho_i^{-1}(V_i) \rightarrow K_i$ is a homeomorphism covering $\varphi_i$ that is linear on fibers, it follows that the map $l \circ (\Phi_1 \times \Phi_2)$ is a homeomorphism from $(\rho_1 \times \rho_2)^{-1}(V_1 \times V_2)$ to $l(K_1 \times K_2)$ covering $\varphi_1 \times \varphi_2$ and is linear on fibers. Moreover, the set $l(K_1 \times K_2)$ is a tame strong bundle retract because it is the image of the tame strong bundle retraction
\begin{align*}
l \circ (R_1 \times R_2) \circ l^{-1} : (U_1 \times U_2) &\lhd (\mathbb{F}_1 \times \mathbb{F}_2) \rightarrow (U_1 \times U_2) \lhd (\mathbb{F}_1 \times \mathbb{F}_2)\\
(x_1,x_2,\xi_1,\xi_2) &\mapsto (r_1(x_1),r_2(x_2),\Gamma_1(x_1,\xi_1),\Gamma_2(x_2,\xi_2)).
\end{align*}
We have shown that the chart \eqref{eq:cartesianproductbundlechart} is indeed a tame strong bundle chart for $\rho_1 \times \rho_2$. Sc-smoothness of bundle transitions between charts constructed in this way follows from sc-smoothness in each factor. Hence the collection of these charts forms a tame bundle atlas for $\rho_1 \times \rho_2$. This completes the proof of (II).
\\

\ul{Proof of (III)}: To prove that $\s_1 \times \s_2$ is a tame sc-Fredholm section (Definition~\ref{dfn:tamescfredholmsection}) of $\rho_1 \times \rho_2$, first observe that it is sc-smooth and regularizing because $\s_1$ and $\s_2$ are sc-smooth and regularizing, and the double filtration satisfies $(\cB_{(1)} \times \cB_{(2)})_m = (\cB_{(1)})_m \times (\cB_{(2)})_m$ and $(\cE_{(1)} \times \cE_{(2)})_{m,m+1} = (\cE_{(1)})_{m,m+1} \times (\cE_{(2)})_{m,m+1}$ for all $ m \geq 0$ (see \eqref{eq:doublefiltrationglobal}).

Consider a smooth point $(x_1,x_2) \in (\cB_{(1)})_{\infty} \times (\cB_{(2)})_{\infty}$. Then since each $\s_i$ is tame sc-Fredholm there exist tame bundle charts as in \eqref{eq:tamebundlechartcartesian} that are in addition tame sc-Fredholm charts for $\s_i$ at $x_i$. First of all, this means that the partial quadrants $C_i \subset \mathbb{E}_i$ are in the standard form \eqref{eq:partialquadrantstandardform}, i.e.\ $C_i = [0,\infty)^{s_i} \times \mathbb{E}_i' \subset \mathbb{R}^{s_i} \times \mathbb{E}_i' = \mathbb{E}_i$. Now, we will prove that a chart similar to the tame bundle chart \eqref{eq:cartesianproductbundlechart} is a tame sc-Fredholm chart for $\s_1 \times \s_2$ at $(x_1,x_2)$. The only issue with \eqref{eq:cartesianproductbundlechart} is that the partial quadrant $C_1 \times C_2 \subset \mathbb{E}_1 \times \mathbb{E}_2$ is not in the standard form. This is easily remedied as follows. Consider the partial quadrant
$$C :=  [0,\infty)^{s_1 \times s_2} \times \mathbb{E}_1' \times \mathbb{E}_2'$$
of the sc-Banach space
$$\mathbb{E} := \mathbb{R}^{s_1 \times s_2} \times \mathbb{E}_1' \times \mathbb{E}_2'.$$
Then the reordering of factors map
$$f : \mathbb{R}^{s_1} \times \mathbb{E}'_1 \times \mathbb{R}^{s_2} \times \mathbb{E}'_2 \rightarrow \mathbb{R}^{s_1 + s_2} \times \mathbb{E}'_1 \times \mathbb{E}'_2$$
is a linear sc-isomorphism ${\mathbb{E}_1 \times \mathbb{E}_2 \rightarrow \mathbb{E}}$ that restricts to an isomorphism ${C_1 \times C_2 \rightarrow C}$. Applying $f$ to the sc-retract $\cO_1 \times \cO_2$ produces a tame sc-retract
$$\cO := f(\cO_1 \times \cO_2) \subset C$$
with tame sc-retraction
\begin{align*}
r &:= f \circ (r_1 \times r_2) \circ f^{-1} : U \rightarrow U,\\
U &:= f(U_1 \times U_2) \subset C,
\end{align*}
onto $r(U) = \cO$. Moreover, the map
$$\varphi := f \circ (\varphi_1 \times \varphi_2) : V_1 \times V_2 \rightarrow \cO$$
is a homeomorphism, and so from the tame $M$-polyfold chart \eqref{eq:cartesianproductchart} we obtain another tame $M$-polyfold chart
$$(V_1 \times V_2, \varphi, (\cO, C, \mathbb{E})).$$
Similarly, applying $f \times id_{\mathbb{F}_1 \times \mathbb{F}_2}$ to the tame bundle retract $l(K_1 \times K_2)$ produces another tame bundle retract
$$K := (f \times id_{\mathbb{F}_1 \times \mathbb{F}_2}) \circ l(K_1 \times K_2)$$
with tame bundle retraction
\begin{align*}
R &:= (f \times id_{\mathbb{F}_1 \times \mathbb{F}_2}) \circ l \circ (R_1 \times R_2) \circ l^{-1} \circ (f \times id_{\mathbb{F}_1 \times \mathbb{F}_2})^{-1}\\
&: U \lhd \mathbb{F}_1 \times \mathbb{F}_2 \rightarrow U \lhd \mathbb{F}_1 \times \mathbb{F}_2
\end{align*}
onto $R(U \lhd \mathbb{F}_1 \times \mathbb{F}_2) = K$. Consider the map
\begin{align*}
\G := (\G_1 \times \G_2) \circ l^{-1} \circ (f \times id_{\mathbb{F}_1 \times \mathbb{F}_2})^{-1},
\end{align*}
and observe that, for $(y, \xi_1,\xi_2) \in U \lhd \mathbb{F}_1 \times \mathbb{F}_2$, we have
$$R(y,\xi_1,\xi_2) = (r(y), \G(y,\xi_1,\xi_2)).$$
The map
$$\Phi := (f \times id_{\mathbb{F}_1 \times \mathbb{F}_2}) \circ l  \circ (\Phi_1 \times \Phi_2) : (\rho_1 \times \rho_2)^{-1}(V_1 \times V_2) \rightarrow K$$
is a homeomorphism that is linear on the fibers and covers $\varphi$. Hence from the tame bundle chart \eqref{eq:cartesianproductbundlechart} we obtain another tame bundle chart
\begin{equation} \label{eq:cartesianproductfredholmchart}
((\rho_1 \times \rho_2)^{-1}(V_1 \times V_2), \Phi,(K, C \lhd \mathbb{F}_1 \times \mathbb{F}_2, \mathbb{E} \lhd \mathbb{F}_1 \times \mathbb{F}_2)).
\end{equation}

We claim that the bundle chart \eqref{eq:cartesianproductfredholmchart} is a tame sc-Fredholm chart for $\s_1 \times \s_2$ at $(x_1,x_2)$. 
We must show that the section
\begin{equation} \label{eq:cartesiantamefredholmgerm}
\t := \Phi \circ (\s_1 \times \s_2) \circ \varphi^{-1} : \cO \rightarrow K
\end{equation}
is a tame sc-Fredholm germ. Since for each $i = 1,2,$ the section
$$\t_i := \Phi_i \circ \s_i \circ \varphi_i^{-1} : \cO_i \rightarrow K_i$$
is a tame sc-Fredholm germ, we can assume that there is a filling
$$h_i : U_i \rightarrow U_i \lhd \mathbb{F}_i$$
of $\mathbold{\t}_i \circ r_i : U_i \rightarrow \mathbb{F}_i$, where $\mathbold{\t}_i : \cO_i \rightarrow \mathbb{F}_i$ denotes the principal part of $\t_i$ in the fiber $\mathbb{F}_i$.
We claim that
$$h := (f \times id_{\mathbb{F}_1 \times \mathbb{F}_2}) \circ l \circ (h_1 \times h_2) \circ f^{-1} : U \rightarrow U \lhd (\mathbb{F}_1 \times \mathbb{F}_2)$$
is a filling of $\mathbold{\t} \circ r : U \rightarrow \mathbb{F}_1 \times \mathbb{F}_2.$ The required properties in Definition~\ref{dfn:localscFredgerm}(b).(i)-(iii) of a filling follow from the corresponding properties for each $h_i$, as we now verify. Given $y \in \cO$, we compute
\begin{align*}
\t(y) &= (f \times id_{\mathbb{F}_1 \times \mathbb{F}_2}) \circ l \circ (\Phi_1 \times \Phi_2) \circ (\s_1 \times \s_2) \circ (\varphi_1 \circ \varphi_2)^{-1} \circ f^{-1}(y)\\
&= (f \times id_{\mathbb{F}_1 \times \mathbb{F}_2}) \circ l \circ (\t_1 \times \t_2) \circ f^{-1}(y)\\
&= (f \times id_{\mathbb{F}_1 \times \mathbb{F}_2}) \circ l \circ (h_1 \times h_2) \circ f^{-1}(y)\\
&= h(y),
\end{align*}
which verifies property (i) for $\t$. To see property $(ii)$, let $y \in U$ and assume
$${\bf h}(y) = \G(r(y),{\bf h}(y)).$$ Then since the principal parts satisfy ${\bf h} = ({\bf h}_1 \times {\bf h}_2) \circ f^{-1}$, we compute that
\begin{align*}
({\bf h}_1 &\times {\bf h}_2) \circ f^{-1}(y) = \G(r(y),{\bf h}(y))\\
&= (\G_1 \times \G_2) \circ l^{-1}\\
&\,\,\,\,\,\,\circ (f \times id_{\mathbb{F}_1 \times \mathbb{F}_2})^{-1} \bigg (f \circ (r_1 \times r_2) \circ f^{-1}(y), ({\bf h}_1 \times {\bf h}_2) \circ f^{-1}(y) \bigg )\\
&= (\G_1 \times \G_2) \circ l^{-1}((r_1 \times r_2) \circ f^{-1}(y), ({\bf h}_1 \times {\bf h}_2) \circ f^{-1}(y)),
\end{align*}
which implies $f^{-1}(y) \in \cO_1 \times \cO_2$ by the corresponding property of $\t_1$ and $\t_2$. Hence $y \in \cO$, proving property (ii) for $\t$. We now verify $(iii)$ for $\t$. By the corresponding property for each $\t_i$, the linearization $D_0L_i$ at $0$ of
\begin{align*}
L_i : U_i &\rightarrow \mathbb{F}_i\\
y_i &\mapsto (id_{\mathbb{F}_i} - \Gamma_i(r_i(y_i),\cdot)) \circ {\bf h}_i(y_i)
\end{align*}
restricts to a linear sc-isomorphism from $\ker D_0r_i$ to $\ker \Gamma_i(0,\cdot)$. We must show that the linearization at $0$ of the map
\begin{align*}
L : U &\rightarrow \mathbb{F}_1 \times \mathbb{F}_2\\
y &\mapsto (id_{\mathbb{F}_1 \times \mathbb{F}_2} - \Gamma(r(y),\cdot)) \circ {\bf h}(y)
\end{align*}
restricts to a linear sc-isomorphism from $\ker D_0r$ to $\ker \G(0,\cdot)$. This follows from the observations that $L = (L_1 \times L_2) \circ f^{-1}$ holds, the map $f^{-1}$ is a linear sc-isomorphism that restricts to an isomorphism
$$\ker D_0r \rightarrow \ker D_0(f \circ (r_1 \times r_2)) = \ker(D_0r_1) \times \ker(D_0r_2),$$ and $\ker \G(0,\cdot) = \ker \G_1(0,\cdot) \times \ker \G_2(0,\cdot).$ Thus $h$ is a filling of $\mathbold{\t} \circ r$, as claimed.

We now verify the properties of $\t$ required in Definition~\ref{dfn:localscFredgerm}(c) of a local sc-Fredholm germ. We in addition show that $\t$ satisfies the stronger properties of a tame sc-Fredholm germ (Definition~\ref{dfn:tamescFredgerm}).

Since each $\t_i$ is a tame sc-Fredholm germ, the corresponding properties hold: There exists a sc$^+$-section 
$$\mathfrak{s}_i : U_i \rightarrow U_i \lhd \mathbb{F}_i$$
satisfying $\mathfrak{s}_i(0) = h_i(0)$, a sc-Banach space $\mathbb{W}_i$, a sc-germ of neighborhoods
\begin{align*}
U'_i \text{ around } 0 \text{ in } [0,\infty)^{s_i} \times \mathbb{R}^{k_i - s_i} \times \mathbb{W}_i
\end{align*}
for some $k_i \geq s_i \geq 0$, and a strong bundle isomorphism
$$\Psi_i : U_i \lhd \mathbb{F}_i \rightarrow U'_i \lhd \mathbb{R}^{k'_i} \times \mathbb{W}_i$$
covering a linear sc-isomorphism
$$\psi_i = id_{[0,\infty)^{s_i}} \times \overline{\psi}_i : U_i \rightarrow U'_i$$
satisfying $\psi_i(0) = 0$, where
$$\overline{\psi}_i : \mathbb{E}'_i \rightarrow \mathbb{R}^{k_i - s_i} \times \mathbb{W}_i$$
is a linear sc-isomorphism, and such that the principal part of the section
\begin{equation} \label{eq:productbasicgermfactorproof}
b_i := \Psi_i \circ (h_i - \mathfrak{s}_i) \circ \psi_i^{-1} : U'_i \rightarrow U'_i \lhd \mathbb{R}^{k'_i} \times \mathbb{W}_i
\end{equation}
is a basic germ. This basic germ property means that the principal part
$$\mathbold{b}_i : U'_i \rightarrow \mathbb{R}^{k'_i} \times \mathbb{W}_i$$
is a sc-smooth germ satisfying $\mathbold{b}_i(0) = 0$ and such that, for every
$$a_i \in [0,\infty)^{s_i}, d_i \in \mathbb{R}^{k_i - s_i}, w_i \in \mathbb{W}_i,$$
we have
$$P_i \circ \mathbold{b}_i(a_i,d_i,w_i) = w_i - B_i(a_i,d_i,w_i)$$
where $P_i : \mathbb{R}^{k'_i} \times \mathbb{W}_i \rightarrow \mathbb{W}_i$ is projection onto $\mathbb{W}_i$ and $B_i$ is a sc-smooth germ satisfying $B_i(0) = 0$ and the contraction property \eqref{eq:basicgerm}.

The section
$$\mathfrak{s} := (f \times id_{\mathbb{F}_1 \times \mathbb{F}_2}) \circ l \circ (\mathfrak{s}_1\times \mathfrak{s}_2) \circ f^{-1} : U \rightarrow U \lhd (\mathbb{F}_1 \times \mathbb{F}_2)$$
is sc$^+$ because each $\mathfrak{s}_i$ is sc$^+$ and the other maps in the composition are linear sc-isomorphisms. Denote the linear sc-isomorphisms given by reordering the factors by
\begin{align*}
q &: ([0,\infty)^{s_1} \times \mathbb{R}^{k_1-s_1} \times \mathbb{W}_1) \times ([0,\infty)^{s_2} \times \mathbb{R}^{k_2-s_2} \times \mathbb{W}_2)\\
&\,\, \rightarrow  [0,\infty)^{s_1 + s_2}  \times \mathbb{R}^{k_1-s_1 + k_2 - s_2} \times \mathbb{W}_1 \times \mathbb{W}_2,\\
Q &: (\mathbb{R}^{k'_1} \times \mathbb{W}_1) \times (\mathbb{R}^{k'_2} \times \mathbb{W}_2) \rightarrow \mathbb{R}^{k'_1 + k'_2} \times \mathbb{W}_1 \times \mathbb{W}_2,\\
l' &: (U'_1 \lhd \mathbb{R}^{k'_1} \times \mathbb{W}_1) \times (U_2' \lhd \mathbb{R}^{k'_2} \times \mathbb{W}_2)\\
&\,\, \rightarrow (U_1' \times U_2') \lhd (\mathbb{R}^{k'_1} \times \mathbb{W}_1 \times \mathbb{R}^{k'_2} \times \mathbb{W}_2).
\end{align*}
Set
\begin{align*}
U' &:= q(U_1' \times U_2'),\\
\psi &:= q \circ (\psi_1 \times \psi_2) \circ f^{-1} : U \rightarrow U',\\
\Psi &:= (q \lhd Q) \circ l' \circ (\Psi_1 \times \Psi_2) \circ l^{-1} \circ (f \times id_{\mathbb{F}_1 \times \mathbb{F}_2})^{-1}\\
&\,\,\,\,: U \lhd (\mathbb{F}_1\times \mathbb{F}_2) \rightarrow U' \lhd (\mathbb{R}^{k'_1 + k'_2} \times \mathbb{W}_1 \times \mathbb{W}_2).
\end{align*}

We claim that the principal part ${\bf b}$ of the section
\begin{equation} \label{eq:productbasicgermproof}
b := \Psi \circ (h - \mathfrak{s}) \circ \psi^{-1}  : U' \rightarrow U' \lhd (\mathbb{R}^{k'_1 + k'_2} \times \mathbb{W}_1 \times \mathbb{W}_2)
\end{equation}
is a basic germ. Observe that the principal parts satisfy
\begin{align*}
{\bf b} &= Q \circ \pi_{\mathbb{R}^{k'_1} \times \mathbb{W}_1 \times \mathbb{R}^{k'_2} \times \mathbb{W}_2} \circ (\Psi_1 \times \Psi_2)\\
&\,\,\,\,\,\,\,\,\circ ((h_1 \times h_2) - (\mathfrak{s}_1 \times \mathfrak{s}_2)) \circ (\psi_1 \times \psi_2)^{-1} \circ q^{-1}\\
&= Q \circ ({\bf b}_1 \times {\bf b}_2) \circ q^{-1}.
\end{align*}
Let $P : \mathbb{R}^{k'_1 + k'_2} \times \mathbb{W}_1 \times \mathbb{W}_2 \rightarrow  \mathbb{W}_1 \times \mathbb{W}_2$ denote the projection onto $\mathbb{W}_1 \times \mathbb{W}_2$. Notice that $P = (P_1 \times P_2) \circ Q^{-1}.$
Then, for all
$$(a_1,a_2) \in [0,\infty)^{s_1} \times [0,\infty)^{s_2}, (d_1,d_2) \in \mathbb{R}^{k_1 - s_1} \times \mathbb{R}^{k_2 - s_2}, (w_1,w_2) \in \mathbb{W}_1 \times \mathbb{W}_2,$$
we compute
 \begin{align*}
P \circ {\bf b}(a_1,a_2,d_1,d_2,w_1,w_2) &= P \circ Q \circ ({\bf b}_1 \times {\bf b}_2) \circ q^{-1}(a_1,a_2,d_1,d_2,w_1,w_2)\\ &= (P_1 \times P_2) \circ ({\bf b}_1 \times {\bf b}_2)(a_1,d_1,w_1,a_2,d_2,w_2)\\
&= (w_1 - B_1(a_1,d_1,w_1), w_2 - B_2(a_2,d_2,w_2))\\
&= (w_1,w_2) - (B_1(a_1,d_1,w_1), B_2(a_2,d_2,w_2))\\
&= (w_1,w_2) - (B_1 \times B_2) \circ q^{-1}(a_1,a_2,d_1,d_2,w_1,w_2),
\end{align*}
so to prove that ${\bf b}$ is a basic germ, it remains to verify that the map
\begin{align*}
B :=& (B_1 \times B_2) \circ q^{-1}\\
&: [0,\infty)^{s_1 + s_2}  \times \mathbb{R}^{k_1-s_1 + k_2 - s_2} \times \mathbb{W}_1 \times \mathbb{W}_2 \rightarrow \mathbb{W}_1 \times \mathbb{W}_2
\end{align*}
satisfies the contraction property \eqref{eq:basicgerm}. Recall our convention that Banach norms on Cartesian products are chosen to be the sum norm (which is equivalent to choosing any standard choice of norm on a product). The contraction property \eqref{eq:basicgerm} for $B$ then follows from the contraction property for each $B_i$, as we now verify. Let $\epsilon > 0$ and $ m \geq 0$. Then, for $i = 1,2,$ there exists $\delta_i > 0$ such that \eqref{eq:basicgerm} holds for $B_i$ with the same choice of $\epsilon.$ Set $\delta := \min(\delta_1, \delta_2)$. Then, given $||(a_1,a_2,d_1,d_2,w_1,w_2)||_m, ||(a_1,a_2,d_1,d_2,w'_1,w'_2)||_m < \delta$ we have
$$||(a_i,d_i,w_i)||_m, ||(a_i,d_i,w'_i)||_m < \delta_i$$
for $i = 1,2,$ from which we compute, using property \eqref{eq:basicgerm} for the $B_i$,
\begin{align*}
||&B(a_1,a_2,d_1,d_2,w_1,w_2) - B(a_1,a_2,d_1,d_2,w'_1,w'_2)||_m\\ &= ||(B_1(a_1,d_1,w_1), B_2(a_2,d_2,w_2)) - (B_1(a_1,d_1,w'_1), B_2(a_2,d_2,w'_2))||_m\\
&= ||(B_1(a_1,d_1,w_1) - B_1(a_1,d_1,w'_1), B_2(a_2,d_2,w_2) - B_2(a_2,d_2,w'_2)||_m\\
&= ||(B_1(a_1,d_1,w_1) - B_1(a_1,d_1,w'_1)||_m +  ||B_2(a_2,d_2,w_2) - B_2(a_2,d_2,w'_2)||_m\\
&\stackrel{\eqref{eq:basicgerm}}{\leq} \epsilon ( ||w_1 - w_1'||_m + ||w_2 - w_2'||_m )\\
&= \epsilon \cdot ||(w_1 - w_1', w_2 - w_2')||_m\\
&= \epsilon \cdot ||(w_1,w_2) - (w_1',w_2')||_m,
\end{align*}
as required. This completes the proof that $\t$ is a local sc-Fredholm germ.

We claim that, in addition, $\t$ is a tame sc-Fredholm germ. We must show that $\psi$ is in the form required by Definition~\ref{dfn:tamescFredgerm}. Given elements $(a_i,e_i) \in [0,\infty)^{s_i} \times \mathbb{E}'_i$ for $i = 1,2,$ write $\overline{\psi}_i(e_i) = (d_i,w_i) \in \mathbb{R}^{k_i-s_i} \times \mathbb{W}_i$ and compute
\begin{align*}
\psi(a_1,a_2,e_1,e_2) &= q \circ (\psi_1 \times \psi_2)(a_1,e_1,a_2,e_2)\\
&= q \circ (id_{[0,\infty)^{s_1}} \times \overline{\psi}_1 \times id_{[0,\infty)^{s_2}} \times \overline{\psi}_2)(a_1,e_1,a_2,e_2)\\
&= (a_1,a_2,d_1,d_2,w_1,w_2).
\end{align*}
So indeed $\psi = id_{[0,\infty)^{s_1 + s_2}} \times \overline{\psi}$ is of the required form, where $$\overline{\psi} : \mathbb{E}_1 \times \mathbb{E}_2 \rightarrow \mathbb{R}^{k_1 - s_1 + k_2 - s_2} \times \mathbb{W}_1 \times \mathbb{W}_2$$ is the linear sc-isomorphism given by $\overline{\psi}(e_1,e_2) = (d_1,d_2,w_1,w_2).$

We have verified that $\s_1 \times \s_2$ is a tame sc-Fredholm section of $\rho_1 \times \rho_2$. To verify the claimed index formula, note first that by definition of sc-Fredholm index \eqref{eq:fredholmindexdef} of a local sc-Fredholm germ and the forms of the basic germs \eqref{eq:productbasicgermfactorproof} and \eqref{eq:productbasicgermproof} we have $ind(\t_i) = k_i - k'_i$ and $$ind(\t) = (k_1 + k_2) - (k_1' + k_2') = ind(\t_1) + ind(\t_2).$$ Then by definition Definition~\ref{dfn:scfredholmsection}(3) of index, for $$(x_1,x_2) \in (\cB_{(1)})_{\infty} \times (\cB_{(2)})_{\infty}$$ we have $$ind_{(x_1,x_2)}(\s_1 \times \s_2) = ind(\t) = ind(\t_1) + ind(\t_2) = ind_{x_1}(\s_1) + ind_{x_2}(\s_2),$$
as claimed.

The final statement about compactness holds because the zero set $(\s_1 \times \s_2)^{-1}(0) = \s_1^{-1}(0) \times \s_2^{-1}(0) \subset (\cB_{(1)})_{\infty} \times (\cB_{(2)})_{\infty}$ is equipped with the product topology in every level $(\cB_{(1)} \times \cB_{(2)})_m = (\cB_{(1)})_m \times (\cB_{(2)})_m$, and for $i = 1,2,$ the subspace $\s_i^{-1}(0) \subset (\cB_{(i)})_m$ is compact for all $m \geq 0$ (see Remark~\ref{rmk:levelshiftglobal} for a further discussion about compactness in the different levels).
\end{proof}

We now generalize the above Cartesian product construction to the ep-groupoid setting.

\begin{lem} \label{lem:ep-CartesianFredholm} \text{}
\begin{enumilistfiber}
\item Consider tame ep-groupoids $\cX_{(i)} = (X_{(i)}, {\bf X}_{(i)})$ for $i = 1,2$. Then, the Cartesian product $$\cX_{(1)} \times \cX_{(2)} = (X_{(1)} \times X_{(2)}, {\bf X}_{(1)} \times {\bf X}_{(2)})$$ is a tame ep-groupoid with degeneracy index satisfying $$d_{X_{(1)} \times X_{(2)}}(x_1,x_2) = d_{X_{(1)}}(x_1) + d_{X_{(2)}}(x_2)$$ for all $(x_1,x_2) \in (X_{(1)})_{\infty} \times (X_{(2)})_{\infty}.$\\

\item Consider tame strong bundles $(P_i : E_{(i)} \rightarrow X_{(i)}, \mu_i)$ over $\cX_{(i)}$ for $i = 1,2.$ Denote the reordering of factors map by
$$l : ({\bf X}_{(1)} \times {\bf X}_{(2)}) \tensor[_{s_1 \times s_2}]{\times}{_{P_1 \times P_2}} (E_{(1)} \times E_{(2)}) \rightarrow ({{\bf X}_{(1)}} \tensor[_{s_1}]{\times}{_{P_1}} E_{(1)}) \times ({{\bf X}_{(2)}} \tensor[_{s_2}]{\times}{_{P_2}} E_{(2)})$$
and set
$$\mu := (\mu_1 \times \mu_2) \circ l : ({\bf X}_{(1)} \times {\bf X}_{(2)}) \tensor[_{s_1 \times s_2}]{\times}{_{P_1 \times P_2}} (E_{(1)} \times E_{(2)}) \rightarrow E_{(1)} \times E_{(2)}.$$
Then, the tuple $$(P_1 \times P_2 : E_{(1)} \times E_{(2)} \rightarrow X_{(1)} \times X_{(2)}, \mu)$$ is a tame strong bundle over $\cX_{(1)} \times \cX_{(2)}$.\\

\item Consider tame sc-Fredholm section functors $\s_i : X_{(i)} \rightarrow E_{(i)}$ (Definition~\ref{dfn:tamescfredholmepgroupoid}) of $(P_i, \mu_i)$ for $i = 1,2.$ Then, the product map $$\s_1 \times \s_2 : X_{(1)} \times X_{(2)} \rightarrow E_{(1)} \times E_{(2)}$$ is a tame sc-Fredholm section functor of $(P_1 \times P_2, \mu)$ with index satisfying $$ind_{(x_1,x_2)}(\s_1 \times \s_2) = ind_{x_1}(\s_1) + ind_{x_2}(\s_2)$$ for all $(x_1,x_2) \in (X_{(1)})_{\infty} \times (X_{(2)})_{\infty}$. If $|\s_i^{-1}(0)|$ is compact for $i = 1,2,$ then $|(\s_1 \times \s_2)^{-1}(0)| = |\s_1^{-1}(0)| \times |\s_2^{-1}(0)|$ is compact.
\end{enumilistfiber}
\end{lem}
\begin{proof}
We prove the statements in (I). The product $\cX_{(1)} \times \cX_{(2)}$ is a groupoid with structure maps (source, target, multiplication, unit, and inverse) as we describe below. First note that Lemma~\ref{lem:M-CartesianPolyfold}(I) provides tame $M$-polyfold structures on the object space $X_{(1)} \times X_{(2)}$ and on the morphism space ${\bf X}_{(1)} \times {\bf X}_{(2)}$ with the claimed degeneracy index. So, to prove (I), it remains to describe the structure maps on $\cX_{(1)} \times \cX_{(2)}$, verify that they are sc-smooth, and verify the \'etale property (Definition~\ref{dfn:tameepgroupoid}(i)) and properness (Definition~\ref{dfn:tameepgroupoid}(iii)).

For $i = 1,2$, let $(s_i,t_i,m_i,u_i,\i_i)$ denote the structure maps on $\cX_{(i)}$. The source $s_1 \times s_2 : {\bf X}_{(1)} \times {\bf X}_{(2)} \rightarrow X_{(1)} \times X_{(2)}$, target $t_1 \times t_2$, unit $u_1 \times u_2$, and inversion $\i_1 \times \i_2$ maps on the product are products of those in each factor, so sc-smoothness follows from sc-smoothness in each factor. Moreover, the \'etale property (Definition~\ref{dfn:tameepgroupoid}(i)) holds because products of surjective local sc-diffeomorphisms are surjective local sc-diffeomorphisms.

To see that the multiplication map on the product is sc-smooth, first note that the reordering of factors map $$({\bf X}_{(1)} \times {\bf X}_{(2)}) \times ({\bf X}_{(1)} \times {\bf X}_{(2)}) \rightarrow ({\bf X}_{(1)} \times {\bf X}_{(1)}) \times ({\bf X}_{(2)} \times {\bf X}_{(2)})$$
is sc-smooth. It restricts to a bijection
$$q : ({\bf X}_{(1)} \times {\bf X}_{(2)}) \tensor[_{s_1 \times s_2}]{\times}{_{t_1 \times t_2}} ({\bf X}_{(1)} \times {\bf X}_{(2)}) \rightarrow ({\bf X}_{(1)} \tensor[_{s_1}]{\times}{_{t_1}} {\bf X}_{(1)}) \times ({\bf X}_{(2)} \tensor[_{s_2}]{\times}{_{t_2}} {\bf X}_{(2)}),$$
which is sc-smooth by \cite[Prop.~2.15]{HWZbook} and \cite[Prop.~2.6(1)]{HWZbook}. The multiplication map on $\cX_{(1)} \times \cX_{(2)}$ is the composition
$$(m_1 \times m_2) \circ q : ({\bf X}_{(1)} \times {\bf X}_{(2)}) \tensor[_{s_1 \times s_2}]{\times}{_{t_1 \times t_2}} ({\bf X}_{(1)} \times {\bf X}_{(2)})  \rightarrow {\bf X}_{(1)} \times {\bf X}_{(2)},$$
hence is sc-smooth.

We now check properness Definition~\ref{dfn:tameepgroupoid}(iii). Let $$(x_1, x_2) \in X_{(1)} \times X_{(2)}$$ and let $V(x_i) \subset X_{(i)}$ be open neighborhoods of $x_i$ such that $$t_i : s_i^{-1}(\overline{V(x_i)}) \rightarrow X_{(i)}$$ are proper maps. Then $V(x_1) \times V(x_2)$ is an open neighborhood of $(x_1,x_2)$ in $X_{(1)} \times X_{(2)}$ and we have
$$(s_1 \times s_2)^{-1}(\overline{V(x_1) \times V(x_2)}) = s_1^{-1}(\overline{V(x_1)}) \times s_2^{-1}(\overline{V(x_2)}).$$
Hence $t_1 \times t_2 : (s_1 \times s_2)^{-1}(\overline{V(x_1) \times V(x_2)}) \rightarrow X_{(1)} \times X_{(2)}$ is proper, as the product of proper maps. This completes the proof that $\cX_{(1)} \times \cX_{(2)}$ is a tame ep-groupoid, and so the statements in (I) are proved.

We now prove the statements in (II). For $i = 1,2,$ the map $$P_i : E_{(i)} \rightarrow X_{(i)}$$ is a tame strong bundle over the $M$-polyfold $X_{(i)},$ so Lemma~\ref{lem:M-CartesianPolyfold}(II) provides a tame strong bundle structure on the product map $$P_1 \times P_2 : E_{(1)} \times E_{(2)} \rightarrow X_{(1)} \times X_{(2)}.$$ The map $\mu$ is a strong bundle map as the composition of the strong bundle maps $\mu_1 \times \mu_2$ and $l$, and the required properties Definition~\ref{dfn:tamestrongbundleepgroupoid}(i)(ii) of $\mu$ follow immediately from those of $\mu_i$ for $i = 1,2$. Hence $(P_1 \times P_2, \mu)$ is a tame strong bundle over $\cX_{(1)} \times \cX_{(2)}$, as claimed.

We now prove the statements in (III). For $i = 1,2,$ the tame sc-Fredholm section functor $\s_i : X_{(i)} \rightarrow E_{(i)}$ is in particular a tame sc-Fredholm section of the bundle $P_i$, in the $M$-polyfold sense. So Lemma~\ref{lem:M-CartesianPolyfold}(III) provides the product map $\s_1 \times \s_2$ with the structure of a tame sc-Fredholm section of the bundle $P_1 \times P_2$ with the claimed sc-Fredholm index. Moreover, $\s_1 \times \s_2$ satisfies the required property \eqref{eq:sectionfunctorcompatibility} of a section functor of $(P_1 \times P_2, \mu)$ by the corresponding property of the sections functors $\s_i$ of $(P_i, \mu_i)$. Indeed, for all morphisms $(\phi_1 \times \phi_2) \in {\bf X}_{(1)} \times {\bf X}_{(2)}$, we compute
\begin{align*}
(\s_1 \times \s_2) \circ (t_1 \times t_2)(\phi_1, \phi_2) &= (\s_1(t_1(\phi_1)), \s_2(t_2(\phi_2)))\\
&= (\mu_1(\phi_1,\s_1(s_1(\phi_1))), \mu_2(\phi_2,\s_2(s_2(\phi_2))))\\
&= (\mu_1 \times \mu_2) \circ l(\phi_1,\phi_2,\s_1(s_1(\phi_1)),\s_2(s_2(\phi_2)))\\
&= \mu(\phi_1,\phi_2,(\s_1 \times \s_2) \circ (s_1 \times s_2)(\phi_1,\phi_2)).
\end{align*}
This completes the proof that $\s_1 \times \s_2$ is a tame sc-Fredholm section functor of $(P_1 \times P_2, \mu)$.

The final statement about compactness holds because the orbit space of $\cX_{(1)} \times \cX_{(2)}$ is equal to the Cartesian product of the orbit spaces of the $\cX_{(i)}$ equipped with the product topology.
\end{proof}

We proceed to construct fiber products of tame sc-Fredholm section functors over ep-groupoids. As usual, the result specializes to the case of $M$-polyfolds by considering an $M$-polyfold $\cB$ as an ep-groupoid with the trivial groupoid structure: the object space is $\cB$ and the morphism space consists of the identity morphisms $\{ 1_x \,\, | \,\, x \in \cB \}$ which is an $M$-polyfold by declaring the bijection $x \mapsto 1_x$ with $\cB$ to be a sc-diffeomorphism.

\begin{cor} \label{cor:ep-fiberFredholm}\text{}
\begin{enumilistfiber}
\item Consider tame ep-groupoids $\cX_{(i)} = (X_{(i)}, {\bf X}_{(i)})$ for $i = 1,2,$ a smooth manifold $Y$, and sc-smooth maps $f_i : X_{(i)} \rightarrow Y$ for $i = 1,2,$ that satisfy the compatibility with morphisms
$$f_i(s_i(\phi)) = f_i(t_i(\phi)) \text{ for all } \phi \in {\bf X}_{(i)}.$$
Assume that the product map $f_1 \times f_2 : X_{(1)} \times X_{(2)} \rightarrow Y \times Y$ is transverse (Definition~\ref{dfn:transverse}) to the diagonal $$\Delta = \{ (y,y) \,\, | \,\, y \in Y \} \subset Y \times Y,$$ and denote the fiber product of object spaces by
$$X_{(1)}\tensor[_{f_1}]{\times}{_{f_2}} X_{(2)} := \{ (x_1,x_2) \in X_{(1)} \times X_{(2)} \,\, | \,\, f_1(x_1) = f_2(x_2) \} \subset X_{(1)} \times X_{(2)}.$$

Then, there exists an open neighborhood
$$\tilde{X} \subset ({X_{(1)}} \tensor[_{f_1}]{\times}{_{f_2}} {X_{(2)}}) \cap ((X_{(1)})_{1} \times (X_{(2)})_{1})$$
of $({X_{(1)}} \tensor[_{f_1}]{\times}{_{f_2}} {X_{(2)}}) \cap ((X_{(1)})_{\infty} \times (X_{(2)})_{\infty})$ such that $\tilde{X} \subset X_{(1)}^1 \times X_{(2)}^1$
is a slice in the $M$-polyfold sense (Definition~\ref{dfn:globalslice}) satisfying $$codim_x(\tilde{X} \subset X_{(1)}^1 \times X_{(2)}^1) = \dim Y$$ for every $x \in \tilde{X}_1 = \tilde{X} \cap ((X_{(1)})_2 \times (X_{(2)})_2)$. In particular, the full subcategory $$\tilde{\cX} := (\tilde{X}, \tilde{\bf X})$$ of $\cX_{(1)}^1 \times \cX_{(2)}^1$ with object space $\tilde{X}$ is a tame ep-groupoid  with degeneracy index satisfying $$d_{\tilde{X}}(x_1,x_2) = d_{X_{(1)}}(x_1) + d_{X_{(2)}}(x_2)$$ for all $(x_1,x_2) \in \tilde{X}_{\infty}.$\\

\item Consider, in addition, tame strong bundles $(P_i : E_{(i)} \rightarrow X_{(i)}, \mu_i)$ over $\cX_{(i)}$ for $i = 1,2.$ Then, there exists a possibly smaller neighborhood $\tilde{X}$ in (I) that is a slice of the bundle $$P_1|_{E_{(1)}^1} \times P_2|_{E_{(2)}^1} : E_{(1)}^1 \times E_{(2)}^1 \rightarrow X_{(1)}^1 \times X_{(2)}^1$$ in the $M$-polyfold sense (Definition~\ref{dfn:globalslice}). In particular, the tuple $(\tilde{P}, \tilde{\mu})$ consisting of the restrictions
\begin{align*}
\tilde{E} &:= (P_1|_{E_{(1)}^1} \times P_2|_{E_{(2)}^1})^{-1}(\tilde{X})\\
\tilde{P} &:= (P_1 \times P_2)|_{\tilde{E}} : \tilde{E} \rightarrow \tilde{X},\\
\tilde{\mu} &:= (\mu_1 \times \mu_2) \circ l |_{\tilde{{\bf X}} \tensor[_{s_1 \times s_2}]{\times}{_{\tilde{P}}} \tilde{E}} : \tilde{{\bf X}} \tensor[_{s_1 \times s_2}]{\times}{_{\tilde{P}}} \tilde{E} \rightarrow \tilde{E},
\end{align*}
is a tame strong bundle over $\tilde{\cX}$, where $$l : ({\bf X}_{(1)} \times {\bf X}_{(2)}) \times (E_{(1)} \times E_{(2)}) \rightarrow ({{\bf X}_{(1)}} \times E_{(1)}) \times ({{\bf X}_{(2)}} \times E_{(2)})$$ is the reordering of factors map.\\

\item Consider, in addition, tame sc-Fredholm section functors (Definition~\ref{dfn:tamescfredholmepgroupoid}) $\s_i : X_{(i)} \rightarrow E_{(i)}$ of $(P_i, \mu_i)$ for $i = 1,2.$ Assume that $f_1 \times f_2$ is $(\s_1 \times \s_2)$-compatibly transverse to $\Delta$ (Definition~\ref{dfn:transverse}).

Then, there exists a possibly smaller neighborhood $\tilde{X}$ in (II) that is a slice of the tame sc-Fredholm section $$(\s_1 \times \s_2)|_{X_{(1)}^1 \times X_{(2)}^1} : X_{(1)}^1 \times X_{(2)}^1 \rightarrow E_{(1)}^1 \times E_{(2)}^1$$ in the $M$-polyfold sense (Definition~\ref{dfn:globalslice}). In particular, the restriction
$$\tilde{\s} := (\s_1 \times \s_2)|_{\tilde{X}} : \tilde{X} \rightarrow \tilde{E}$$
is a tame sc-Fredholm section functor of the bundle $(\tilde{P}, \tilde{\mu})$ with index satisfying 
$$ind_{(x_1,x_2)}(\tilde{\s}) = ind_{x_1}(\s_1) + ind_{x_2}(\s_2) - \dim Y$$
for all $(x_1,x_2) \in \tilde{X}_{\infty}$. If $|\s_i^{-1}(0)|$ is compact for $i = 1,2,$ then $|\tilde{\s}^{-1}(0)|$ is compact.
\end{enumilistfiber}
\end{cor}
\begin{proof}
We prove the statements in (I). Lemma~\ref{lem:ep-CartesianFredholm}(I) provides a tame ep-groupoid structure on ${\cX}_{(1)} \times \cX_{(2)}$ with degeneracy index satisfying $d_{X_{(1)} \times X_{(2)}}(x_1,x_2) = d_{X_{(1)}}(x_1) + d_{X_{(2)}}(x_2)$ for $(x_1,x_2) \in (X_{(1)})_{\infty} \times (X_{(2)})_{\infty}$. We claim that Corollary~\ref{cor:globalimplicitepgroupoid}(I) applies to the product map $$f_1 \times f_2 : X_{(1)} \times X_{(2)} \rightarrow Y \times Y$$ and the codimension-$\dim Y$ submanifold $\Delta \subset Y \times Y$. Indeed, $f_1 \times f_2$ is transverse to $\Delta$ by hypothesis and the required morphism compatibility $(f_1 \times f_2) \circ (s_1 \times s_2)(\phi_1, \phi_2) = (f_1 \times f_2) \circ (t_1 \times t_2)(\phi_1, \phi_2)$ holds by the hypothesis $f_i(s_i(\phi_i)) = f_i(t_i(\phi_i))$. Since the fiber product is the preimage of the diagonal
$$X_{(1)}\tensor[_{f_1}]{\times}{_{f_2}} X_{(2)} = (f_1 \times f_2)^{-1}(\Delta),$$
the result of Corollary~\ref{cor:globalimplicitepgroupoid}(I) is exactly the assertions in (I).

Similarly, to prove (II), we note that Lemma~\ref{lem:ep-CartesianFredholm}(II) provides a tame strong bundle structure on $(P_1 \times P_2, (\mu_1 \times \mu_2) \circ l)$, and then Corollary~\ref{cor:globalimplicitepgroupoid}(II) provides the desired result.

We now prove (III). Lemma~\ref{lem:ep-CartesianFredholm}(III) shows that $$\s_1 \times \s_2 : X_{(1)} \times X_{(2)} \rightarrow E_{(1)} \times E_{(2)}$$ is a tame sc-Fredholm section functor of $(P_1 \times P_2, (\mu_1 \times \mu_2) \circ l)$ with index satisfying $ind_{(x_1,x_2)}(\s_1 \times \s_2) = ind_{x_1}(\s_1) + ind_{x_2}(\s_2)$ and with $|\s_1^{-1}(0)| \times |\s_2^{-1}(0)|$ compact. Then by Corollary~\ref{cor:globalimplicitepgroupoid}(III), we conclude that $\tilde{\s}$ is a tame sc-Fredholm section functor with index satisfying $ind_{(x_1,x_2)}(\tilde{\s})  = ind_{x_1}(\s_1) + ind_{x_2}(\s_2) - \dim Y$ and such that $|\tilde{\s}^{-1}(0)|$ is compact.
\end{proof}

\bibliographystyle{amsplain}
\bibliography{../../../references}

\providecommand{\bysame}{\leavevmode\hbox to3em{\hrulefill}\thinspace}
\providecommand{\MR}{\relax\ifhmode\unskip\space\fi MR }
\providecommand{\MRhref}[2]{%
  \href{http://www.ams.org/mathscinet-getitem?mr=#1}{#2}
}
\providecommand{\href}[2]{#2}
\begin{thebibliography}{10}

\bibitem{MR1826267}
Y.~Eliashberg, A.~Givental, and H.~Hofer, \emph{Introduction to {S}ymplectic
  {F}ield {T}heory}, Geom. Funct. Anal. (2000), no.~Special Volume, Part II,
  560--673.

\bibitem{MR3576532}
O.~Fabert, J.~W. Fish, R.~Golovko, and K.~Wehrheim, \emph{Polyfolds: {A} first
  and second look}, EMS Surv. Math. Sci. \textbf{3} (2016), no.~2, 131--208.

\bibitem{AFFW}
B.~Filippenko and K.~Wehrheim, \emph{A polyfold proof of the {A}rnold
  conjecture}, arXiv:1810.06180, 2018.

\bibitem{FZW}
B.~Filippenko, Z.~Zhou, and K.~Wehrheim, \emph{Counterexamples in {S}cale
  {C}alculus}, Proceedings of the National Academy of Sciences \textbf{116}
  (2019), no.~18, 8787--8797.

\bibitem{HWZsymplecticfieldtheory}
J.~W. Fish and H.~Hofer, \emph{{A}pplications of {P}olyfold {T}heory {I}{I}:
  {S}ymplectic {F}ield {T}heory}, in preparation.

\bibitem{PolyfoldConstructions}
\bysame, \emph{Polyfold {C}onstructions: {T}ools, {T}echniques, and
  {F}unctors}, in preparation.

\bibitem{FHLecturesOnSFT}
\bysame, \emph{Lectures on {P}olyfolds and {S}ymplectic {F}ield {T}heory},
  arXiv:1808.07147, 2018.

\bibitem{SFTPrimerI}
\bysame, \emph{{P}olyfold and {S}{F}{T} {N}otes {I}: {A} {P}rimer on
  {P}olyfolds and {C}onstruction {T}ools}, arXiv:1806.07025, 2018.

\bibitem{SFTPrimerII}
\bysame, \emph{{P}olyfold and {S}{F}{T} {N}otes {I}{I}: {L}ocal-{L}ocal
  {M}-{P}olyfold {C}onstructions}, arXiv:1808.04939, 2018.

\bibitem{MR2341834}
H.~Hofer, K.~Wysocki, and E.~Zehnder, \emph{A {G}eneral {F}redholm {T}heory
  {I}: {A} splicing-based differential geometry}, J. Eur. Math. Soc. (JEMS)
  \textbf{9} (2007), no.~4, 841--876.

\bibitem{MR2507223}
\bysame, \emph{A {G}eneral {F}redholm {T}heory {II}: {I}mplicit function
  theorems}, Geom. Funct. Anal. \textbf{19} (2009), no.~1, 206--293.

\bibitem{MR2515707}
\bysame, \emph{A {G}eneral {F}redholm {T}heory {III}: {F}redholm functors and
  polyfolds}, Geom. Topol. \textbf{13} (2009), no.~4, 2279--2387.

\bibitem{MR2558891}
\bysame, \emph{Integration theory on the zero sets of polyfold {F}redholm
  sections}, Math. Ann. \textbf{346} (2010), no.~1, 139--198.

\bibitem{MR2644764}
\bysame, \emph{{S}c-smoothness, retractions and new models for smooth spaces},
  Discrete Contin. Dyn. Syst. \textbf{28} (2010), no.~2, 665--788.

\bibitem{MR3683060}
\bysame, \emph{Applications of {P}olyfold {T}heory {I}: {T}he polyfolds of
  {G}romov-{W}itten theory}, Mem. Amer. Math. Soc. \textbf{248} (2017),
  no.~1179, v+218.

\bibitem{HWZbook}
\bysame, \emph{Polyfold and {F}redholm {T}heory}, arXiv:1707.08941v1, 2017.

\bibitem{jiayong}
J.~Li and K.~Wehrheim, \emph{${A}_\infty$-structures from {M}orse trees with
  pseudoholomorphic disks}, preprint available at
  \href{http://math.berkeley.edu/~katrin/papers/disktrees.pdf}{math.berkeley.edu/~katrin/papers/disktrees.pdf}.

\bibitem{MR1173211}
J.~Margalef-Roig and E.~Outerelo~Dominguez, \emph{Differential topology},
  North-Holland Mathematics Studies, vol. 173, North-Holland Publishing Co.,
  Amsterdam, 1992.

\bibitem{MR2954391}
D.~McDuff and D.~Salamon, \emph{{$J$}-holomorphic curves and symplectic
  topology}, second ed., American Mathematical Society Colloquium Publications,
  vol.~52, American Mathematical Society, Providence, RI, 2012.

\bibitem{MR1432464}
S.~Piunikhin, D.~Salamon, and M.~Schwarz, \emph{Symplectic {F}loer-{D}onaldson
  theory and quantum cohomology}, Contact and symplectic geometry ({C}ambridge,
  1994), Publ. Newton Inst., vol.~8, Cambridge Univ. Press, Cambridge, 1996,
  pp.~171--200.

\bibitem{WolfPullbackPerturbation}
W.~Schmaltz, \emph{Naturality of polyfold invariants and pulling back abstract
  perturbations}, arXiv:1912.13370, 2019.

\bibitem{wehrheimfredholm}
K.~Wehrheim, \emph{Fredholm notions in scale calculus and {H}amiltonian {F}loer
  theory}, arXiv:1209.4040, 2012.

\bibitem{MR3084244}
\bysame, \emph{{S}mooth structures on {M}orse trajectory spaces, featuring
  finite ends and associative gluing}, Proceedings of the {F}reedman {F}est,
  Geom. Topol. Monogr., vol.~18, Geom. Topol. Publ., Coventry, 2012,
  pp.~369--450.

\bibitem{ZQuotient}
Z.~Zhou, \emph{Quotient theorems in polyfold theory and ${S}^1$-equivariant
  transversality}, to appear in Proceedings of the London Mathematical Society.

\end{thebibliography}
\end{document}